\documentclass{amsart} %add [draft] to dedect overfull boxes

\usepackage{amssymb,amsmath,stmaryrd,mathrsfs}
\def\definetac{\newif\iftac}    % Can't define a \newif inside another \if!
\ifx\tactrue\undefined
  \definetac
  \ifx\state\undefined\tacfalse\else\tactrue\fi
\fi
\iftac\else\usepackage{amsthm}\fi
\usepackage[all,2cell]{xy}
\UseAllTwocells
\usepackage{enumitem}
\usepackage{xcolor}
\definecolor{darkgreen}{rgb}{0,0.45,0} 
\usepackage[pagebackref,colorlinks,citecolor=darkgreen,linkcolor=darkgreen]{hyperref}
\usepackage{mathtools}          % for all sorts of things
\usepackage{tikz}
\usepackage{braket}             % for \Set, etc.

\usepackage{url}                % for citations to web sites
\usepackage{xspace}             % put spaces after a \command in text

\makeatletter
\let\ea\expandafter

\def\mdef#1#2{\ea\ea\ea\gdef\ea\ea\noexpand#1\ea{\ea\ensuremath\ea{#2}\xspace}}
\def\alwaysmath#1{\ea\ea\ea\global\ea\ea\ea\let\ea\ea\csname your@#1\endcsname\csname #1\endcsname
  \ea\def\csname #1\endcsname{\ensuremath{\csname your@#1\endcsname}\xspace}}

\DeclareRobustCommand\widecheck[1]{{\mathpalette\@widecheck{#1}}}
\def\@widecheck#1#2{%
    \setbox\z@\hbox{\m@th$#1#2$}%
    \setbox\tw@\hbox{\m@th$#1%
       \widehat{%
          \vrule\@width\z@\@height\ht\z@
          \vrule\@height\z@\@width\wd\z@}$}%
    \dp\tw@-\ht\z@
    \@tempdima\ht\z@ \advance\@tempdima2\ht\tw@ \divide\@tempdima\thr@@
    \setbox\tw@\hbox{%
       \raise\@tempdima\hbox{\scalebox{1}[-1]{\lower\@tempdima\box
\tw@}}}%
    {\ooalign{\box\tw@ \cr \box\z@}}}

\newcount\foreachcount

\def\foreachletter#1#2#3{\foreachcount=#1
  \ea\loop\ea\ea\ea#3\@alph\foreachcount
  \advance\foreachcount by 1
  \ifnum\foreachcount<#2\repeat}

\def\foreachLetter#1#2#3{\foreachcount=#1
  \ea\loop\ea\ea\ea#3\@Alph\foreachcount
  \advance\foreachcount by 1
  \ifnum\foreachcount<#2\repeat}

\def\definescr#1{\ea\gdef\csname s#1\endcsname{\ensuremath{\mathscr{#1}}\xspace}}
\foreachLetter{1}{27}{\definescr}
\def\definecal#1{\ea\gdef\csname c#1\endcsname{\ensuremath{\mathcal{#1}}\xspace}}
\foreachLetter{1}{27}{\definecal}
\def\definebold#1{\ea\gdef\csname b#1\endcsname{\ensuremath{\mathbf{#1}}\xspace}}
\foreachLetter{1}{27}{\definebold}
\def\definebb#1{\ea\gdef\csname l#1\endcsname{\ensuremath{\mathbb{#1}}\xspace}}
\foreachLetter{1}{27}{\definebb}
\def\definefrak#1{\ea\gdef\csname f#1\endcsname{\ensuremath{\mathfrak{#1}}\xspace}}
\foreachletter{1}{9}{\definefrak} % \fi is something else!
\foreachletter{10}{27}{\definefrak}
\def\definebar#1{\ea\gdef\csname #1bar\endcsname{\ensuremath{\overline{#1}}\xspace}}
\foreachLetter{1}{27}{\definebar}
\foreachletter{1}{8}{\definebar} % \hbar is something else!
\foreachletter{9}{15}{\definebar} % \obar is something else!
\foreachletter{16}{27}{\definebar}
\def\definetil#1{\ea\gdef\csname #1til\endcsname{\ensuremath{\widetilde{#1}}\xspace}}
\foreachLetter{1}{27}{\definetil}
\foreachletter{1}{27}{\definetil}
\def\definehat#1{\ea\gdef\csname #1hat\endcsname{\ensuremath{\widehat{#1}}\xspace}}
\foreachLetter{1}{27}{\definehat}
\foreachletter{1}{27}{\definehat}
\def\definechk#1{\ea\gdef\csname #1chk\endcsname{\ensuremath{\widecheck{#1}}\xspace}}
\foreachLetter{1}{27}{\definechk}
\foreachletter{1}{27}{\definechk}
\def\defineul#1{\ea\gdef\csname u#1\endcsname{\ensuremath{\underline{#1}}\xspace}}
\foreachLetter{1}{27}{\defineul}
\foreachletter{1}{27}{\defineul}

\def\autofmt@n#1\autofmt@end{\mathrm{#1}}
\def\autofmt@b#1\autofmt@end{\mathbf{#1}}
\def\autofmt@l#1#2\autofmt@end{\mathbb{#1}\mathsf{#2}}
\def\autofmt@c#1#2\autofmt@end{\mathcal{#1}\mathit{#2}}
\def\autofmt@s#1#2\autofmt@end{\mathscr{#1}\mathit{#2}}
\def\autofmt@f#1\autofmt@end{\mathsf{#1}}
\def\autofmt@u#1\autofmt@end{\underline{\smash{\mathsf{#1}}}}
\def\autofmt@U#1\autofmt@end{\underline{\underline{\smash{\mathsf{#1}}}}}
\def\autofmt@h#1\autofmt@end{\widehat{#1}}
\def\autofmt@r#1\autofmt@end{\overline{#1}}
\def\autofmt@t#1\autofmt@end{\widetilde{#1}}
\def\autofmt@k#1\autofmt@end{\check{#1}}

\def\auto@drop#1{}
\def\autodef#1{\ea\ea\ea\@autodef\ea\ea\ea#1\ea\auto@drop\string#1\autodef@end}
\def\@autodef#1#2#3\autodef@end{%
  \ea\def\ea#1\ea{\ea\ensuremath\ea{\csname autofmt@#2\endcsname#3\autofmt@end}\xspace}}
\def\autodefs@end{blarg!}
\def\autodefs#1{\@autodefs#1\autodefs@end}
\def\@autodefs#1{\ifx#1\autodefs@end%
  \def\autodefs@next{}%
  \else%
  \def\autodefs@next{\autodef#1\@autodefs}%
  \fi\autodefs@next}

\DeclareSymbolFont{bbold}{U}{bbold}{m}{n}
\DeclareSymbolFontAlphabet{\mathbbb}{bbold}

\mdef\delbar{\overline{\partial}}

\mdef\hf{\textstyle\frac12 }
\mdef\thrd{\textstyle\frac13 }
\mdef\qtr{\textstyle\frac14 }

\newcommand{\op}{^{\mathrm{op}}}

\newcommand{\coop}{^{\mathrm{coop}}}

\SelectTips{cm}{}
\newdir{ >}{{}*!/-10pt/@{>}}    % extra spacing for tail arrows in XYpic
\newcommand{\pushoutcorner}[1][dr]{\save*!/#1+1.2pc/#1:(1,-1)@^{|-}\restore}
\newcommand{\pullbackcorner}[1][dr]{\save*!/#1-1.2pc/#1:(-1,1)@^{|-}\restore}

\mdef\Id{\mathrm{Id}}
\mdef\id{\mathrm{id}}
\alwaysmath{ell}
\alwaysmath{infty}
\alwaysmath{odot}
\def\frc#1/#2.{\frac{#1}{#2}}   % \frc x^2+1 / x^2-1 .
\mdef\ten{\mathrel{\otimes}}

\mdef\sqten{\mathrel{\boxtimes}}

\DeclareRobustCommand\widecheck[1]{{\mathpalette\@widecheck{#1}}}
\def\@widecheck#1#2{%
    \setbox\z@\hbox{\m@th$#1#2$}%
    \setbox\tw@\hbox{\m@th$#1%
       \widehat{%
          \vrule\@width\z@\@height\ht\z@
          \vrule\@height\z@\@width\wd\z@}$}%
    \dp\tw@-\ht\z@
    \@tempdima\ht\z@ \advance\@tempdima2\ht\tw@ \divide\@tempdima\thr@@
    \setbox\tw@\hbox{%
       \raise\@tempdima\hbox{\scalebox{1}[-1]{\lower\@tempdima\box
\tw@}}}%
    {\ooalign{\box\tw@ \cr \box\z@}}}

\DeclareMathOperator\colim{colim}

\DeclareMathOperator\Ho{Ho}

\DeclareMathOperator\Hom{Hom}

\newcommand{\ot}{\ensuremath{\leftarrow}}

\mdef\we{\overset{\sim}{\longrightarrow}}
\mdef\leftwe{\overset{\sim}{\longleftarrow}}

\let\xto\xrightarrow

\def\rightarrowtailfill@{\arrowfill@{\Yright\joinrel\relbar}\relbar\rightarrow}
\newcommand\xrightarrowtail[2][]{\ext@arrow 0055{\rightarrowtailfill@}{#1}{#2}}

\def\twoheadrightarrowfill@{\arrowfill@{\relbar\joinrel\relbar}\relbar\twoheadrightarrow}
\newcommand\xtwoheadrightarrow[2][]{\ext@arrow 0055{\twoheadrightarrowfill@}{#1}{#2}}

\def\slashedarrowfill@#1#2#3#4#5{%
  $\m@th\thickmuskip0mu\medmuskip\thickmuskip\thinmuskip\thickmuskip
   \relax#5#1\mkern-7mu%
   \cleaders\hbox{$#5\mkern-2mu#2\mkern-2mu$}\hfill
   \mathclap{#3}\mathclap{#2}%
   \cleaders\hbox{$#5\mkern-2mu#2\mkern-2mu$}\hfill
   \mkern-7mu#4$%
}
\def\rightslashedarrowfill@{%
  \slashedarrowfill@\relbar\relbar\mapstochar\rightarrow}
\newcommand\xslashedrightarrow[2][]{%
  \ext@arrow 0055{\rightslashedarrowfill@}{#1}{#2}}
\mdef\hto{\xslashedrightarrow{}}
\mdef\htoo{\xslashedrightarrow{\quad}}

\def\toiso{\xto{\smash{\raisebox{-.5mm}{$\scriptstyle\sim$}}}}

\long\def\my@drawfill#1#2;{%
\@skipfalse
\fill[#1,draw=none] #2;
\@skiptrue
\draw[#1,fill=none] #2;
}
\newif\if@skip
\newcommand{\skipit}[1]{\if@skip\else#1\fi}
\newcommand{\drawfill}[1][]{\my@drawfill{#1}}

\newif\ifhyperref
\@ifpackageloaded{hyperref}{\hyperreftrue}{\hyperreffalse}
\iftac
  \let\your@state\state
  \def\state#1{\gdef\currthmtype{#1}\your@state{#1}}
  \let\your@staterm\staterm
  \def\staterm#1{\gdef\currthmtype{#1}\your@staterm{#1}}
  \let\defthm\newtheorem
  \def\currthmtype{}
  \ifhyperref
    \def\autoref#1{\ref*{label@name@#1}~\ref{#1}}
  \else
    \def\autoref#1{\ref{label@name@#1}~\ref{#1}}
  \fi
  \AtBeginDocument{%
    \let\old@label\label%
    \def\label#1{%
      {\let\your@currentlabel\@currentlabel%
        \edef\@currentlabel{\currthmtype}%
        \old@label{label@name@#1}}%
      \old@label{#1}}
  }
\else
  \ifhyperref
    \def\defthm#1#2{%
      \newtheorem{#1}{#2}[section]%
      \expandafter\def\csname #1autorefname\endcsname{#2}%
      \expandafter\let\csname c@#1\endcsname\c@thm}
  \else
    \def\defthm#1#2{\newtheorem{#1}[thm]{#2}}
    \ifx\SK@label\undefined\let\SK@label\label\fi
    \let\old@label\label
    \let\your@thm\@thm
    \def\@thm#1#2#3{\gdef\currthmtype{#3}\your@thm{#1}{#2}{#3}}
    \def\currthmtype{}
    \def\label#1{{\let\your@currentlabel\@currentlabel\def\@currentlabel%
        {\currthmtype~\your@currentlabel}%
        \SK@label{#1@}}\old@label{#1}}
    \def\autoref#1{\ref{#1@}}
  \fi
\fi

\newtheorem{thm}{Theorem}[section]

\defthm{cor}{Corollary}
\defthm{prop}{Proposition}
\defthm{lem}{Lemma}
\defthm{sch}{Scholium}
\defthm{assume}{Assumption}
\defthm{claim}{Claim}
\defthm{conj}{Conjecture}
\defthm{hyp}{Hypothesis}
\defthm{fact}{Fact}
\iftac\theoremstyle{plain}\else\theoremstyle{definition}\fi
\defthm{defn}{Definition}
\defthm{pseudodefn}{Pseudo-definition}
\defthm{notn}{Notation}
\defthm{discl}{Disclaimer}
\iftac\theoremstyle{plain}\else\theoremstyle{remark}\fi
\defthm{rmk}{Remark}
\defthm{eg}{Example}
\defthm{egs}{Examples}
\defthm{ex}{Exercise}
\defthm{ceg}{Counterexample}
\defthm{warn}{Warning}
\defthm{con}{Construction}

\def\thmqedhere{\expandafter\csname\csname @currenvir\endcsname @qed\endcsname}

\setitemize[1]{leftmargin=2em}
\setenumerate[1]{leftmargin=*}

\iftac
  \let\c@equation\c@subsection
\else
  \let\c@equation\c@thm
\fi
\numberwithin{equation}{section}

\@ifpackageloaded{mathtools}{\mathtoolsset{showonlyrefs,showmanualtags}}{}

\alwaysmath{alpha}
\alwaysmath{beta}
\alwaysmath{gamma}
\alwaysmath{Gamma}
\alwaysmath{delta}
\alwaysmath{Delta}
\alwaysmath{epsilon}
\mdef\ep{\varepsilon}
\alwaysmath{zeta}
\alwaysmath{eta}
\alwaysmath{theta}
\alwaysmath{Theta}
\alwaysmath{iota}
\alwaysmath{kappa}
\alwaysmath{lambda}
\alwaysmath{Lambda}
\alwaysmath{mu}
\alwaysmath{nu}
\alwaysmath{xi}
\alwaysmath{pi}
\alwaysmath{rho}
\alwaysmath{sigma}
\alwaysmath{Sigma}
\alwaysmath{tau}
\alwaysmath{upsilon}
\alwaysmath{Upsilon}
\alwaysmath{phi}
\alwaysmath{Pi}
\alwaysmath{Phi}
\mdef\ph{\varphi}
\alwaysmath{chi}
\alwaysmath{psi}
\alwaysmath{Psi}
\alwaysmath{omega}
\alwaysmath{Omega}

\makeatother

\tikzset{lab/.style={auto,font=\scriptsize}} % arrow labels
\usetikzlibrary{arrows}

\usepackage{xcolor}% provides \colorlet
\usepackage{fixme}
\fxsetup{
    status=draft,
    author=,
    layout=margin,
    theme=color
}

\definecolor{fxnote}{rgb}{1.0000,0.0000,0.0000}
\colorlet{fxnotebg}{yellow}

\newcommand{\tw}{\ensuremath{\operatorname{tw}}}
\newcommand{\D}{\sD}
\newcommand{\E}{\sE}

\newcommand{\V}{\sV}
\newcommand{\W}{\sW}

\autodefs{\cDER\cCat\cGrp\cCAT\cMONCAT\cTriaCAT\cAddCAT\cPreCAT\cPtCAT\cSp\bSet\cProf\bCat
\cPro\cMONDER\cBICAT\ncomp\nid\ncoker\niso\ndia\nIm\cSprof\ncyl\fC\fZ\fT\nhocolim\nholim\cPsNat\ncoll\nCh\bsSet\cAlg\cIm\cI\cP}

\def\ho{\mathscr{H}\!\mathit{o}\xspace}

\let\oldboxtimes\boxtimes
\def\boxtimes{\mathrel{\oldboxtimes}}

\newcommand{\fib}{\mathsf{fib}}
\newcommand{\cof}{\mathsf{cof}}

\newcommand{\cube}[1]{\square^{#1}}

\def\ccsub{_{\mathrm{cc}}}
\def\pdh(#1,#2){\llbracket #1,#2\rrbracket}
\def\ldh(#1,#2){\llbracket #1,#2\rrbracket\ccsub}
\def\pend(#1){\pdh(#1,#1)}
\def\lend(#1){\ldh(#1,#1)}

\def\DTl#1#2#3#4#5#6#7{%
  \xymatrix@C=3pc{{#1} \ar[r]^-{#2} &
    {#3} \ar[r]^-{#4} &
    {#5} \ar[r]^-{#6} &
    {#7}
  }}

\newsavebox{\tvabox}
\savebox\tvabox{\hspace{1mm}\begin{tikzpicture}[>=latex',baseline={(0,-.18)}]
  \draw[->] (0,.1) -- +(1,0);
  \node at (.5,0) {$\scriptscriptstyle\bot$};
  \draw[->] (1,-.1) -- +(-1,0);
  \draw[->] (1,-.2) -- +(-1,0);
\end{tikzpicture}\hspace{1mm}}

\renewcommand{\ex}{\mathrm{ex}}
\newcommand{\Ch}{\mathrm{Ch}}

\newcommand{\Mod}[1]{\mathrm{Mod}({#1})}

\newcommand{\sse}{\stackrel{\mathrm{s}}{\sim}}

\newcommand{\cok}{\mathrm{cok}}

\newcommand{\A}[1]{{\Vec{A}_{#1}}}

\newtheorem*{thm*}{\textbf{Theorem}}

\title{Higher symmetries in abstract stable homotopy theories}

\author{Moritz Groth and Moritz Rahn}
\address{Rheinische Friedrich-Wilhelms-Universit{\"a}t Bonn, Mathematisches Institut, Ende-nicher Allee 60, 53115 Bonn, Germany, and Johannes Gutenberg-Universit{\"a}t Mainz, Institut f{\"u}r Mathematik, Staudingerweg 9, 55128 Mainz, Germany}
\email{mgroth@math.uni-bonn.de, morgroth@uni-mainz.de}

\date{\today}

\begin{document}

%\linenumbers
%\modulolinenumbers[5]

\begin{abstract}
This survey offers an overview of an on-going project on uniform symmetries in abstract stable homotopy theories. This project has calculational, foundational, and representation-theoretic aspects, and key features of this emerging field on abstract representation theory include the following. First, generalizing the classical focus on representations over fields, it is concerned with the study of representations over rings, differential-graded algebras, ring spectra, and in more general abstract stable homotopy theories. Second, restricting attention to specific shapes, it offers an explanation of the axioms of triangulated categories, higher triangulations, and monoidal triangulations. This has led to fairly general results concerning additivity of traces. Third, along similar lines of thought it suggests the development of abstract cubical homotopy theory as an additional calculational toolkit. An interesting symmetry in this case is given by a global form of Serre duality. Fourth, abstract tilting equivalences give rise to non-trivial elements in spectral Picard groupoids and hence contribute to their calculation. And, finally, it stimulates a deeper digression of the notion of stability itself, leading to various characterizations and relative versions of stability.
\end{abstract}

\maketitle

\tableofcontents

\section{Introduction}
\label{sec:intro}

A good part of modern mathematics is centered around the study of symmetries.\footnote{Always begin an introduction by drawing a big picture.} Symmetries arise in various parts of mathematics, and in many specific situations good control over the available symmetries is helpful in (if not crucial to) applications. Here we are interested in global symmetries which are common to all abstract stable homotopy theories. In one form or another (see further below), abstract stable homotopy theories are in the background of many areas of pure mathematics. This certainly includes algebra and representation theory, homotopy theory and algebraic topology as well as algebraic geometry. The overall goal of this project is a detailed study of uniform symmetries which arise in these fields. The credo is that some of these symmetries provide a convenient calculational toolkit that is common to all these situations. Other symmetries are more interesting from an abstract representation-theoretic perspective.

In order to fill this first paragraph with more life, for the time being we focus on \emph{triangulated categories} as one of the classical approaches to an ``abstract stable homotopy theory''. Triangulated categories originated in algebraic geometry in the study of derived categories $D(X)$ of schemes \cite{verdier:thesis,verdier:derived} and in homotopy theory in the study of \emph{the} stable homotopy category of spectra $\mathcal{SHC}$ \cite{boardman:thesis,puppe:stabil,vogt:boardman}. In both of these situations, the classical triangulations on $D(X)$ and $\mathcal{SHC}$ encode aspects of the calculus of cones or, equivalently, derived cokernels -- as it is visible to the respective category alone. And symmetries show up very prominently in the definition of a triangulated category \cite{neeman:triangulated,hjr:triangulated}. First, a triangulated category $\cT$ is by definition endowed with a suspension functor $\Sigma\colon\cT\toiso\cT$ which is an equivalence of categories. Second, the rotation axiom allows us to rotate distinguished triangles forward and backwards. Triangulated categories also encode information about pairs of composable morphisms, but here the classical axioms miss some of the existing symmetries. 

Correspondingly, our first goal is to give a different explanation of the axioms of a triangulated category. In order to achieve this in a fairly elementary way, in~\S\ref{sec:rep-thy-tria} we stick to the specific situation of chain complexes of modules over a ring $R$. The main goal in that section is to give a somewhat unorthodox construction of the classical Verdier triangulation on the derived category $D(R)$. In fact, we take a representation-theoretic perspective and study representations of the $A_n$-quivers
\[
\A{1}=1, \quad \A{2}=(1\to 2),\quad \text{and} \quad \A{3}=(1\to 2\to 3)
\]
with values in chain complexes. Suitable combinations of derived cokernels and more general derived pushouts reveal interesting symmetries in the representation theories of these $A_n$-quivers. For instance, a representation
\begin{equation}\label{eq:A3-intro}
x\stackrel{f}{\to} y\stackrel{g}{\to} z
\end{equation}
of $\A{3}$ in chain complexes can be equivalently encoded by a diagram of chain complexes as in \autoref{fig:octa-intro}. In that diagram all squares are derived pushouts (or, equivalently, derived pullbacks) and the diagram vanishes on the boundary stripes. 

A suitable restriction of \autoref{fig:octa-intro} gives rise to the octahedral diagram of the representation \eqref{eq:A3-intro}. For instance, the key distinguished triangle from the octahedral 

\begin{figure}[h]
\centering
\[
\xymatrix{
\ar@{}[dr]|{\ddots}&\ar@{}[dr]|{\ddots}&\ar@{}[dr]|{\ddots}&\ar@{}[dr]|{\ddots}&\ar@{}[dr]|{\ddots}&&&&&\\
&0\ar[r]&\tilde z\ar[d]\ar[r]&\tilde v\ar[r]\ar[d]&\tilde w\ar[d]\ar[r]&0\ar[d]&&&&\\
&&0\ar[r]&x\ar[r]^-f\ar[d]&y\ar[r]^-g\ar[d]&z\ar[d]\ar[r]&0\ar[d]&&&\\
&&&0\ar[r]&u\ar[r]\ar[d]\ar@{}[dr]|{\fbox{1}}&v\ar[r]\ar[d]&x'\ar[d]\ar[r]&0\ar[d]&&\\
&&&&0\ar[r]\ar@{}[dr]|{\ddots}&w\ar[r]\ar@{}[dr]|{\ddots}&y'\ar[r]\ar@{}[dr]|{\ddots}&u'\ar[r]\ar@{}[dr]|{\ddots}&0\ar@{}[dr]|{\ddots}&\\
&&&&& & & & &
}
\]
\caption{A symmetric presentation of representations of $\A{3}$.}
\label{fig:octa-intro}
\end{figure}
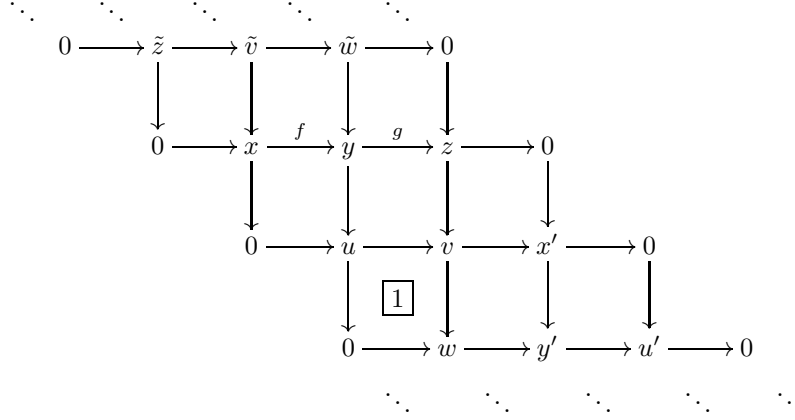
\noindent
axiom (``the cone of cones is a cone'') is induced by the square $\fbox{1}$ which exhibits the cone $C(g)$ as the cone of $C(f)\to C(gf)$ (see \autoref{rmk:octa} for more details). But the diagrams in \autoref{fig:octa-intro} have the advantage that there are two obvious symmetry operations. For instance, there is the symmetry which shifts such diagrams along the diagonal. And this symmetry results in a relation between the octahedral diagram of \eqref{eq:A3-intro} and the one of the induced representation
\[
C(f)\to C(gf)\to\Sigma x. 
\]
Also representations of $\A{2}$ and $\A{1}$ can be encoded by similar symmetric diagrams (see \autoref{fig:BP} and \autoref{fig:spectra}). And in that case the symmetries induce the familiar rotation symmetry and suspension at the level of $D(R)$.

Following this line of thought, we are led to a construction of the classical Verdier triangulation on $D(R)$. The way we present this construction suggests that these arguments do not rely in an essential way on the specifics of chain complexes in $R$-modules. Instead, only formal aspects of the calculus of chain complexes were invoked, namely the existence of a well-behaved calculus of derived limits and colimits with the following additional key properties.
\begin{enumerate}
\item There is a zero object.
\item A square is a derived pushout square if and only if it is a derived pullback square.
\end{enumerate}
These two properties are the typical defining features of an ``abstract stable homotopy theory''. There is an entire zoo of models all of which make such a notion precise. This zoo includes stable cofibration categories \cite{schwede:p-order,lenz:derivators}, stable model categories \cite{hovey:model}, stable $\infty$-categories \cite{HTT,HA}, and stable derivators \cite{heller:stable,franke:adams,maltsiniotis:seminar} (see \autoref{discl:higher-cats} for more references). For many of our purposes the precise choice of the model is not too relevant. In fact, many of the proofs are quite formal (and that is one of the points of the project) so that they are essentially model independent. Let us hasten to emphasize that we believe in diversity of technology. 

These abstract stable homotopy theories are enhancements of triangulated categories. In fact, in the above cases one can use the defining exactness properties of \emph{stability} in order to construct canonical triangulations. More interesting than this result itself are the techniques leading to its proof. These are developed in \S\ref{sec:crash} in the framework of stable derivators (but see \autoref{discl:higher-cats}), and these techniques are crucial to our later studies of symmetries. One convenient feature of stable derivators (which is, for instance, also enjoyed by stable $\infty$-categories \cite[Prop.~1.1.3.1]{HA} but not by triangulated categories) is the following. Given a stable derivator \D and a small category $A$ there is the stable derivator $\D^A$ of representations of shape $A$ with values in $\D$ (exponentials exist in stable derivators). And this opens the door for the study of symmetries of stable homotopy theories of representations of more complicated shapes than the $A_n$-quivers for $n=1,2,3$ which show up implicitly in the axioms of triangulated categories.

In \S\ref{sec:ART} we finally get to the heart of this project, and this is where we take the above exponentials serious. Given a small category $A$ we associate to it the $2$-functor
\begin{equation}\label{eq:exponentials-intro}
(-)^A\colon\cDER_{\mathrm{St,ex}}\to\cDER\colon\D\mapsto\D^A
\end{equation}
which sends a stable derivator \D to the stable derivator of $A$-shaped representations in \D. By specializing to specific \D, this yields the homotopy theory of representations over fields, over rings, in quasi-coherent modules over schemes, of differential-graded or of spectral representations (see \autoref{egs:stable-2} for details). We think of $(-)^A$ as a gadget that neatly encodes the abstract representation theory of the shape $A$. Correspondingly, given two small categories $A$ and $B$ (which could be the same), a \emph{strong stable equivalence} $\Phi\colon A\sse B$ is a family of equivalences
\[
\Phi_\D\colon\D^A\simeq\D^B,\quad\D \text{ a stable derivator,}
\]
which is pseudo-natural with respect to exact morphisms in \D. As a slogan, such a strong stable equivalence shows that $A$ and $B$ have the same abstract representation theory (see \S\ref{subsec:sse} for a more detailed discussion).

A priori, any strong stable equivalence $\Phi\colon A\sse B$ reveals a potentially interesting symmetry in abstract stable homotopy theories. However, we also believe in the following slogan.
\begin{center}
``Some shapes are more important than others.''
\end{center}
In fact, some shapes are more foundational than others, and strong stable equivalences between such shapes buy us more. We illustrate this by the following three shapes which are closely related to triangulated categories, higher triangulated categories and the Waldhausen $S_\bullet$-construction, and monoidal triangulated categories and Goodwillie calculus.
\begin{enumerate}
\item The abstract representation theories of the Dynkin quivers $\A{1}, \A{2},$ and $\A{3}$ are intimately related with the axioms of triangulated categories. This relation is provided through the calculus of derived (co)kernels of morphisms and pairs of composable morphisms \cite{franke:adams,maltsiniotis:seminar,groth:ptstab,groth:revisit}. Similarly, longer chains of composable morphisms are the same as abstract representations of Dynkin quivers
\[
\A{n}=(1\to2\to\ldots\to n).
\]
And their abstract representation theory is directly connected with higher octahedral diagrams and $\infty$-triangulated categories \cite{beilinson:perverse,maltsiniotis:higher,balmer:separability,gst:Dynkin-A}. Interesting symmetries are made visible by diagrams as in Walhausen's $S_\bullet$-construction in algebraic $K$-theory \cite{waldhausen:k-theory,gst:Dynkin-A}. Higher-dimensional versions of these diagrams have recently been studied in \cite{poguntke:higher-segal,beckert:thesis,dycker-jasso-walde:higher-AR}. 
\item Similarly, the abstract representation theory of the commutative square
\[
\square=\A{2}\times\A{2}
\]
is closely related to refined axioms for monoidal, triangulated or tensor triangulated categories. Of course, representations of the square arise in monoidal contexts as pointwise products of pairs of morphisms. A prominent operation in this context is the pushout product operation (\autoref{eg:pushout-product}). Shadows of this operation together with various compatibilities with rotations of the two morphisms, symmetry, and dualization have been axiomatized by May \cite{may:additivity}. The square is strongly stably equivalent to the trivalent source and this offers a more representation-theoretic perspective on May's axioms \cite{keller-neeman:D4,gst:basic}.   

This calculus leads to applications in the study of dualizability phenomena. More specifically, part of Grothendieck's original motivation to look for enhancements of triangulated categories was the following observation: traces of endomorphisms do not satisfy additivity at the level of triangulated categories (see \autoref{rmk:monoidal-triang} for a precise formulation). In \cite{gps:additivity} a result concerning additivity of traces is established for stable, monoidal derivators (following May's proof \cite{may:additivity} for model categories). And, again, possibly more interesting than the additivity result are the techniques developed in that paper. (For further developments of this additivity result see \cite{ps:linearity,ps:linearity-fp,gallauer:traces,jin-yang}.) 
\item Finally, also the abstract representation theory of the $n$-cube
\[
\square^n=\A{2}\times\ldots\times\A{2}
\]
and some of its subposets is fairly foundational. It leads to a detailed understanding of the calculus of iterated partial cofiber constructions, total cofibers, and of an interpolation between cocartesian and strongly cocartesian $n$-cubes \cite{bg:cubical}. Such shapes show up naturally in monoidal contexts and also in Goodwillie calculus \cite{goodwillie:II,munson-volic}. Moreover, there is a global form of an abstract Serre duality on representations of these posets \cite[\S12]{bg:cubical}, which, conjecturally, can be extended to arbitrary homotopy finite shapes. Jointly with Falk Beckert, we intend to come back to this elsewhere.
\end{enumerate}

Besides these more foundational cases, there are various shapes that have been studied very systematically in representation theory. In representation theory the focus is often on representations of quivers, posets, groups, or more general small categories with values in vector spaces over a field. (Of course a good deal of representation theory is concerned with more general algebras, but here we focus on those aspects that can be captured by \eqref{eq:exponentials-intro}.) Tilting theory is a derived version of Morita theory, and it offers many tools to study the corresponding derived categories and the calculus of derived equivalences (see for example \cite{apr:tilting,brenner-butler:tilting,happel-ringel:tilted-algebras,rickard:derived-fun,keller:deriving-dg} or the survey articles in \cite{angeleri-happel-krause:handbook}). It turns out that some aspects of tilting theory have interesting counterparts in the context of abstract stable homotopy theories.

First, this is the case for the classical BGP reflection functors. Let us recall that Bern{\v{s}}te{\u\i}n, Gel$'$fand, and Ponomarev \cite{bernstein-gelfand-ponomarev:Coxeter} introduced reflection functors and used them to give an elegant proof of Gabriel's classification result of connected hereditary representation-finite algebras over an algebraically closed field \cite{gabriel:unzerlegbare}. For every quiver~$Q$ and every source or sink $q\in Q$, the reflected quiver $Q'=\sigma_qQ$ of $Q$ is obtained by reversing the orientations of all edges adjacent to~$q$. Associated with these reflections at the level of quivers are reflection functors
\[
\mathrm{Mod}(kQ)\to\mathrm{Mod}(kQ')
\]
between the module categories of the corresponding path algebras. While these reflection functors fail to be equivalences, Happel \cite{happel:fd-algebra} showed that for acyclic quivers derived reflection functors are exact equivalences 
\[
D(kQ)\stackrel{\Delta}{\simeq} D(kQ').
\]
Later Ladkani~\cite{ladkani:posets} established a similar equivalence $D(\cA^Q) \simeq D(\cA^{Q'})$ for arbitrary abelian categories. It was shown in \cite{gst:tree,gst:acyclic} that Happel's theorem can be strengthened further by showing that reflection functors yield a strong stable equivalence
\[
Q\sse Q'.
\]
And this also holds for not necessarily finite or acyclic quivers. By purely combinatorial arguments one concludes from this that if $T$ is a finite oriented tree and if $T'$ is an arbitrary reorientation of $T$, then there is a strong stable equivalence
\[
T\sse T'.
\]

Using the more flexible framework of stable $\infty$-categories further generalizations were recently constructed by Dyckerhoff, Jasso, and Walde \cite{dycker-jasso-walde:BGP}. In \emph{loc.~cit.} certain procedures to glue stable $\infty$-categories are developed, and this gluing works neither for triangulated categories nor for stable derivators. These techniques are then applied to offer a fairly general construction of reflection functors for stable $\infty$-categories. Their result specializes both to the reflection functors in \cite{gst:tree,gst:acyclic} and also to additional examples considered by Ladkani \cite{ladkani:posets,ladkani:posets-tilting,ladkani:posets-cluster-tilting}.

A second aspect from tilting theory that generalizes nicely is the calculus of tilting complexes. Recall that these are chain complexes of bimodules such that the corresponding derived tensor and hom functors are derived equivalences. To obtain a counterpart of this for abstract stable homotopy theories, one invokes the universality of the homotopy theory of spectra: it is the free stable homotopy theory generated by the sphere spectrum. In the framework of derivators references related to this result include \cite{heller:htpythies,heller:stable,franke:adams,cisinski:derived-kan,tabuada:universal-invariants,cisinski-tabuada:non-connective,cisinski-tabuada:non-commutative} (but see also \cite{dugger:universal,lenhardt:frames} and \cite{HTT,HA,ggn:infinite} for closely related results in model categories and $\infty$-categories, respectively). As a consequence of this universality and Brown representability, every stable derivator is enriched over spectra. And this allows us to define tensor and hom functors associated to spectral bimodules (aka.~weighted colimits and weighted limits \cite{kelly:enriched}). 

It turns out that the strong stable equivalences mentioned above are induced by \emph{universal tilting modules}. These are certain spectral bimodules and the terminology is justified by the following two facts. First, these bimodules realize strong stable equivalences in arbitrary stable derivators and, second, they are spectral refinements of the classical tilting complexes. In fact, the classical tilting complexes are recovered from the universal tilting modules by smashing with the Eilenberg--MacLane spectrum. One convenient feature of these spectral bimodules is that they can be calculated quite explicitly. For instance, there are explicit spectral bimodules that govern the calculus of homotopy finite limits and colimits. We illustrate this in \S\ref{subsec:modules} by some examples related to the calculus of cofibers such as the universal constructor for cofiber sequences.

Universal tilting modules realize strong stable \emph{equivalences} and they are consequently \emph{invertible} spectral bimodules. Considered from this perspective, the construction of strong stable equivalences contributes to the calculation of spectral Picard groupoids. To the best of the knowledge of the author, the only shape for which the spectral Picard group is known is the trivial shape $A=\ast$. In that case there is an isomorphism $\lZ\cong\mathrm{Pic}_{\cSp}(\ast)$
and a generator is given by $\lS^1=\Sigma\lS$ (\cite{HMS:Picard} or \cite[Thm.~2.2]{strickland:interpolation}). One of the current goals of this project is to obtain a calculation of spectral Picard groups at least for some shapes, but for the time being only first steps have been achieved (\autoref{rmk:Picard}).

We refrain from giving a more detailed description of the content. Instead we refer the reader to the respective introductions of the individual sections \S\ref{sec:rep-thy-tria}, \S\ref{sec:crash}, and \S\ref{sec:ART}. 

\textbf{Philosophy of this writing.} It was our main goal to write an account which is readable to representation theorists, to triangulated category theorists, and also to (abstract) homotopy theorists. This explains the high level of detail which is offered at various places, and we want to apologize for this to these different communities. Still our feeling is that this is a reasonable compromise. As additional guiding principles, we tried to stress the main ideas in the project, we refer to the original literature for most of the proofs, we advertise a few loose ends which are to be pursued further, and we include many references to indicate the various connections to other parts of mathematics. 

The level of sophistication increases from section to section. In \S\ref{sec:rep-thy-tria} we deliberately only invoke elementary techniques from homological algebra and some category theory. In \S\ref{sec:crash} we give a survey on derivators which offers one of the many approaches to abstract homotopy theory (but see \autoref{discl:higher-cats}). We just develop the basics which are needed for the applications in the final section (and here we gloss over the more technical aspects of the theory despite the fact that the author actually likes them). Finally, in \S\ref{sec:ART} we freely use the language of derivators (but see \autoref{discl:higher-cats}) in order to summarize and explain the main results of this project on higher symmetries. 

\textbf{Introduction of new coauthor.} It is very tempting to make various jokes such as ``the first author thanks the second author for having joined the project while the second author in turn thanks the first author for fruitful discussions''. But to cut this short, the author recently got married and changed his last name from Groth to Rahn. All future publications will appear under this new last name.

\textbf{Acknowledgments.} It is a great pleasure to begin by thanking my Ph.D. supervisor Stefan Schwede for having suggested to me to think about abstract homotopy theory in the first place. My understanding of the subject was sharpened through various more or less directly related cooperations, hence special thanks go to my coauthors Falk Beckert \cite{bg:cubical}, David Gepner and Thomas Nikolaus \cite{ggn:infinite}, Kate Ponto and Mike Shulman \cite{gps:mayer,gps:additivity}, again Mike Shulman \cite{gs:generalized} and, last but not least (watch the alphabetic order!), Jan {\v S}{\v t}ov{\'\i}{\v c}ek \cite{gst:basic,gst:tree,gst:Dynkin-A,gst:acyclic}. During the last years many colleagues and friends gave me (willingly or not) the opportunity to discuss this project and/or to have a good time. It is a great pleasure to thank Dimitri Ara, Peter Arndt, Denis-Charles Cisinski, Ivo Dell'Ambrogio, David Gepner, Drew Heard, Gustavo Jasso, Andr{\'e} Joyal, Philipp Jung, Bernhard Keller, Steffen K{\"o}nig, Henning Krause, Rosie Laking, Georges Maltsiniotis, Thomas Nikolaus, Justin No{\"e}l, Eric Peterson, Mona Rahn, George Raptis, Ulrich Schlickewei, Timo Sch{\"u}rg, Stefan Schwede, Mike Shulman, Johan Steen, Greg Stevenson, and Jan {\v S}{\v t}ov{\'\i}{\v c}ek.

The author also thanks the organizers of the International Conference on Representation theory of Algebras in Prague in August 2018 (ICRA2018). On occasion of that conference the author offered a series of three talks aiming for an overview of this project. The sections of this account are in obvious bijection with the talks given at ICRA2018, but here we intend to draw a slightly more complete picture.

%formal study of stability; relative stabilization; version of rep thy with values therein; alg vs top triangulated cats.
 %
%\begin{enumerate}
%\item \emph{Formal study of stability: }
%\item \emph{Generalization of respresentation theory: }
%\end{enumerate}
%
%Here we chose\fxnote{Check mails of Shulman and pointer to arxiv} to work in stable land; variants (abstract stabilization; alternatively over other stable monoidal derivators; for instance if work over the integers instead of working over the sphere spectrum algebraic versus triangulated categories) (motivic, parametrized, equivariant twist?)

%These are slightly expanded lecture notes for a series of three talks given by the author at ICRA~2018 in Prague. The overall goal of these talks was to give an introduction to the project on \emph{abstract representation theory} which is to a large extent joint with Jan Stovicek \cite{,,,} and, more recently, also with Falk Beckert \cite{bg:cubical}.  

\section{A representation-theoretic perspective on triangulated categories}
\label{sec:rep-thy-tria}

\emph{Triangulated categories} were introduced in the 1960's, originally motivated by situations in algebraic geometry \cite{verdier:thesis,verdier:derived} and topology \cite{boardman:thesis,puppe:stabil,vogt:boardman}. Since their invention triangulated categories have become a valuable tool in many areas of pure mathematics and this ubiquity is one of the features of triangulated categories. In algebraic geometry triangulated categories arise as derived categories of schemes (\cite{verdier:thesis,verdier:derived} or \cite{huybrechts:fourier}), in representation theory they come up as derived categories of algebras (see \cite{happel:triangulated} or \cite{angeleri-happel-krause:handbook}), in modular representation theory as stable module categories \cite{benson-rickard-carlson:thick-stmod}, and in homotopy theory as homotopy categories of spectra or related (stable model) categories (\cite{vogt:boardman} or \cite{hovey:model}). Nice surveys on this ubiquity of triangulated categories can be found in \cite{hjr:triangulated}, in \cite{balmer:TTG}, and from a different perspective in \cite{schwede-shipley:morita}.

In this first section we want to contribute to the understanding of the axioms of triangulated categories. The main idea we intend to transport is that the axioms of triangulated categories are closely related to (abstract) representation theory of $A_1$-quivers (!), $A_2$-quivers, and $A_3$-quivers. To keep things simple and very explicit, in most of this section we specialize to chain complexes of modules over a ring~$R$. Assuming only elementary techniques from homological algebra, we revisit the classical Verdier triangulation on the derived category $D(R)$ of a ring and offer a somewhat unorthodox construction of it. In \S\ref{sec:crash} we will see that the arguments invoked here easily extend to more general situations.

In \S\ref{subsec:cones} we revisit the classical derived cokernel and derived pushout construction in abelian categories. In \S\ref{subsec:der-limit} we specialize to the case of a module category and collect key aspects of the calculus of more general derived limits and Kan extensions. As an illustration we express derived cokernels in terms of derived Kan extensions. In \S\ref{subsec:BP} we iterate the construction of cofibers. This shows that morphisms can be equivalently encoded by cofiber sequences and also by Barratt--Puppe sequences. This latter reformulation showcases the existing derived symmetries of $\A{2}$-representations. In \S\ref{subsec:octa} we consider derived representations of $\A{3}$ and encode these in terms of refined ocatahedral diagrams. We discuss at some length in which sense these diagrams are more symmetric versions of the diagrams in the classical octahedral axiom.

\subsection{Functorial cones}
\label{subsec:cones}

To set the stage, let us begin by establishing some basic notation.

\begin{notn}
Throughout this section, $\cA$ denotes an abelian category. From \S\ref{subsec:der-limit} on we focus on the case $\cA=\Mod{R}$, the category of (not necessarily finitely generated) left modules over a (not necessarily commutative) ring $R$.
\end{notn}

We recall that abelian categories essentially axiomatize the usual calculus of finite direct sums of modules and kernels and cokernels of homomorphisms of modules. In classical homological algebra, the focus is on additive functors between abelian categories and on an investigation of their behavior with respect to short exact sequences. In particular, various techniques are developed in order to approximate a given functor by exact functors in a universal way. To make such a statement more specific, we collect the following lemma.

\begin{notn}
For every abelian category $\cA$, we denote by $\nCh(\cA)$ the category of unbounded chain complexes in $\cA$.
\end{notn} 

\begin{lem}
Let $\cA$ and $\cB$ be abelian categories and let $F\colon\cA\to\cB$ be an additive functor. The following are equivalent.
\begin{enumerate}
\item The functor $F\colon\cA\to\cB$ is exact.
\item The induced functor $F\colon\nCh(\cA)\to\nCh(\cB)$ preserves quasi-isomorphisms.
\end{enumerate}
\end{lem}
\begin{proof}
The passage to homology objects is obtained by the formation of subquotients, and the proof is hence straightforward.
\end{proof}

This suggests that the above-mentioned `exact approximations' of more general additive functors are defined at the level of derived categories.

\begin{con}
For every abelian category $\cA$, we denote by $W_\cA$ the class of quasi-isomorphisms in $\nCh(\cA)$. Up to set-theoretic considerations, there is the localization functor
\begin{equation}\label{eq:loc}
\vcenter{
\xymatrix{
\nCh(\cA)\ar[rrrrr]^-\gamma&&&&&D(\cA)=\nCh(\cA)[(W_\cA)^{-1}],
}
}
\end{equation}
which universally inverts the quasi-isomorphisms $W_\cA$ \cite{gabriel-zisman:calculus}. In the case of a module category $\cA=\Mod{R}$ we use the standard notation $\nCh(R)\to D(R)$.
\end{con}

It is common to think of the derived category $D(\cA)$ as a rather refined invariant of the abelian category $\cA$. A good deal of information about homological invariants is encoded by derived categories, and many such invariants are invariant under derived equivalences (see, for instance, \cite{keller:der-cat-and-tilt} and the many references therein). 

Here we want to take a different, but related perspective and focus on formal properties of derived categories. The reason that we drew the morphism in \eqref{eq:loc} in such a long way is that we want to stress that the categories $\nCh(\cA)$ and $D(\cA)$ live in rather different worlds. The category $\nCh(\cA)$ is an abelian category and hence enjoys rather nice properties. In particular, it has all finite limits and finite colimits. In contrast to this, the category $D(\cA)$ is rather ill-behaved (the 1-categorical localization procedure is an act of violence, and it destroys these nice properties). There are two relevant ways to make this more precise.
\begin{enumerate}
\item First, derived categories $D(\cA)$ tend to not have many limits or colimits. In fact, the only typical (co)limits which exist are finite biproducts and in some cases also infinite (co)products.  
\item More importantly, the calculus of \emph{derived (co)limits} is not visible to the category $D(\cA)$ \emph{alone}.
\end{enumerate} 

There are various ways to deal with these observations (see \autoref{discl:higher-cats} and references there), and here we begin by following Verdier and Grothendieck. The following classical theorem offers one way to react upon the above issues, and the remaining goal of this first section is to explain that result from the perspective of symmetries.

\begin{thm}[Verdier \cite{verdier:thesis,verdier:derived}]\label{thm:verdier}
The derived category $D(\cA)$ ``is'' a triangulated category.
\end{thm}

In the formulation of this theorem, we were rather picky and put the verb into quotation marks. This is in order to stress that a triangulation amounts to the specification of additional \emph{structure} on $D(\cA)$ (in contrast to asking for mere \emph{properties} of $D(\cA)$). In particular, this means that there are 
\begin{enumerate}
\item an auto-equivalence $\Sigma\colon D(\cA)\toiso D(\cA)$ and
\item a class of distinguished triangles $x\stackrel{f}\to y\to z\to \Sigma x$ in $D(\cA)$,
\end{enumerate}
and these are subject to certain axioms \cite{neeman:triangulated,hjr:triangulated}. Given such a distinguished triangle, the object $z$ is referred to as ``the'' cone or cofiber of $f$. This time, the quotation marks are meant to allude to the following well-known fact. Approached this way, cones are only unique up to non-canonical isomorphism and only weakly functorial. 

The following proposition describes an alternative approach to the cone construction. It exhibits the cone as the total left derived cokernel functor, thereby guaranteeing that the cone is a canonical and functorial construction.

\begin{prop}\label{prop:C-cok}
Let $\cA$ be an abelian category, let $[1]$ be the poset $(0<1)$, let $\cA^{[1]}$ be the category of morphisms in $\cA$, and let $\cok\colon\cA^{[1]}\to\cA$ be the cokernel functor.
\begin{enumerate}
\item The cokernel functor is right exact.
\item The cokernel functor is exact on monomorphisms.
\item The cokernel functor has a total left derived functor given by the cone construction,
\[
C\cong L\cok\colon D(\cA^{[1]})\to D(\cA).
\]
\end{enumerate}
\end{prop}
\begin{proof}
A short exact sequence in $\cA^{[1]}$ is simply a morphism of short exact sequences in $\cA$,
\[
\xymatrix{
0\ar[r]&x'\ar[r]\ar[d]_-{f'}&x\ar[r]\ar[d]^-f&x''\ar[r]\ar[d]^-{f''}&0\\
0\ar[r]&y'\ar[r]&y\ar[r]&y''\ar[r]&0.
}
\]
\begin{enumerate}
\item This is an immediate consequence of the snake lemma, which yields a 6-term exact sequence
\[
0\to \ker(f')\to\ker(f)\to\ker(f'')\to\cok(f')\to\cok(f)\to\cok(f'')\to 0.
\]
\item In the case of monomorphisms, the above sequence reduces to a short exact sequence, so also this statement is immediate from the snake lemma.
\item While the third statement does not follow as directly from the snake lemma, the snake lemma suggests what we are supposed to do. Namely, we should approximate an arbitrary chain map up to quasi-isomorphism by a monomorphic one and apply the cokernel to it instead. The details of this are as follows. 

Given a chain map $f\colon x\to y$, we note that $\cok(f)$ is the pushout in the quadrilateral on the right in 
\begin{equation}\label{eq:C-cok}
\vcenter{
\xymatrix{
x\ar[r]^-f\ar[d]_-i&y\ar@{-->}[rdd]^-{\cof(f)}& &&& x\ar[r]^-f\ar[d]&y\ar@{-->}[rdd]\\
Cx\ar@{-->}[rrd]&& &&& 0\ar@{-->}[rrd]&&\\
&&Cf\ar@{..>}[rrrrr] &&& && \cok(f).
}
}
\end{equation}
The intended monomorphic approximation up to quasi-isomorphism is displayed in the remaining quadrilateral. Therein, $x\mapsto Cx$ denotes the classical functor which sends a \emph{chain complex} to its cone (see, for instance, \cite{gelfand-manin:homological,weibel:homological}). The relevant key facts about this construction are that, first, the cone $Cx$ is an acyclic complex and, second, it comes with a natural monomorphic chain map $i\colon x\to Cx$. Now, the cone $Cf$ of the \emph{chain map}~$f$ is defined by the pushout square on the left in \eqref{eq:C-cok}. Note that, since the induced map $x\to Cx\oplus y$ is a monomorphism, by the first two parts the resulting functor $C\colon\nCh(\cA^{[1]})\to\nCh(\cA)\colon f\mapsto Cf$ preserves quasi-isomorphisms. Consequently, it descends to a cone functor at the level of derived categories
\[
C\colon D(\cA^{[1]})\to D(\cA).
\]
In order to exhibit this as the total left derived functor of the cokernel, it only remains to construct a corresponding natural transformation. Since the cone $Cx$ is acyclic, the unique chain map $Cx\to 0$ together with the identities $\id_x$ and $\id_y$ induces a levelwise quasi-isomorphism 
\[
(Cx\ot x\to y)\to (0\ot x\to y),
\]
i.e.,~the intended monomorphic resolution. At the level of pushouts this yields a natural chain map $Cf\to\cok (f)$ (the dotted arrow in \eqref{eq:C-cok}), which in turn gives rise to the intended natural transformation in
\[
\xymatrix{
\nCh(\cA^{[1]})\ar[r]^-C\ar[d]_-\gamma&\nCh(\cA)\ar[d]^-\gamma\\
D(\cA^{[1]})\ar[r]_-C&D(\cA).\ultwocell\omit{\varepsilon}
}
\]
Standard techniques imply that $(C,\varepsilon)$ is a model for $L\cok.$ 
\end{enumerate}
(For more details and a more systematic explanation we refer to \cite[\S3]{groth:thy-of-der}.)
\end{proof}

We want to briefly discuss this result.

\begin{rmk}\label{rmk:der-mor}
For every abelian category $\cA$ the category of morphisms $\cA^{[1]}$ is again abelian. It is important to distinguish the two categories
\[
D(\cA^{[1]})\qquad\text{and}\qquad D(\cA)^{[1]},
\]
the \emph{derived category of the morphism category} and the \emph{morphism category of the derived category}. A relation between these categories is provided by a forgetful functor
\[
\ndia_{[1]}\colon D(\cA^{[1]})\to D(\cA)^{[1]}
\]
which is constructed as follows. An object $X\in D(\cA^{[1]})$ is a chain complex of morphisms which we can hence rewrite as a functor $X\colon[1]\to\nCh(\cA)$. Written this way, we can compose $X$ with the localization functor $\gamma\colon\nCh(\cA)\to D(\cA)$, and we make the definition
\[
\ndia_{[1]}(X)=\gamma\circ X\colon [1]\to D(\cA).
\] 
The crucial observation now is that this functor discards relevant information. In particular, the functorial cone from \autoref{prop:C-cok} does not factor through this underlying diagram functor,
\[
\xymatrix{
D(\cA^{[1]})\ar[r]^-C\ar[d]_-{\ndia_{[1]}}&D(\cA)\\
D(\cA)^{[1]}\ar[ru]_-{\nexists C}.&
}
\]
(This fails already for vector spaces over a field as is detailed in \cite[\S5]{groth:thy-of-der}.)
\end{rmk}

Thus, to put this into plain English, there is a functorial cone construction on the \emph{derived category of the morphism category} but not on the \emph{morphism category of the derived category}. In order to develop more intuition for these underlying diagram functors, we collect the following straightforward result \cite[Cor.~3.9]{bg:cubical}.

\begin{eg}\label{eg:dia-H}
Let $R$ be a semi-simple ring, let $B$ be a category with finitely many objects, and let $RB$ be the $R$-linear category algebra. There is a similar underlying diagram functor
\[
\ndia_B\colon D(\Mod{R}^B)\to D(R)^B\colon X\mapsto\gamma\circ X,
\]
which is equivalent to the $\lZ$-graded homology functor $H_\ast\colon D(RB)\to \Mod{RB}^\lZ$.
\end{eg}

This example illustrates very nicely that underlying diagram functors discard relevant information. It turns out that not only these underlying diagram functors tend to fail to be equivalences, but also that the categories
\[
D(\cA^B)\qquad\text{and}\qquad D(\cA)^B
\]
are not equivalent (see \cite[\S5]{groth:thy-of-der} for an explicit example). Moreover, there is no simple recipe which allows us to reconstruct derived categories of diagram categories (categories of the form $D(\cA^B)$) from the derived category $D(\cA)$. If we want to study derived (co)limits at the level of derived categories only (in contrast to the approaches in \autoref{discl:higher-cats}), then we have to keep track of the categories $D(\cA^B)$ themselves. 

Before we state a general theorem about the existence of derived (co)limits (\autoref{thm:grothendieck}), we give an elementary proof in the following central example. We denote by $\ulcorner$ the full subposet of the cube looking like:
\begin{equation}\label{eq:span}
\vcenter{
\[
\xymatrix{
(0,0)\ar[r]\ar[d]&(0,1)\\
(1,0)&
}
\]
}
\end{equation}
Correspondingly, the diagram category $\cA^\ulcorner$ is the category of spans in $\cA$ and the colimit functor
\begin{equation}\label{eq:pushout-ab}
\colim_{\ulcorner}\colon\cA^\ulcorner\to\cA
\end{equation}
simply forms pushouts.

\begin{prop}\label{prop:der-push}
Let $\cA$ be an abelian category. The pushout functor \eqref{eq:pushout-ab} has a total left derived functor
\[
L\colim_\ulcorner\colon D(\cA^\ulcorner)\to D(\cA).
\]
\end{prop}
\begin{proof}
The strategy of the proof is very similar to the one in \autoref{prop:C-cok}, and we are more sketchy. As a left adjoint the pushout functor is right exact, and it turns out to be exact on those spans
\[
\xymatrix{
x\ar[r]\ar[d]&y\\
z,&
}
\]
such that the induced map $x\to y\oplus z$ is a monomorphism. Functorial resolutions adapted to the pushout are hence given by any of the spans
\[
\xymatrix{
x\ar[r]\ar[d]&y && x\ar[r]\ar[d]&Cx\oplus y && x\ar[r]\ar[d]&Cx\oplus y\\
Cx\oplus z,& && z,& && Cx\oplus z,& 
}
\]
together with the obvious maps collapsing the cones. Choosing for instance the resolution of the left, the derived pushout functor
\[
L\colim_\ulcorner\colon D(\cA^\ulcorner)\to D(\cA)
\]
is induced by applying $\colim_\ulcorner$ to these resolution. The canonical natural transformation
\[
\xymatrix{
\nCh(\cA^\ulcorner)\ar[r]^-{\colim_\ulcorner}\ar[d]_-\gamma&\nCh(\cA)\ar[d]^-\gamma\\
D(\cA^\ulcorner)\ar[r]_-{L\colim_\ulcorner}&D(\cA)\ultwocell\omit{\varepsilon}
}
\]
again comes from the chain maps which collapse the cones. (For a more systematic approach see \cite[\S7]{groth:thy-of-der}.)
\end{proof}

\subsection{Derived limits and derived Kan extensions}
\label{subsec:der-limit}

In this short subsection we collect key formal properties of the calculus of derived limits in categories of modules. More generally, restriction functors between derived diagram categories have adjoints on both sides. These derived Kan extension functors can be calculated in terms of derived (co)limits. It was the philosophy of Grothendieck \cite[pp.~196-200]{grothendieck:stacks} that the ``entire triangulated information'' of a derived category is encoded by this system of derived diagram categories and derived limit functors.

\begin{notn}\label{notn:diagonal}
For every small category $B$, we denote the derived category of the diagram category $\Mod{R}^B$ by
\[
\D_R(B)=D(\Mod{R}^B).
\]
As a special case $\D_R(\ast)$ is simply the derived category $D(R)$. The diagonal functor $\Delta_B\colon\Mod{R}\to\Mod{R}^B$ is exact and hence induces a similar diagonal functor
\[
\Delta_B\colon\D_R(\ast)\to\D_R(B).
\] 
\end{notn}

Recall that (co)limits are simply adjoints to diagonal functors \cite[\S~IV.2]{maclane}. Generalizing \autoref{prop:der-push} there is the following derived version of this.

\begin{thm}\label{thm:der-(co)lim}
For every small category $B$ there are adjunctions
\[
(L\colim_B,\Delta_B)\colon\D_R(B)\rightleftarrows\D_R(\ast)\quad\text{and}\quad
(\Delta_B,R\mathrm{lim}_B)\colon\D_R(\ast)\rightleftarrows\D_R(B).
\]
\end{thm}
\begin{proof}
This is part of the statement of \autoref{thm:grothendieck}.
\end{proof}

In order to build towards a further generalization, we recall the following definition. Given a functor $u\colon A\to B$ between small categories, there is the \emph{restriction} or \emph{precomposition functor}
\begin{equation}\label{eq:restriction-underived}
u^\ast\colon\Mod{R}^B\to\Mod{R}^A\colon X\mapsto X\circ u.
\end{equation}
In particular, for an object $b\in B$, we obtain the \emph{evaluation functor} 
\[
b^\ast\colon\Mod{R}^B\to\Mod{R}.
\]
In this case we simplify notation by writing $f_b\colon X_b\to Y_b$ for the value of $f\colon X\to Y$ under the evaluation functor.

\emph{Kan extensions} are some kind of relative versions of (co)limits. To explain what we mean by this it suffices to identify diagonal functors as restriction functors. In more detail, for every small category $B$ let $\pi_B\colon B\to\ast$ be the unique functor to the terminal category. Note that under the isomorphism $\Mod{R}^\ast\cong\Mod{R}$ the functors $\Delta_B$ and $(\pi_B)^\ast$ are identified. Thus, to put this differently, the (co)limit functors are adjoints to restriction along these particular functors. It turns out that the existence of limits and colimits in $\Mod{R}$ also guarantees the existence of adjoints to arbitrary restriction functors \eqref{eq:restriction-underived}. Given a functor $u\colon A\to B$ between small categories, there are adjunctions
\[
(u_!,u^\ast)\colon\Mod{R}^A\rightleftarrows\Mod{R}^B\quad\text{and}\quad(u^\ast,u_\ast)\colon\Mod{R}^B\rightleftarrows\Mod{R}^A.
\]
The functor $u_!$ is the \textbf{left Kan extension} functor along $u$, and $u_\ast$ is referred to as the \textbf{right Kan extension} along $u$. For the basic theory of these Kan extension functors see \cite[\S X]{maclane} or \cite[\S3.7]{borceux:2}. For a discussed geared towards our current situation we also refer to \cite[\S 6]{groth:thy-of-der}. 

The restriction functors \eqref{eq:restriction-underived} are again exact. Hence, for trivial reasons we obtain derived restriction functors
\[
u^\ast\colon\D_R(B)\to\D_R(A).
\]
As a relative version of \autoref{thm:der-(co)lim}, also the Kan extension adjunctions have derived versions.

\begin{thm}\label{thm:der-Kan}
For every functor $u\colon A\to B$ between small categories there are adjunctions
\[
(u_!,u^\ast)\colon\D_R(A)\rightleftarrows\D_R(B)\quad\text{and}\quad
(u^\ast,u_\ast)\colon\D_R(B)\rightleftarrows\D_R(A).
\]
\end{thm}
\begin{proof}
This is part of the statement of \autoref{thm:grothendieck}.
\end{proof}

For every functor $u\colon A\to B$ we hence have a derived left Kan extension functors and a derived right Kan extension functor
\[
u_!\colon\D_R(A)\to\D_R(B)\quad\text{and}\quad u_\ast\colon\D_R(A)\to\D_R(B).
\]
To be of any use, it is important to be able to calculate these derived Kan extension functors. More specifically, given $X\in\D_R(A)$ and $b\in B$ we would like to be able to express the values
\[
u_!(X)_b, u_\ast(X)_b\in\D_R(\ast)\cong D(R)
\]
of the derived Kan extensions $u_!(X),u_\ast(X)\in\D_R(B)$ in terms of $X$, $u$, and $b$ only. Before dealing with this derived situation, let us recall the corresponding fact for module categories.

\begin{con}\label{con:ptws-kan}
Let $u\colon A\to B$ be a functor between small categories and let $X\colon A\to\Mod{R}$ be a diagram of modules. In order to construct the left Kan extension $u_!(X)\colon B\to\Mod{R}$ we recall the definition of \textbf{slice categories}. For every object $b\in B$ the slice category $(u/b)$ of objects $u$-over $b$ has the following description. Objects in $(u/b)$ are pairs $(a,f\colon ua\to b)$ consisting of an object $a\in A$ and a morphism $f\colon ua\to b$ in $B$. A morphism $g\colon(a,f)\to(a',f')$ in $(u/b)$ is a morphism $g\colon a\to a'$ in $A$ such that $f'\circ u(g)=f$. Note that there is forgetful functor $p\colon (u/b)\to A\colon (a,f)\mapsto a$, and to construct $u_!(X)_b$ it suffices to form a colimit over this slice category. More precisely, for every $b\in B$ we define
\[
u_!(X)_b=\colim_{(u/b)} X\circ p \in\Mod{R}.
\]
One checks that this defines a diagram $u_!(X)\colon B\to\Mod{R}$, the left Kan extension of $X$ along $u$. 

Dually, for every $b\in B$ there is the slice category $(b/u)$ of objects $u$-under~$b$. And also this category comes with a forgetful functor $q\colon (b/u)\to A\colon (a,f\colon b\to ua)\mapsto a$. It turns out that also the right Kan extension $u_\ast(X)$ can be constructed pointwise by the formula
\[
u_\ast(X)_b=\mathrm{lim}_{(b/u)} X\circ q\in\Mod{R}.
\]
\end{con} 

This concludes the recap of the pointwise formulas for Kan extensions of diagrams of modules (which, of course, generalizes to arbitrary complete and cocomplete categories). The following theorem shows that similar formulas extend to the derived Kan extensions in \autoref{thm:der-Kan}. For convenience, we also summarize some additional key properties of the calculus of derived limits, derived colimits, and derived Kan extensions over a ring. 

\begin{thm}[Grothendieck]\label{thm:grothendieck}
Let $R$ be a ring. The formation of derived categories of diagram categories defines a $2$-functor
\[
\D_R\colon B\mapsto\D_R(B)= D(\Mod{R}^B)
\]
from small categories to not necessarily small categories. The $2$-functor $\D_R$ enjoys the following properties.
\begin{enumerate}
\item[(Der1)] The canonical inclusion functors $B_j\to\coprod_{i\in I} B_i, j\in I,$ induce an equivalence of categories
  \[
  \D_R(\coprod_{i\in I} B_i)\toiso\prod_{i\in I}\D_R(B_i).
  \]
\item[(Der2)] A morphism $f\colon X\to Y$ in $\D_R(B)$ is an isomorphism if and only if the morphisms $f_b\colon X_b\to Y_b, b\in B,$ are isomorphisms in $\D_R(\ast)=D(R).$
\item [(Der3)] For every functor $u\colon A\to B$, the restriction functor $u^\ast\colon\D_R(B)\to\D_R(A)$ has a left adjoint $u_!$ and a right adjoint $u_\ast$,
\begin{equation}\label{eq:der-kan-adj}
(u_!,u^\ast)\colon\D_R(A)\rightleftarrows\D_R(B),\qquad (u^\ast,u_\ast)\colon\D_R(B)\rightleftarrows\D_R(A).
\end{equation}
\item[(Der4)] For every functor $u\colon A\to B$, the functors $u_!,u_\ast\colon\D_R(A)\to\D_R(B)$ can be calculated pointwise. More specifically, for $X\in\D_R(A)$ and $b\in B$ there are canonical isomorphisms
\[
L\colim_{(u/b)}p^\ast (X)\toiso u_!(X)_b\quad\text{and}\quad u_\ast(X)_b\toiso R\mathrm{lim}_{(b/u)}q^\ast (X).
\]
\end{enumerate}
\end{thm}
\begin{proof}
While an elementary proof building on classical homological algebra only does not seem to exist in the literature, we think that it would be nice to expand on the techniques behind \autoref{prop:C-cok} and \autoref{prop:der-push} in order to obtain such a proof. Currently, this result is a consequence of more general theorems. The category $\Mod{R}$ is Grothendieck abelian and $\nCh(R)$ can hence be endowed with a combinatorial Quillen model structure (see \cite[Prop.~3.13]{beke:sheafifiable}, \cite[\S2]{hovey:sheaves}, or \cite[\S\S2-3]{cisinski-deglise:local-stable}). For combinatorial Quillen model categories a reasonably simple proof of a variant of this theorem can be found in \cite[\S1.3]{groth:ptstab}. 
\end{proof}

\begin{rmk}
We want to stress that, by definition, $\D_R$ sends a small category $B$ to the derived category $D(\Mod{R}^B)$. In this assignment the derived category is considered as a plain category only and hence not as a triangulated category. In fact, the main goal of \S\ref{subsec:BP} and \S\ref{subsec:octa} now consists of showing how to construct the Verdier triangulation on $D(R)$ (\autoref{thm:verdier}) using certain nice properties of this $2$-functor $\D_R$ only. Those arguments will indicate the above-mentioned intimate relation between the triangulation and symmetries.

Formally, in the remainder of this section we focus on modules over a ring. We want to stress, however, that variants of \autoref{thm:grothendieck} arise in many other situations as well. Moreover, many of the arguments which are sketched further below apply almost verbatim to those more general situations (see \S\ref{sec:crash}).
\end{rmk}

The following key property will be used over and over again. 

\begin{prop}\label{prop:Kan-ff-D_R}
Let $R$ be a ring and let $u\colon A\to B$ a fully faithful functor. The Kan extension functors $u_!,u_\ast\colon\D_R(A)\to\D_R(B)$ are fully faithful.
\end{prop}
\begin{proof}
This is a special case of \autoref{prop:Kan-ff}.
\end{proof}

This result is of central importance in many constructions, and to explain its relevance we collect the following reformulation (based on standard results on adjunctions): For every fully faithful $u\colon A\to B$ and every $X\in\D_R(A)$ the adjunction units and counits
\[
\eta\colon X\toiso u^\ast u_! X\qquad\text{and}\qquad \varepsilon\colon u^\ast u_\ast X\toiso X
\]
are invertible. Thus, this result allows us to extend diagrams defined on smaller categories to larger ones without affecting the diagram on the smaller category, and in this sense \emph{Kan extensions} are proper ``extensions''.

As a first example of the derived Kan extension functors from \autoref{thm:grothendieck}, we reformulate the proof of \autoref{prop:C-cok} in terms of these functors. This needs some preparation.

\begin{notn}
We denote by $\square$ the category $[1]\times[1]$,
\[
\xymatrix{
(0,0)\ar[r]\ar[d]&(0,1)\ar[d]\\
(1,0)\ar[r]&(1,1).
}
\]
For every ring $R$, a \textbf{square} in $\D_R$ is an object $X\in\D_R(\square)$. Such a square is hence a commutative square of chain complexes
\begin{equation}\label{eq:square}
\vcenter{
\xymatrix{
x\ar[r]\ar[d]&y\ar[d]\\
z\ar[r]&w
}
}
\end{equation}
considered as an object in the derived category $D(\Mod{R}^\square).$ Associated to every square, there are the following canonical maps.
\begin{enumerate}
\item Ignoring the chain complex $w$ for a moment, there is the corresponding left derived pushout $(L\colim_{\ulcorner})(X\!\!\mid_{\ulcorner})$ (\autoref{prop:der-push}). This chain complex comes with a canonical comparison map
\begin{equation}\label{eq:cocart}
(L\colim_{\ulcorner})(X\!\!\mid_{\ulcorner})\to w.
\end{equation}
\item Dually, ignoring the chain complex $x$, there is the right derived pullback $R\mathrm{lim}_\lrcorner(X\!\!\mid_{\lrcorner})$, and this chain complex comes with a canonical comparison map
\begin{equation}\label{eq:cart}
x\to (R\mathrm{lim}_{\lrcorner})(X\!\!\mid_{\lrcorner}).
\end{equation}
\end{enumerate}
\end{notn}

\begin{defn}\label{defn:cocart-D_R}
Let $R$ be a ring and let $X\in\D_R(\square)$.
\begin{enumerate}
\item The square $X$ is \textbf{cocartesian} if \eqref{eq:cocart} is an isomorphism in $\D_R(\ast)=D(R)$.
\item The square $X$ is \textbf{cartesian} if \eqref{eq:cart} is an isomorphism in $\D_R(\ast)=D(R)$.
\end{enumerate}
\end{defn}
 
The following two properties of $\D_R$ are crucial to all later constructions.
 
\begin{prop}[Stability of $\D_R$]\label{prop:DR-stable}
Let $R$ be a ring.
\begin{enumerate}
\item The category $\D_R(\ast)=D(R)$ has a zero object.
\item A square $X\in\D_R(\square)$ is cocartesian if and only if it is cartesian. We refer to these squares as \textbf{bicartesian squares}.
\end{enumerate}
\end{prop}
\begin{proof}
The first statement is obvious since $D(R)$ is an additive category. We invite the reader to come up with a direct proof of the second statement. That result also follows more indirectly from the fact that $\Sigma\colon D(R)\to D(R)$ is an equivalence combined with \autoref{thm:grothendieck} and \autoref{thm:stable}.
\end{proof}

Whenever we want to stress that a given square is bicartesian, we decorate it by a square in the middle such as in
\[
\xymatrix{
x\ar[r]\ar[d]\ar@{}[rd]|{\square}&y\ar[d]\\
z\ar[r]&w.
}
\]

\begin{rmk}
On a first view, it seems to be a bit strange, that in \autoref{defn:cocart-D_R} we introduce the notions of cocartesian and cartesian squares in $\D_R$ for every ring~$R$, in order to then observe in \autoref{prop:DR-stable} that these two classes coincide. The point is, of course, that these notions make sense in many other situations as well (see \S\ref{sec:crash}), and that often these classes do not agree. For instance, in a complete and cocomplete category $\cC$ with a zero object, the classes of pushout and pullback squares agree if and only if $\cC$ is equivalent to the terminal category (which is equivalent to $\cC$ being a contractible groupoid). In this precise sense the conjunction of the two properties in \autoref{prop:DR-stable} are invisible to ordinary category theory and they require more refined techniques (see also \autoref{rmk:stable-invisible}).
\end{rmk}

With these basic notions in place, we can now revisit \autoref{prop:C-cok}.
 
\begin{defn}\label{defn:cof-sq-D_R}
Let $R$ be a ring. A \textbf{cofiber square} in $\D_R$ is a cocartesian square $X\in\D_R(\square)$ which vanishes at the lower left corner,
\begin{equation}\label{eq:cof-sq}
\vcenter{
\xymatrix{
x\ar[r]^-f\ar[d]\ar@{}[rd]|{\square}&y\ar[d]\\
0\ar[r]&z.
}
}
\end{equation}
We denote by $\D_R(\square)^\cof\subseteq\D_R(\square)$ the full subcategory spanned by the cofiber squares.\end{defn}

We want to stress that the lower left corner is a zero object in $D(R)$, so it really could be any acyclic chain complex. The intuition is that a cofiber square is essentially determined by the upper horizontal morphism $
(f\colon x\to y)$. This morphism is obtained formally from the cofiber square by restriction along
\begin{equation}\label{eq:hor-mor}
k\colon[1]\to\square\colon i\mapsto (0,i).
\end{equation}

\begin{prop}\label{prop:cof-sq-D_R}
For every ring $R$ restriction along \eqref{eq:hor-mor} induces an equivalence of categories
\[
k^\ast\colon\D_R(\square)^\cof\toiso\D_R([1]).
\]
\end{prop}
\begin{proof}
We sketch the proof and begin by noting that the functor \eqref{eq:hor-mor} factors as
\[
k=j\circ i\colon[1]\to\ulcorner\to\square.
\]
Here, $\ulcorner$ is again suggestive notation for the poset \eqref{eq:span} which corepresents spans. By \autoref{thm:grothendieck} there are derived Kan extension adjunctions
\[
(i^\ast,i_\ast)\colon\D_R(\ulcorner)\rightleftarrows\D_R([1])\qquad\text{and}\qquad (j_!,j^\ast)\colon\D_R(\ulcorner)\rightleftarrows\D_R(\square).
\]
And in this case the construction of $i_\ast$ and $j_!$ can be given very explicitly.
\begin{enumerate}
\item Before passing to derived categories, it is straightforward to check that the functor $i_\ast$ which sends a chain map $(x\to y)$ to the span $(0\ot x\to y)$ is right adjoint to $i^\ast$ (alternatively, this also follows from the pointwise formulas in \autoref{con:ptws-kan}). Since both functors are exact, this immediately yields the first adjunction. 
\item In the case of the second adjunction a direct argument can be sketched as follows. While pushouts fail to be exact on all spans, they are exact on spans such that at least one of the maps is a monomorphism. Hence, as in the proof of \autoref{prop:der-push}, we can resolve an arbitrary span by adding inclusions into cones on at least one of the edges, and the corresponding pushout squares are models for the left derived pushout square functor
\[
j_!\colon\D_R(\ulcorner)\to\D_R(\square).
\]
\end{enumerate}

With these adjunctions at our disposal we can now consider the composition
\[
j_!\circ i_\ast\colon\D_R([1])\to\D_R(\ulcorner)\to\D_R(\square),
\]
and this yields the desired equivalence. In more detail, the functors $i_\ast$ and $j_!$ are both fully faithful (\autoref{prop:Kan-ff-D_R}), and one checks that the essential image of their composition consists precisely of the cofiber squares. 
\end{proof}

The above sketch has shown that the left quadrilateral in \eqref{eq:C-cok} can be constructed in terms of derived Kan extension functors. This justifies the terminology \emph{cofiber square} since for every such square \eqref{eq:cof-sq} there is a canonical isomorphism between $z$ and the cone or cofiber of $f$,
\[
Cf\toiso z.
\]

There is the following important special case of cofiber squares.

\begin{defn}\label{defn:susp-sq-D_R}
Let $R$ be a ring. A \textbf{suspension square} in $\D_R$ is a cocartesian square which vanishes at the lower left and the upper right corner,
\begin{equation}\label{eq:susp-sq}
\vcenter{
\xymatrix{
x\ar[r]\ar[d]\ar@{}[rd]|{\square}&0\ar[d]\\
0\ar[r]&x'.
}
}
\end{equation}
\end{defn}

\begin{rmk}\label{rmk:susp-sq-D_R}
For every ring $R$ we denote by $\D_R(\square)^\Sigma\subseteq\D_R(\square)$ the full subcategory spanned by the suspension squares. Similar to the case of \autoref{prop:cof-sq-D_R}, evaluation at $(0,0)\in\square$ induces an equivalence of categories
\[
(0,0)^\ast\colon \D_R(\square)^\Sigma\toiso \D_R(\ast)=D(R).
\]
We invite the reader to adapt the arguments of the previous case in order to sketch a proof of this result. An inverse of this equivalence is given by forming the cofiber square associated to $(x\to 0)$. Hence, the proof of \autoref{prop:C-cok} shows that we simply form the cokernel of the inclusion $x\to Cx$, which is clearly isomorphic to the usual suspension $\Sigma x$ of the chain complex $x$. The terminology \emph{suspension squares} is hence justified by the fact that every such square \eqref{eq:susp-sq} induces a canonical isomorphism
\[
\Sigma x\toiso x'.
\]
\end{rmk}

\subsection{$A_2$-quivers and Barratt--Puppe sequences}
\label{subsec:BP}

The point of the constructions at the end of the previous subsection was the following. A morphism in $\D_R$ can be equivalently encoded by a cofiber square (\autoref{prop:cof-sq-D_R}) and, similarly, an object in $\D_R$ is as good as a suspension square (\autoref{rmk:susp-sq-D_R}). In this subsection we want to expand a bit on this observation for the case of morphisms. Representations of $\A{2}$ in $\D_R$ turn out to be equivalent to \emph{Barratt--Puppe sequences}. Moreover, the symmetry of these Barratt--Puppe sequence corresponds to the rotation at the level of triangulated categories. 

As a warmup we consider \emph{cofiber sequences} which are essentially obtained by forming two cofiber squares. To carry this out in detail, we introduce the following notation.

\begin{notn}
For every natural number $n$, we denote by $[n]$ the finite linear order $(0<1<\ldots <n)$. The product $\boxbar=[1]\times[2]$ hence agrees with the shape of two adjacent commuting squares
\[
\xymatrix{
(0,0)\ar[r]\ar[d]&(0,1)\ar[r]\ar[d]&(0,2)\ar[d]\\
(1,0)\ar[r]&(1,1)\ar[r]&(1,2).
}
\]
\end{notn}

In representation theoretic parlance, the poset $[n]$ is the linearly oriented $A_{n+1}$-quiver
\[
\A{n+1}=(1\to\ldots\to n+1).
\]
Note that the labeling conventions only match up to a shift by one. There is a unique isomorphism between the two shapes which will always be used implicitly, hopefully not arising in a confusion concerning the various labels.

\begin{defn}\label{defn:cof-seq}
Let $R$ be a ring. A \textbf{cofiber sequence} in $\D_R$ is an object $X\in\D_R(\boxbar)$ such that the square on the left and the square on the right are cocartesian and such that $X$ vanishes at $(1,0)$ and $(0,2)$,
\begin{equation}\label{eq:cof-seq}
\vcenter{
\xymatrix{
x\ar[r]^-f\ar[d]\ar@{}[rd]|{\square}&y\ar[d]_-g\ar[r]\ar@{}[rd]|{\square}&0\ar[d]\\
0\ar[r]&z\ar[r]&x'.
}
}
\end{equation}
We denote by $\D_R(\boxbar)^\cof\subseteq\D_R(\square)$ the full subcategory spanned by the cofiber squares.
\end{defn}

Morally, a cofiber sequence is as good as a morphism only. In fact, given a cofiber sequence \eqref{eq:cof-seq}, the morphism $z\to x'$ can be reconstructed from $g\colon y\to z$ by \autoref{prop:cof-sq-D_R}. Similarly, by that same result also the morphism $g\colon y\to z$ is determined by $f\colon x\to y$. This makes the following result plausible in which we denote the inclusion of the upper left morphism by
\begin{equation}\label{eq:hor-mor-cof-seq}
k\colon[1]\to\boxbar\colon i\mapsto (0,i).
\end{equation}

\begin{prop}\label{prop:cof-seq}
For every ring $R$ restriction along \eqref{eq:hor-mor-cof-seq} induces an equivalence of categories
\[
k^\ast\colon\D_R(\boxbar)^\cof\toiso\D_R([1]).
\]
\end{prop}
\begin{proof}
The strategy is very similar to the one in the earlier cases. The functor $k$ in \eqref{eq:hor-mor-cof-seq} factors as a composition
\[
j\circ i\colon [1]\to B\to[1]\times[2]
\]
of fully faithful inclusions where $B$ is given by
\[
\xymatrix{
(0,0)\ar[r]\ar[d]&(0,1)\ar[r]&(0,2)\\
(1,0).&&
}
\]
At the level of chain complexes it is straightforward to check that the restriction functor $i^\ast$ has a right adjoint $i_\ast$ which simply adds two zero chain complexes and the unique chain maps. Alternatively, this follows immediately from \autoref{con:ptws-kan} since the corresponding slice categories are empty. Thus, both $i^\ast$ and $i_\ast$ are exact, and the induced functor $i_\ast\colon\D_R([1])\to\D_R(B)$ is hence under control. It is fully faithful (\autoref{prop:Kan-ff-D_R}) and the essential image consists of all $X\in\D_R(B)$ which vanish at $(1,0)$ and $(0,2)$. Similarly, at the level of chain complexes a left adjoint $j_!$ to $j^\ast$ is given by the functor which adds two pushout squares. The corresponding left derived functor $j_!$ hence adds two derived pushout squares, and one checks that the essential image of this fully faithful functor can be characterized by this. As an upshot, both functors in 
\[
j_!\circ i_\ast\colon\D_R([1])\to\D_R(B)\to\D_R(\boxbar)
\]
are fully faithful and the essential image of their composition consists precisely of the cofiber sequences.
\end{proof}

In order to analyze cofiber sequences in more detail, we recall the following pasting and cancellation property of bicartesian squares in $\D_R$. There are three obvious inclusions of the square $\square$ in $\boxbar$. Let us denote by
\[
\iota_{01}\colon\square\to\boxbar,\quad \iota_{12}\colon\square\to\boxbar,\quad\text{and}\quad \iota_{02}\colon\square\to\boxbar
\]
the inclusion of the left square, the right square, and the composed square, respectively. 

\begin{prop}\label{prop:2-out-of-3-D_R}
Let $R$ be a ring and let $X\in\D_R(\boxbar)$. If two of the squares $\iota_{01}^\ast(X),$ $\iota_{12}^\ast(X),$ and $\iota_{02}^\ast(X)$ are bicartesian, then so is the third square.
\end{prop}
\begin{proof}
This is a special case of \cite[Prop.~4.6]{groth:ptstab}.
\end{proof}

\begin{con}\label{con:dist-tria}
Let $R$ be a ring and let $X\in\D_R(\boxbar)^\cof$ be a cofiber sequence looking like \eqref{eq:cof-seq}. By \autoref{prop:2-out-of-3-D_R} the composed square $\iota_{02}^\ast(X)$ is again bicartesian and it takes the form
\[
\xymatrix{
x\ar[r]\ar[d]\ar@{}[rd]|{\square}&0\ar[d]\\
0\ar[r]&x'.
}
\]
Thus, $\iota_{02}^\ast(X)$ is a suspension square (\autoref{defn:susp-sq-D_R}), and by \autoref{rmk:susp-sq-D_R} we obtain a canonical isomorphism $\Sigma x\toiso x'$. It follows that associated to every cofiber sequence there is a triangle
\[
x\stackrel{f}{\to} y\stackrel{g}{\to} z\stackrel{h}{\to} \Sigma x.
\]
This is an ordinary diagram in $\D_R(\ast)=D(R)$, and the morphism $h$ is obtained from the unlabeled morphism $z\to x'$ in \eqref{eq:cof-seq} thanks to the isomorphism $\Sigma x\toiso x'$.
\end{con}

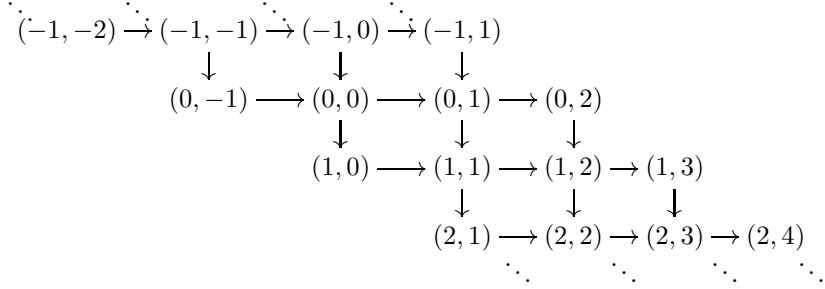
\begin{figure}
\centering
\[
\xymatrix@=1.0em{
\ar@{}[dr]|{\ddots}&\ar@{}[dr]|{\ddots}&\ar@{}[dr]|{\ddots}&\ar@{}[dr]|{\ddots}&&&&&&\\
&(-1,-2)\ar[r]&(-1,-1)\ar[d]\ar[r]&(-1,0)\ar[r]\ar[d]&(-1,1)\ar[d]&&&&&\\
&&(0,-1)\ar[r]&(0,0)\ar[r]\ar[d]&(0,1)\ar[r]\ar[d]&(0,2)\ar[d]&&&&\\
&&&(1,0)\ar[r]&(1,1)\ar[r]\ar[d]&(1,2)\ar[r]\ar[d]&(1,3)\ar[d]&&&\\
&&&&(2,1)\ar[r]\ar@{}[dr]|{\ddots}&(2,2)\ar[r]\ar@{}[dr]|{\ddots}&(2,3)\ar[r]\ar@{}[dr]|{\ddots}&(2,4)\ar@{}[dr]|{\ddots}&&\\
&&&&& & & & &
}
\]
\caption{The full subposet $M_2\subseteq\lZ\times\lZ$}
\label{fig:M2-shape}
\end{figure}

This already suggests that cofiber sequences are closely related to distinguished triangles, and we will come back to this at the end of \S\ref{subsec:octa}. Here, instead, we summarize some of the constructions carried out so far. For every ring~$R$, at the level of derived categories of diagram categories there are equivalences of categories
\[
\D_R(\boxbar)^\cof\toiso\D_R(\square)^\cof\toiso\D_R([1]).
\]
These equivalences are induced by suitable restriction functors, and inverse equivalences are obtained by extending a morphism to a cofiber square or by iterating this construction twice. We now simply iterate this construction countably many times in both directions. The combinatorial details are as follows.

\begin{notn}
Let $M_2\subseteq\lZ\times\lZ$ be the full subposet given by all $(i,j)\in\lZ\times\lZ$ such that $i-1\leq j\leq i+2$. 
\end{notn}

A part of this poset is displayed in \autoref{fig:M2-shape}. The reason that we consider this poset is that, for every ring $R$, this shape allows us to simultaneously encode a morphism $(f\colon x\to y)\in\D_R([1])$, all its iterated derived cokernels, and also all its iterated derived kernels. This is achieved by forming \emph{Barratt--Puppe sequences} in the following sense.

\begin{defn}
Let $R$ be a ring. A diagram $X\in\D_R(M_2)$ is a \textbf{Barratt--Puppe sequence} in $\D_R$ if the diagram $X$ vanishes on the two boundary stripes and all squares in $X$ are bicartesian. We denote by $\D_R(M_2)^\ex\subseteq\D_R(M_2)$ the full subcategory spanned by all Barratt--Puppe sequences.
\end{defn}

The terminology is motivated by similar constructions with pointed topological spaces. In that case this notion yields the classical Barratt--Puppe sequences as in \cite{puppe:induzierten-I}, which provide the space-level origin of many long exact sequences in algebraic topology. The following result makes precise that Barratt--Puppe sequences are determined by their restrictions along the inclusion
\begin{equation}\label{eq:i-M2}
i\colon[1]\to M_2\colon k\mapsto (0,k).
\end{equation}
 
\begin{thm}\label{thm:BP}
For every ring $R$ restriction along \eqref{eq:i-M2} induces an equivalence of categories
\[
i^\ast\colon\D_R(M_2)^\ex\toiso\D_R([1]).
\]
\end{thm}
\begin{proof}
We only sketch the arguments and refer the reader to \cite[\S4]{gst:Dynkin-A} for a detailed proof of a more general result. The idea is to construct an inverse equivalence
\[
F_{[1]}\colon\D_R([1])\toiso\D_R(M_2)^\ex
\]
as a composition of suitable derived Kan extensions. To this end we observe that the inclusion \eqref{eq:i-M2} factors as
\[
i\colon [1]\stackrel{i_1}{\to}K_1\stackrel{i_2}{\to}K_2\stackrel{i_3}{\to}K_3\stackrel{i_4}{\to}M_2,
\]
where the $K_j$ are the following full subposets of $M_2$ and the $i_j$ are the obvious inclusions.
\begin{enumerate}
\item $K_1$ is obtained from the image of $i$ by adding those objects on the boundary stripe which sit under the image of $i$, i.e., the objects $(n,n+2), n\geq 0,$ and the objects $(n,n-1),n>0.$
\item $K_2$ contains besides the objects in $K_1$ also all remaining objects under the image of $i$, i.e., the objects $(n,n)$ and $(n,n+1)$ for $n>0.$
\item $K_3$ is obtained from $K_2$ by also adding the remaining objects on the boundary stripes, i.e., the objects $(n,n+2), n< 0,$ and the objects $(n,n-1),n\leq 0.$
\end{enumerate}
We know from \autoref{prop:Kan-ff-D_R} that the following four functors
\[
\D_R([1])\stackrel{(i_1)_\ast}{\to}\D_R(K_1)\stackrel{(i_2)_!}{\to}\D_R(K_2)\stackrel{(i_3)_!}{\to}\D_R(K_3)\stackrel{(i_4)_\ast}{\to}\D_R(M_2)
\]
are fully faithful. Invoking arguments similar to the previous cases in combination with homotopy (co)finality arguments, one checks that  
\begin{enumerate}
\item $(i_1)_\ast\colon \D_R([1])\to\D_R(K_1)$ simply adds zero objects at the new components,
\item $(i_2)_!\colon \D_R(K_1)\to\D_R(K_2)$ inductively adds bicartesian squares,
\item $(i_3)_!\colon \D_R(K_2)\to\D_R(K_3)$ again adds zero objects at the new components,
\item and that $(i_4)_\ast\colon \D_R(K_3)\to\D_R(M_2)$ constructs new bicartesian squares.
\end{enumerate}
Thus, the composition of these four functors indeed extends $(f\colon x\to y)\in\D_R([1])$ to a Barratt--Puppe sequence. Moreover, in each of these four steps the above describes precisely the corresponding essential image, and we hence conclude that
\[
F_{[1]}=(i_4)_\ast\circ (i_3)_!\circ (i_2)_!\circ (i_1)_\ast
\]
is the desired inverse equivalence of $i^\ast\colon\D_R(M_2)^\ex\to\D_R([1])$.
\end{proof}

\begin{figure}
\centering
\[
\xymatrix{
\ar@{}[dr]|{\ddots}&\ar@{}[dr]|{\ddots}&\ar@{}[dr]|{\ddots}&\ar@{}[dr]|{\ddots}&&&&&&\\
&0\ar[r]&\tilde y\ar[d]\ar[r]&\tilde z\ar[r]\ar[d]&0\ar[d]&&&&&\\
&&0\ar[r]&x\ar[r]^-f\ar[d]\ar@{}[dr]|{\fbox{1}}&y\ar[r]\ar[d]\ar@{}[dr]|{\fbox{2}}&0\ar[d]&&&&\\
&&&0\ar[r]&z\ar[r]\ar[d]\ar@{}[dr]|{\fbox{3}}&x'\ar[r]\ar[d]\ar@{}[dr]|{\fbox{4}}&0\ar[d]&&&\\
&&&&0\ar[r]\ar@{}[dr]|{\ddots}&y'\ar[r]\ar@{}[dr]|{\ddots}&z'\ar[r]\ar@{}[dr]|{\ddots}&0\ar@{}[dr]|{\ddots}&&\\
&&&&& & & & &
}
\]
\caption{Iterated derived (co)kernels of $(f\colon x\to y)\in\D_R([1])$}
\label{fig:BP}
\end{figure}
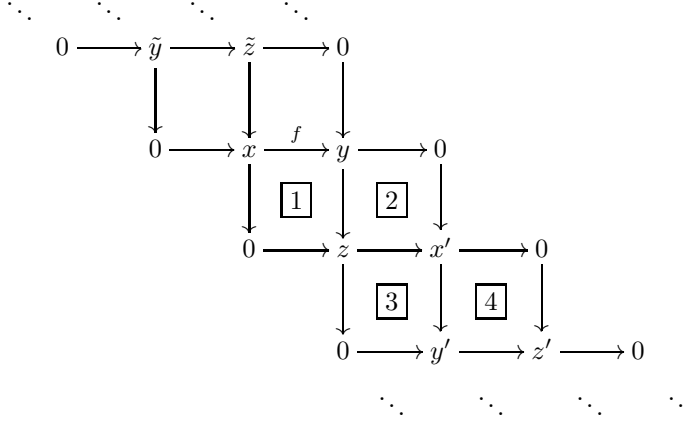

Let us discuss these Barratt--Puppe sequences to some extent.

\begin{rmk}\label{rmk:BP}
Let $R$ be a ring and let $X\in\D_R(M_2)^\ex$ be a Barratt--Puppe sequence looking like \autoref{fig:BP}. In that Barratt--Puppe sequence any two adjacent squares determine a cofiber sequence and hence give rise to a triangle in $\D_R(\ast)$ (\autoref{con:dist-tria}). In particular, the cofiber sequence determined by the squares $\fbox{1},\fbox{2}$ induces a triangle associated to $f\colon x\to y$. Similarly, the squares $\fbox{2},\fbox{3}$ determine a triangle associated to the morphism $g\colon y\to z$, and this is a rotated triangle of the previous one. As an upshot, the Barratt--Puppe sequence in \autoref{fig:BP} encodes all iterated rotations of triangles associated to $f\colon x\to y$. 

The defining exactness properties of Barratt--Puppe sequences (vanishing on boundary stripes and all squares are bicartesian) are invariant under the obvious symmetries of the shape $M_2$. Hence, as alluded to in the introduction, the rotation at the level of distinguished triangles is a shadow of certain symmetries on representations of $\A{2}=(1\to 2)\cong[1]$ with values in chain complexes. As a variant of this, let us consider the square $\fbox{1}$ in \autoref{fig:BP} as a ``triangular fundamental domain'' of the Barratt--Puppe sequence. The flip symmetry of the shape $M_2$ identifies the squares $\fbox{1}$ and $\fbox{4}$. At the level of diagrams this flip symmetry corresponds to the suspension, and \autoref{fig:BP} encodes the fundamental domain and all its iterated (de)suspensions. Of course, the suspension is simply the cube of the cofiber (\cite[Lem.~5.13]{gps:mayer}), and in that generality this can be interpreted as an example of the abstract fractionally Calabi--Yau property of $\A{n}$-quivers (\cite{keller:CY-triang},\cite[Cor.~5.20]{gst:Dynkin-A}).
\end{rmk}
 
\begin{rmk}\label{rmk:basic-techniques}
The strategy behind the constructions discussed so far is fairly typical and it relies on two ingredients. First, for every fully faithful functor $u\colon A\to B$ the derived Kan extension functors $u_!,u_\ast\colon\D_R(A)\to\D_R(B)$ are again fully faithful. In particular, they induce equivalences onto their essential images. Second, for suitable classes of fully faithful functors, the functors $u_!,u_\ast$ simply add zero chain complexes at the new objects. We refer the reader to \cite{groth:ptstab} for a more systematic discussion of this.
\end{rmk}

\subsection{$A_3$-quivers and refined octahedral diagrams}
\label{subsec:octa}
 
In this subsection, we conclude the unorthodox proof of \autoref{thm:verdier} and there are only two main steps left. First, we repeat to some extent the discussion in \S\ref{subsec:BP} but starting with representations of $\A{3}$ instead. Second, for suitable shapes the underlying diagram functors from \autoref{eg:dia-H} are full and essentially surjective, and this can be invoked to establish the existence of the Verdier triangulation. 
 
The octahedral axiom of triangulated categories asks for a version of the third Noether isomorphism for cones in triangulated categories. Hence, it is essentially a statement about the situation of two composable morphisms, which is to say of a representation of the linearly oriented $A_3$-quiver 
\[
\A{3}\cong[2]=(0<1<2).
\]
The following is a variant for $\A{3}$ of the above discussion of Barratt--Puppe sequences.

\begin{notn}
Let $M_3\subseteq\lZ\times\lZ$ be the full subposet given by all $(i,j)\in\lZ\times\lZ$ such that $i-1\leq j\leq i+3$. 
\end{notn}

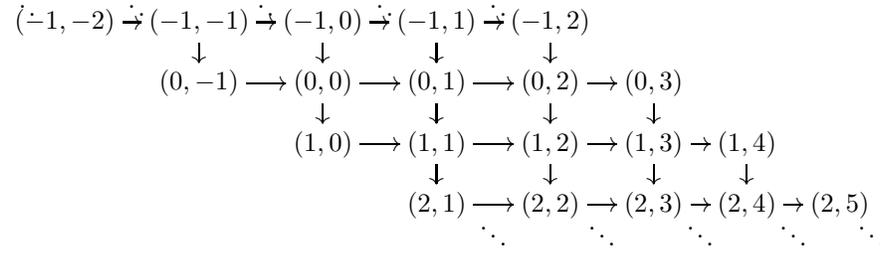
\begin{figure}
\centering
\[
\xymatrix@=0.7em{
\ar@{}[dr]|{\ddots}&\ar@{}[dr]|{\ddots}&\ar@{}[dr]|{\ddots}&\ar@{}[dr]|{\ddots}&\ar@{}[dr]|{\ddots}&&&&&\\
&(-1,-2)\ar[r]&(-1,-1)\ar[d]\ar[r]&(-1,0)\ar[r]\ar[d]&(-1,1)\ar[d]\ar[r]&(-1,2)\ar[d]&&&&\\
&&(0,-1)\ar[r]&(0,0)\ar[r]\ar[d]&(0,1)\ar[r]\ar[d]&(0,2)\ar[d]\ar[r]&(0,3)\ar[d]&&&\\
&&&(1,0)\ar[r]&(1,1)\ar[r]\ar[d]&(1,2)\ar[r]\ar[d]&(1,3)\ar[d]\ar[r]&(1,4)\ar[d]&&\\
&&&&(2,1)\ar[r]\ar@{}[dr]|{\ddots}&(2,2)\ar[r]\ar@{}[dr]|{\ddots}&(2,3)\ar[r]\ar@{}[dr]|{\ddots}&(2,4)\ar[r]\ar@{}[dr]|{\ddots}&(2,5)\ar@{}[dr]|{\ddots}&\\
&&&&& & & & &
}
\]
\caption{The full subposet $M_3\subseteq\lZ\times\lZ$}
\label{fig:M3-shape}
\end{figure}

A part of this poset is displayed in \autoref{fig:M3-shape}, and this poset is the shape of \emph{refined octahedral diagrams} in the following precise sense. 

\begin{defn}
Let $R$ be a ring. A diagram $X\in\D_R(M_3)$ is a \textbf{refined octahedral diagram} in $\D_R$ if the diagram $X$ vanishes on the two boundary stripes and all squares in $X$ are bicartesian. We denote by $\D_R(M_3)^\ex\subseteq\D_R(M_3)$ the full subcategory spanned by all refined octahedral diagrams.
\end{defn}

Similar to the case of Barratt--Puppe sequences, a refined octahedral diagram is determined by restriction along various embeddings of $\A{3}\cong[2]$ to $M_3$. To be specific, we consider the standard embedding
\begin{equation}\label{eq:i-M3}
i\colon[2]\to M_3\colon k\mapsto (0,k).
\end{equation}

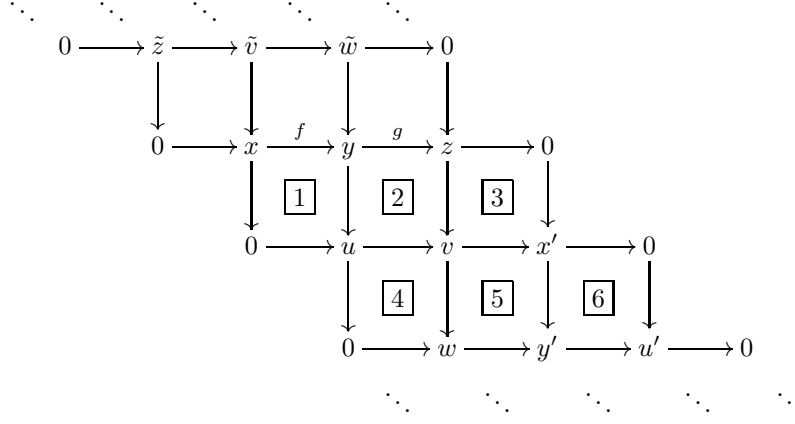
\begin{figure}
\centering
\[
\xymatrix{
\ar@{}[dr]|{\ddots}&\ar@{}[dr]|{\ddots}&\ar@{}[dr]|{\ddots}&\ar@{}[dr]|{\ddots}&\ar@{}[dr]|{\ddots}&&&&&\\
&0\ar[r]&\tilde z\ar[d]\ar[r]&\tilde v\ar[r]\ar[d]&\tilde w\ar[d]\ar[r]&0\ar[d]&&&&\\
&&0\ar[r]&x\ar[r]^-f\ar[d]\ar@{}[dr]|{\fbox{1}}&y\ar[r]^-g\ar[d]\ar@{}[dr]|{\fbox{2}}&z\ar[d]\ar[r]\ar@{}[dr]|{\fbox{3}}&0\ar[d]&&&\\
&&&0\ar[r]&u\ar[r]\ar[d]\ar@{}[dr]|{\fbox{4}}&v\ar[r]\ar[d]\ar@{}[dr]|{\fbox{5}}&x'\ar[d]\ar[r]\ar@{}[dr]|{\fbox{6}}&0\ar[d]&&\\
&&&&0\ar[r]\ar@{}[dr]|{\ddots}&w\ar[r]\ar@{}[dr]|{\ddots}&y'\ar[r]\ar@{}[dr]|{\ddots}&u'\ar[r]\ar@{}[dr]|{\ddots}&0\ar@{}[dr]|{\ddots}&\\
&&&&& & & & &
}
\]
\caption{The refined octahedral diagram of $(x\stackrel{f}{\to} y\stackrel{g}{\to} z)\in\D_R([2])$}
\label{fig:octa}
\end{figure}
 
\begin{thm}\label{thm:octa}
For every ring $R$ restriction along \eqref{eq:i-M3} induces an equivalence of categories
\[
i^\ast\colon\D_R(M_3)^\ex\toiso\D_R([2]).
\]
\end{thm}
\begin{proof}
It is straightforward to adapt the strategy of the proof of \autoref{thm:BP} in order to also cover this situation. For a detailed proof we refer to \cite[\S4]{gst:Dynkin-A}. For later reference we denote the inverse equivalence by $F_{[2]}\colon\D_R([2])\toiso\D_R(M_3)^\ex$.
\end{proof}

The connection to the octahedral axiom is as follows, and this justifyies the above terminology.

\begin{rmk}\label{rmk:octa}
Let $R$ be a ring and let $X\in\D_R(M_3)^\ex$ be a refined octahedral diagram looking like \autoref{fig:octa}. This diagram is determined by the restriction $i^\ast(X)\in\D_R([2])$, and this restriction in turn gives rise to the three morphisms
\[
f\colon x\to y,\quad g\colon y\to z,\quad\text{and}\quad g\circ f\colon x\to z.
\]
Given three such morphisms in $\D_R(\ast)$, let us recall that the octahedral axiom relates triangles of $f$, $g$, and $g\circ f$ by an additional triangle incorporating the various cones. These four triangles can be deduced from \autoref{fig:octa} as follows.
\begin{enumerate}
\item The square $\fbox{1}$ is a cofiber square, and there is hence a canonical isomorphism $Cf\toiso u$. Similarly, if we consider the square $\fbox{1}$ and the horizontal composition $\fbox{2}+\fbox{3}$, then we obtain by \autoref{prop:2-out-of-3-D_R} a cofiber sequence
\[
\xymatrix{
x\ar[r]^-f\ar[d]\ar@{}[rd]|{\square}&y\ar[d]_-g\ar[r]\ar@{}[rd]|{\square}&0\ar[d]\\
0\ar[r]&u\ar[r]&x'.
}
\]
This cofiber sequence of $f$ induces by \autoref{con:dist-tria} the first of the four triangles. 
\item Similarly, the vertical composition $\fbox{2}+\fbox{4}$ is by \autoref{prop:2-out-of-3-D_R} a cofiber square and we obtain the identification $Cg\toiso w$. Jointly with the vertical composition $\fbox{3}+\fbox{5}$ this yields the desired cofiber sequence of $g$.
\item As of the cofiber sequence of $g\circ f$, we note that the horizontal composition $\fbox{1}+\fbox{2}$ is a cofiber square for $g\circ f$, leading to $C(g\circ f)\toiso v$. In combination with the square $\fbox{3}$ we obtain the cofiber sequence of $g\circ f$.
\item Finally, the square $\fbox{4}$ and the horizontal composition $\fbox{5}+\fbox{6}$ define a cofiber sequence
\[
\xymatrix{
u\ar[r]\ar[d]\ar@{}[rd]|{\square}&v\ar[d]\ar[r]\ar@{}[rd]|{\square}&0\ar[d]\\
0\ar[r]&w\ar[r]&u'.
}
\]
Invoking the above identifications of $u,v,$ and $w$ as cones, this cofiber sequence induces the fourth triangle
\[
Cf\to C(g\circ f)\to Cg\to\Sigma Cf.
\]
\end{enumerate}
One checks directly that all the necessary compatibilities of the octahedral axiom are satisfied (see \cite[\S4]{groth:ptstab}).
\end{rmk}

For completeness, let us include a few comments on a variant of \autoref{thm:BP} and \autoref{thm:octa} for $\A{1}\cong\ast$.

\begin{rmk}
Let $R$ be a ring. We denote by $M_1\subseteq\lZ\times\lZ$ the full subposet given by all $(i,j)\in\lZ\times\lZ$ such that $i-1\leq j\leq i+1$.
\begin{figure}
\centering
\[
\xymatrix@=1.0em{
\ar@{}[dr]|{\ddots}&\ar@{}[dr]|{\ddots}&\ar@{}[dr]|{\ddots}& &&&&\\
&(-1,-2)\ar[r]&(-1,-1)\ar[d]\ar[r]&(-1,0)\ar[d]& &&&\\
&&(0,-1)\ar[r]&(0,0)\ar[r]\ar[d]&(0,1)\ar[d]& &&\\
&&&(1,0)\ar[r]&(1,1)\ar[r]\ar[d]&(1,2)\ar[d]& &\\
&&&&(2,1)\ar[r]\ar@{}[dr]|{\ddots}&(2,2)\ar[r]\ar@{}[dr]|{\ddots}&(2,3)\ar@{}[dr]|{\ddots}&\\
&&&&  & & &
}
\]
\caption{The full subposet $M_1\subseteq\lZ\times\lZ$}
\label{fig:M1-shape}
\end{figure}
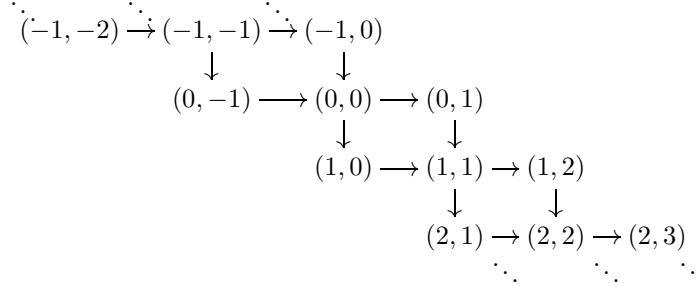
A part of this poset is displayed in \autoref{fig:M1-shape}. Moreover, let us write
\[
\D_R(M_1)^\ex\subseteq\D_R(M_1)
\]
for the full subcategory spanned by all diagrams which vanish on the boundary stripes and which make all squares bicartesian. Evaluation at $(0,0)\in M_1$ induces an equivalence of categories
\begin{equation}\label{eq:equiv-M1}
(0,0)^\ast\colon\D_R(M_1)^\ex\toiso D(R),
\end{equation}
as one checks by copying the proof of \autoref{thm:BP} (see \cite[\S4]{gst:Dynkin-A}). A typical diagram $X\in\D_R(M_1)^\ex$ is displayed in \autoref{fig:spectra}, and therein all squares are suspension squares (\autoref{rmk:susp-sq-D_R}).

The shape $M_1$ has two obvious generating symmetries, namely translation along the diagonal direction and the swap symmetry which keeps the diagonal fixed while it interchanges the objects $(n,n+1)$ and $(n+1,n)$ for all $n$. These symmetries of $M_1$ induce symmetries on $\D_R(M_1)^\ex$ since the defining exactness properties are clearly invariant under these symmetries. Interpreted at the level of $D(R)$ (by a conjugation with \eqref{eq:equiv-M1}), we obtain induced self-equivalences on $D(R)$. The first symmetry gives rise to the suspension equivalence $\Sigma\colon D(R)\toiso D(R)$, and the second symmetry can be identified with $-\id\colon D(R)\toiso D(R)$. One can make precise in which sense this latter symmetry is responsible for the sign in the rotation axiom and we will come back to this in \S\ref{subsec:stable} (see \autoref{rmk:loop-sign} and \autoref{rmk:add}).
\end{rmk}

\begin{figure}
\centering
\[
\xymatrix{
\ar@{}[dr]|{\ddots}&\ar@{}[dr]|{\ddots}&\ar@{}[dr]|{\ddots}&&&&&\\
&0\ar[r]&\tilde x\ar[d]\ar[r]\ar@{}[rd]|{\square}&0\ar[d]&&&&\\
&&0\ar[r]&x\ar[r]\ar[d]\ar@{}[rd]|{\square}&0\ar[d]&&&\\
&&&0\ar[r]&x'\ar[r]\ar[d]\ar@{}[rd]|{\square}&0\ar[d]&&\\
&&&&0\ar[r]\ar@{}[dr]|{\ddots}&x''\ar[r]\ar@{}[dr]|{\ddots}&0\ar@{}[dr]|{\ddots}&\\
&&&&&  & &
}
\]
\caption{A diagram $X\in\D_R(M_1)^\ex$}
\label{fig:spectra}
\end{figure}

In order to turn towards the alternative proof of \autoref{thm:verdier}, there is only one ingredient missing. Note that both in \autoref{thm:BP} and in \autoref{thm:octa} we gave alternative presentations of the derived categories of diagram categories
\[
\D_R([1])=D(\Mod{R}^{[1]})\qquad\text{and}\qquad \D_R([2])=D(\Mod{R}^{[2]}).
\]
In contrast to this, in the classical Verdier triangulation we consider morphisms or pairs of composable morphisms in $D(R)$, which is to say objects in the categories
\[
D(R)^{[1]}\qquad\text{and}\qquad D(R)^{[2]}.
\]
In order to connect these two perspectives, the following result proves crucial.

\begin{prop}\label{prop:D_R-strong}
For every ring $R$ and every $n\geq 0$ the underlying diagram functor
\[
\ndia_{[n]}\colon\D_R([n]))=D(\Mod{R}^{[n]})\to D(R)^{[n]}
\]
is essentially surjective and full.
\end{prop}
\begin{proof}
We invite the reader to check this directly. Alternatively, see for instance \cite[Prop.~2.15]{cisinski:derivable} or \cite[Thm.~10.3.3]{RB:ABC} for much more general statements.
\end{proof}

\begin{rmk}
Let $R$ be a ring.
\begin{enumerate}
\item The functor $\ndia_{[0]}$ is, of course, an equivalence of categories, but already in the case of $n=1$ the functor $\ndia_{[1]}$ is not faithful. For instance, in the case of a field $R=k$, one can invoke Auslander--Reiten theory to make precise that the extension group $\mathrm{Ext}^1_{k\A{2}}((k\to 0),(0\to k))\cong k$ is precisely the reason for this failure (see also \cite[\S5]{groth:thy-of-der}).
\item One can show, that, more generally for every free category $F$ the underlying diagram functors $\ndia_F$ is essentially surjective and full (\cite[Prop.~2.15]{cisinski:derivable} or \cite[Thm.~10.3.3]{RB:ABC}). (Let us recall that a free category is the path category of an oriented graph.)
\item In contrast to this, already in the case of a field $R=k$ and the non-free category $\square$, the underlying diagram functor $\ndia_\square$ does not preserve isomorphism types (see \cite[Example~3.17]{bg:cubical}) and is hence not full and essentially surjective.
\end{enumerate}
\end{rmk}

We now sketch the slightly unorthodox proof of \autoref{thm:verdier}. 

\begin{proof}\emph{(of \autoref{thm:verdier})}
We carry out the sketch proof in the case of $\Mod{R}$, but the arguments also apply to arbitrary abelian categories (see \autoref{rmk:unorthodox}). Moreover, we focus on those aspects which allow for nice ``representation theoretic explanations''. Let us take for granted that the derived category $D(R)$ is additive (see also \autoref{rmk:add}). The suspension functor $\Sigma\colon D(R)\to D(R)$ is simply the shift functor (\autoref{rmk:susp-sq-D_R}), and it clearly is an equivalence of categories.

We now define the class of distinguished triangles in $D(R)$. To this end, for every $(f\colon x\to y)\in \D_R([1])$ we can consider by \autoref{prop:cof-seq} the corresponding cofiber sequence $X\in\D_R([1]\times[2])^\cof$ as in \eqref{eq:cof-seq}. As detailed in \autoref{con:dist-tria} we obtain an associated triangle
\begin{equation}\label{eq:stand-tria}
x\stackrel{f}{\to} y\to z\to \Sigma x.
\end{equation}
This is defined to be the \emph{standard triangle} of $f\in\D_R([1])$, and a triangle in $D(R)$ is distinguished when it is isomorphic to a standard triangle. It is worthwhile to summarize the construction by the following diagram
\[
\xymatrix{
\D_R([1])\ar[r]_-\sim^-{F_{[1]}}\ar[d]_-{\ndia_{[1]}}\ar[rd]_-{\mathrm{tria}}&\D_R(M_2)^\ex\ar[d]\\
D(R)^{[1]}&D(R)^{[3]}.
}
\]
The horizontal functor is the equivalence constructed in the proof of \autoref{thm:BP}. The unlabeled vertical functor takes a Barratt--Puppe sequence, restricts it to the appropriate cofiber sequence and then passes to the underlying distinguished triangle. Note that the vertical functors and hence also the diagonal functor amount to a loss of information since in all three cases underlying diagram functors are involved.

In order to show that every morphism $f\colon x\to y$ in $D(R)$ extends to a distinguished triangle we invoke that $\ndia_{[1]}$ is essentially surjective (\autoref{prop:D_R-strong}). Thus, we can find $Y\in\D_R([1])$ and an isomorphism $\alpha\colon \ndia_{[1]}(Y)\toiso f$, and a combination of the standard triangle $\mathrm{tria}(Y)$ of $Y$ and the isomorphism $\alpha$ yields the intended distinguished triangle. Similarly, the weak functoriality of distinguished triangles is a consequence of $\ndia_{[1]}$ being essentially surjective and full. Up to the sign issue, the rotation axiom is essentially a consequence of the symmetries of Barratt--Puppe sequences (see also \autoref{rmk:BP}). The sign issue will be taken up again \S\ref{sec:crash} (see the discussion of additivity in \autoref{rmk:add}).

Instead, we content ourselves by taking care of the octahedral axiom. Let $B$ be the shape of the diagrams showing up in the octahedral axiom. For every
\[
X=(x\to y\to z)\in\D_R([2])
\]
we can consider the corresponding refined octahedral diagram $F_{[2]}(X)\in\D_R(M_3)^\ex$ (\autoref{thm:octa}). The above construction of standard triangles and \autoref{rmk:octa} imply that from $F_{[2]}(X)$ we can construct the intended octahedral diagram
\[
\mathrm{octa}(X)\colon B\to D(R).
\]
Similarly to the previous case, the situation can be summarized by the diagram
\begin{equation}\label{eq:dia-octa}
\vcenter{
\xymatrix{
\D_R([2])\ar[r]_-\sim^-{F_{[2]}}\ar[d]_-{\ndia_{[2]}}\ar[rd]_-{\mathrm{octa}}&\D_R(M_3)^\ex\ar[d]\\
D(R)^{[2]}&D(R)^B.
}
}
\end{equation}
Here, $F_{[2]}$ is the equivalence from \autoref{thm:octa} and the remaining functors amount to a loss of information. Since also the functor $\ndia_{[2]}$ is essentially surjective (\autoref{prop:D_R-strong}), this establishes the octahedral axiom.
\end{proof}
 
Let us collect some of the benefits which we obtain from this alternative construction of the classical Verdier triangulation on $D(R)$. 
 
\begin{rmk}\label{rmk:unorthodox}
Let $R$ be a ring, let $(x\stackrel{f}{\to}y\stackrel{g}{\to}z)\in\D_R([2])$, and let $F_{[2]}(X)$ be the associated refined actahedral diagram.
\begin{enumerate}
\item The shape $M_3$ (\autoref{fig:M3-shape}) has two obvious generating symmetries (translation and flip symmetry). By definition refined octahedral diagrams are preserved by restriction along these symmetries. This implies, for instance, that $X$ also encodes the refined octahedral diagram associated to the canonical representation of $\A{3}$ looking like
\[
(Cf\to C(g\circ f)\to \Sigma x).
\]
And this leads to a relation between the respective octahedral diagrams.
\item The refined octahedral diagram $F_{[2]}(X)$ encodes additional distinguished triangles teaching us something about the cone $C(g\circ f)$ of the composition. In fact, let us recall that every bicartesian square 
\[
\xymatrix{
x\ar[r]\ar[d]\ar@{}[rd]|{\square}&y\ar[d]\\
z\ar[r]&w
}
\]
in $\D_R$ gives rise to a cofiber square
\[
\xymatrix{
x\ar[r]\ar[d]\ar@{}[rd]|{\square}&y\oplus z\ar[d]\\
0\ar[r]&w,
}
\]
and hence to a distinguished triangle
\[
x\to y\oplus z\to w\to \Sigma x.
\]
See, for instance, \cite{gps:mayer} for a construction of these \textbf{Mayer--Vietoris sequences}. Let us specialize this to  bicartesian squares occurring in $F_{[2]}(X)$ (\autoref{fig:octa}). In this situation, the square $\fbox{2}$ gives rise to an additional distinguished triangle
\[
y\to z\oplus Cf\to C(g\circ f)\to \Sigma y.
\]
Similarly, the Mayer--Vietoris sequence associated to the square $\fbox{5}$ looks like
\[
C(g\circ f)\to \Sigma x\oplus Cg\to \Sigma y\to \Sigma C(g\circ f).
\]
\item The proof of the octahedral axiom (as summarized by \eqref{eq:dia-octa}) used only that $\ndia_{[2]}$ is essentially surjective. Since this functor is also full, we conclude that octahedral diagrams depend weakly functorially on pairs of composable morphisms in derived categories.
\item For simplicity, we presented the proof in the case of $\Mod{R}$ only. However, there are variants of the main ingredients (such as \autoref{thm:BP} and \autoref{thm:octa}) for arbitrary abelian categories $\cA$. In fact, by (co)finality arguments the constructions only rely on the existence of finite limits and colimits, and these exist in all abelian categories. Also the above remarks generalize to this more general situation.
\end{enumerate}
\end{rmk}

This sketch proof exhibits some of the axioms of triangulated categories as certain shadows of symmetries of Barratt--Puppe sequences and of refined octahedral diagrams. Put differently, this strategy makes precise the connection of these  axioms to abstract representations of the quivers $\A{2}$ and $\A{3}$. The quivers $\A{n}$ for $n\geq 4$ play a similar role for higher triangulations and we will study their abstract representations in more detail in \S\ref{subsec:Dynkin-A} (for the connection to higher triangulations see \autoref{rmk:higher-triangulations}).

\section{A crash course on derivators}
\label{sec:crash}

In this section we give a short introduction to derivators, which offer one of the many approaches to higher category theory, abstract homotopy theory, or homotopical algebra. The precise choice of this model is not too relevant for our later purposes (see \autoref{discl:higher-cats}), and here we essentially only present what is necessary for our later discussion of strong stable equivalences of quivers (or small categories). In fact, the main goal of this section is to provide enough background in order to turn the following pseudo-definition into an actual definition.

\setcounter{pseudodefn}{-1}

\begin{pseudodefn}
Two quivers $Q$ and $Q'$ are \textbf{strongly stably equivalent} if for every stable homotopy theory \D there is an equivalence between the homotopy theory $\D^Q$ of representations of shape $Q$ and the homotopy theory $\D^{Q'}$ of representations of shape $Q'$,
\[
\Phi_\D\colon\D^Q\toiso\D^{Q'},
\]
which is pseudo-natural with respect to exact morphisms $F\colon\D\to\E$ of stable homotopy theories.
\end{pseudodefn}

This goal pretty much dictates what we are supposed to do. In particular, as a first step we have to make precise what we mean by an ``abstract homotopy theory''. Here we choose to work with \emph{derivators} but see the following disclaimer.

\begin{discl}\label{discl:higher-cats}
By now there are various approaches to formalize an ``abstract homotopy theory''. All of these notions enhance certain defects of the more classical homotopy categories (which arise as 1-categorical localizations).
\begin{enumerate}
\item One of the most classical approaches is given by Quillen model categories \cite{quillen:ha,hovey:model,hirschhorn:model}. The homotopy theory is in this case encoded by a class of weak equivalences accompanied by classes of fibrations and cofibrations which are subject to certain axioms.
\item There is the zoo of different models for a theory of $(\infty,1)$-categories. One way to think of an individual $(\infty,1)$-category is as the result of a ``higher categorical localization'' of a category with weak equivalences. A very prominent model is given by $\infty$-categories, and this model was developed to an impressive extent most notably in \cite{joyal:barca} and \cite{HTT,HA}. With all these techniques at hand, this is currently the most flexible notion. Among the alternative approaches to such a theory are simplicial categories \cite{bergner:scat}, Segal categories \cite{hirschowitz-simpson}, and complete Segal spaces \cite{rezk:model}. Additional references and more details can for example be found in \cite{bergner:survey,simpson:higher,camerona:whirlwind,groth:scinfinity}. Moreover, a model-independent approach to higher category theory can be found in \cite{riehl-verity:2-cat-qcats} and the sequels.
\item Derivators \cite{grothendieck:derivators,heller:htpythies} axiomatize key properties of the calculus of homotopy limits and homotopy Kan extensions. In this language, these are characterized by ordinary universal properties making them accessible to elementary techniques.
\end{enumerate}

The notion of a derivator is a compromise: a derivator encodes more information than a mere homotopy category but somewhat less information than a full homotopy theory. Some of the defects of classical homotopy categories and triangulated categories are addressed successfully by the theory of derivators -- and this is done in a reasonably elementary way by means of (at most $2$-) categorical techniques only. For instance, derivators allow for a construction of exponentials and for a formal study of stability, thereby offering a first framework for this project.

At the same time there are obvious limitations to the theory of derivators (such as the absence of gluing of derivators or the  incoherence of morphisms and natural transformations), and for various purposes it is more convenient to use the more flexible theory of $\infty$-categories. Going even beyond this, for instance for the calculus of universal tilting modules (as in \S\ref{subsec:modules}), it would be desirable to have a more well-developed \emph{theory} of $(\infty,2)$-categories and monoidal such. 

In any case, the author believes in diversity of technology. Most of the results of this project were established in the language of derivators and this is hence also the model we choose in this account. However, it is very likely that sooner or later $\infty$-categories will enter the picture more prominently (see already \cite{dycker-jasso-walde:higher-AR,dycker-jasso-walde:BGP}).  
\end{discl}

The structure of this section is dictated by the above-mentioned main goal. In \S\ref{subsec:der} we introduce derivators and collect some examples. In \S\ref{subsec:stable} we turn to stable derivators and discuss their relation to triangulated categories. In \S\ref{subsec:exponentials} we construct exponentials for derivators which encode the calculus of parametrized limits and Kan extensions. Finally, in \S\ref{subsec:mor} we discuss morphisms of derivators, adjunctions, equivalences, and related notions.

\subsection{Derivators}
\label{subsec:der}

In this section we discuss the philosophy of derivators which offer one of the many approaches to abstract homotopy theory. Derivators were introduced independently by Grothendieck \cite{grothendieck:derivators} and Heller \cite{heller:htpythies}, and closely related notions were also considered by Franke \cite{franke:adams} and Keller \cite{keller:epivalent} (see \autoref{rmk:derivator-2} for some additional references). The main focus of derivators is on diagram categories and the related calculus of (derived or homotopy) limits, colimits, and Kan extensions. 

\begin{notn}
We denote by $\cCat$ the $2$-category of small categories, functors, and natural transformation, and similarly, by $\cCAT$ the $2$-category of not necessarily small categories.
\end{notn}

In this section we use very basic terminology related to $2$-categories. The discussions in \cite[\S XII]{maclane} and \cite[\S7]{borceux:1} suffice for our purposes (but see also \cite{kelly-street:review,lack:2-cat-companion}). 

\begin{defn}
A \textbf{prederivator} is a $2$-functor $\D\colon\cCat\op\to\cCAT$.
\end{defn}

Thus, a prederivator $\D$ associates to every small category $A\in\cCat$ a category $\D(A)$, the category of \textbf{coherent diagrams of shape $A$} in $\D$. For the trivial shape $\ast$, we refer to $\D(\ast)$ as the \textbf{underlying category} of $\D$. Moreover, every functor $u\colon A\to B$ induces a \textbf{restriction functor} or \textbf{precomposition functor} $u^\ast\colon\D(B)\to\D(A)$. Similarly, every natural transformation $\alpha\colon u\to v$ of functors $u,v\colon A\to B$ induces a natural transformation $\alpha^\ast\colon u^\ast\to v^\ast$, but in these notes we on purpose decide to not be too explicit about these $2$-dimensional aspects.

\begin{warn}
The above is only terminology. Given a prederivator $\D$ and $A\in\cCat$, an object $X\in\D(A)$ is just an abstract object, and not a diagram in any precise sense. However, as we will see, every such $X$ gives rise to a usual diagram $A\to\D(\ast)$, and it is crucial to distinguish between these two.
\end{warn}

\begin{eg}
Let $\cC$ be a category and $A\in\cCat$. We denote by $\cC^A$ the category of functors $X\colon A\to\cC$ and natural transformation between them. The formation of these diagram categories defines a $2$-functor
\[
y_\cC\colon\cCat\op\to\cCAT\colon A\mapsto\cC^A,
\]
the \textbf{prederivator represented by $\cC$}. The underlying category is $\cC$ itself.
\end{eg}

More interesting examples arise from this one by \emph{localization}. The following example collects two key examples of a construction which applies to arbitrary relative categories (categories with a class of weak equivalences), and this relies on the fact that localizations are $2$-localizations.

\begin{egs}
\begin{enumerate}
\item Let $\cA$ be an abelian category. The prederivator of $\cA$ is the $2$-functor
\[
\D_\cA\colon\cCat\op\to\cCAT\colon B\mapsto D(\cA^B).
\]
The underlying category of $\D_\cA$ is the derived category $D(\cA)$.
\item Let $\cM$ be a Quillen model category with class of weak equivalences $W$. For every $B\in\cCat$ we denote by $W^B$ those natural transformations $\alpha\colon X\to Y$ between diagrams $X,Y\colon B\to \cM$ such that the components $\alpha_b\colon X_b\to Y_b$ are weak equivalences for all $b\in B$, i.e., $W^B$ is the class of levelwise weak equivalences. The homotopy prederivator of $\cM$ is the $2$-functor
\[
\ho_\cM\colon\cCat\op\to\cCAT\colon B\mapsto \Ho(\cM^B)=\cM^B[(W^B)^{-1}].
\]
The underlying category of $\ho_\cM$ is the homotopy category $\Ho(\cM)$. 
\end{enumerate}
\end{egs}

\begin{con}\label{con:dia}
Let \D be a prederivator, $B\in\cCat$, and $b\in B$. We can consider the functor $b\colon\ast\to B$ which sends the unique object to $b\in B$, and the corresponding restriction functor 
\[
b^\ast\colon\D(B)\to\D(\ast)
\]
is referred to as an \textbf{evaluation functor} (at $b$). For every $f\colon X\to Y$ in $\D(B)$ we write $f_b\colon X_b\to Y_b$ for its image under the evaluation functor. As an exercise in $2$-functoriality, we invite the reader to check that every $X\in\D(B)$ gives rise to an underlying diagram
\[
\ndia_B(X)\colon B\to\D(\ast)\colon b\mapsto X_b,
\] 
and that this defines an \textbf{underlying diagram functor}
\[
\ndia_B\colon\D(B)\to\D(\ast)^B.
\]
\end{con}

\begin{warn}\label{warn:incoh}
Let \D be a derivator and $B\in\cCat$. In general, the underlying diagram functor $\ndia_B$ fails to be an equivalence. In fact, in many cases the categories $\D(B)$ and $\D(\ast)^B$ are not equivalent. We refer to the category $\D(\ast)^B$ as the category of \textbf{incoherent diagrams of shape $B$} in $\D$.

However, the underlying diagram functors are useful in order to visualize abstract coherent diagrams, i.e., we often draw $\ndia_B(X)$ and say that $X\in\D(B)$ looks like $\ndia_B(X)$. To reissue this warning, it is important to distinguish between these two.
\end{warn}

In specific situations the underlying diagram functors take the following form.

\begin{eg}
Let $B\in\cCat$.
\begin{enumerate}
\item In the case of a represented prederivator $y_\cC$, the underlying diagram $\ndia_B$ is the isomorphism of categories
\[
\ndia_B\colon \cC^B\toiso(\cC^\ast)^B.
\]
\item For every abelian category $\cA$ and the corresponding prederivator $\D_\cA$, the underlying diagram functors 
\[
\ndia_B\colon D(\cA^B)\to D(\cA)^B
\]
were already considered in \S\ref{sec:rep-thy-tria} (see, in particular, \autoref{rmk:der-mor}, \autoref{eg:dia-H}, and their discussion). From that discussion we know that these functors often fail to be equivalences.
\item Similarly, for every Quillen model category $\cM$, the corresponding underlying diagram functor of $\ho_\cM$ takes the form
\[
\ndia_B\colon\Ho(\cM^B)\to\Ho(\cM)^B.
\]
In general, homotopy categories of diagram categories and diagram categories of homotopy categories are not equivalent, and these underlying diagram functors hence again fail to be equivalences.
\end{enumerate}
\end{eg}

The point is that coherent diagrams (such as objects in $D(\cA^B)$ and $\ho(\cM^B)$) carry more information than incoherent ones. This information is crucial when one wants to calculate their derived or homotopy (co)limits. In fact, as we have seen in \S\ref{sec:rep-thy-tria}, morphisms in derived categories do not suffice to canonically determine their derived (co)kernels, and similarly for derived pushouts. The key properties listed in \autoref{thm:grothendieck} are now turned into an abstract definition, thereby capturing a formal calculus of abstract limits and colimits. Axiom (Der4) will be made more precise in the discussion that follows the definition.

\begin{defn}\label{defn:derivator}
A prederivator $\D\colon\cCat\op\to\cCAT$ is a \textbf{derivator} if it enjoys the following properties.
\begin{enumerate}
\item[(Der1)] The canonical inclusion functors $B_j\to\coprod_{i\in I} B_i, j\in I,$ induce an equivalence of categories
  \[
  \D(\coprod_{i\in I} B_i)\toiso\prod_{i\in I}\D(B_i).
  \]
\item[(Der2)] A morphism $f\colon X\to Y$ in $\D(B)$ is an isomorphism if and only if the morphisms $f_b\colon X_b\to Y_b, b\in B,$ are isomorphisms in $\D(\ast).$
\item [(Der3)] For every functor $u\colon A\to B$, the restriction functor $u^\ast\colon\D(B)\to\D(A)$ has a left adjoint $u_!$ and a right adjoint $u_\ast$,
\begin{equation}\label{eq:kan-adj}
(u_!,u^\ast)\colon\D(A)\rightleftarrows\D(B),\qquad (u^\ast,u_\ast)\colon\D(B)\rightleftarrows\D(A).
\end{equation}
\item[(Der4)] For every functor $u\colon A\to B$, the functors $u_!,u_\ast\colon\D(A)\to\D(B)$ can be calculated pointwise.
\end{enumerate}
\end{defn}

\begin{rmk}\label{rmk:derivator}
Some discussion of this definition is in order.
\begin{enumerate}
\item Axiom (Der1) makes precise that that coherent diagrams of shape given by a disjoint union are determined by the canonical restrictions to the respective summands. Axiom (Der2) is motivated by the following two examples.
\begin{enumerate}
\item A natural transformation $\alpha\colon X\to Y$ in a diagram category $\cC^B$ is an isomorphism if and only if all components $\alpha_b\colon X_b\to Y_b$ are isomorphisms.
\item For every abelian category $\cA$ and $B\in\cCat$ there is an obvious isomorphism $\nCh(\cA^B)\toiso \nCh(\cA)^B$. Under this isomorphism the quasi-isomorphisms $W_{\cA^B}$ in $\cA^B$ correspond precisely to the levelwise quasi-isomorphisms $W_\cA^B$.
\end{enumerate}
\item Axioms (Der3) and (Der4) jointly encode, first,  a ``\emph{homotopical} completeness and cocompleteness property'', thereby guaranteeing that ``homotopical versions'' of limits, colimits, and Kan extensions exist, and, second, formulas which allow us to calculate the (homotopy) Kan extensions. Let us expand a bit on this.

As a special case, for every $A\in\cCat$, there is the unique functor $\pi_A\colon A\to\ast$. Correspondingly, by (Der3) the functor $\pi_A^\ast\colon\D(\ast)\to\D(A)$ has a left adjoint $(\pi_A)_!$ and a right adjoint $(\pi_A)_\ast$, which we respectively also denote by
\[
\colim_A=(\pi_A)_!\colon\D(A)\to\D(\ast)\quad\text{and}\quad\mathrm{lim}_A=(\pi_A)_\ast\colon\D(A)\to\D(\ast).
\]
These abstract \textbf{colimit} and \textbf{limit} functors generalize ordinary categorical (co)limits, derived (co)limits, and homotopy (co)limits (see \autoref{egs:der}).

More generally, the adjoints $u_!$ and $u_\ast$ are referred to as \textbf{left} and \textbf{right Kan extensions}, respectively. To be able to work with these adjoints, it is crucial that they can be calculated as in the classical case (\autoref{con:ptws-kan}). To formulate this abstractly let us assume that \D is a prederivator which satisfies (Der3). For every functor $u\colon A\to B$, for every $b\in B$, and every coherent diagram $X\in\D(A)$ there are canonical maps
\begin{equation}\label{eq:mates-in-Der4}
\colim_{(u/b)}p^\ast (X)\to u_!(X)_b\quad\text{and}\quad u_\ast(X)_b\to \mathrm{lim}_{(b/u)}q^\ast (X).
\end{equation}
(Here, $p$ and $q$ are the projection functors associated to slice categories as in \autoref{con:ptws-kan}.) Axiom (Der4) asks that these canonical maps are isomorphisms, and in this precise sense Kan extensions in derivators are pointwise.
\item The construction of the canonical maps in \eqref{eq:mates-in-Der4} is an instance of the calculus of canonical mates \cite{kelly-street:review}. While working with derivators one often runs into the situation that outputs of certain universal constructions ``obviously are isomorphic''. In many cases, the formalism of mates and related notion of \emph{homotopy exact squares} allows one to actually conclude this by, first, providing canonical maps between such gadgets and, second, guaranteeing that these maps are isomorphisms \cite{groth:ptstab,maltsiniotis:htpy-exact}. For a more detailed discussion of the formalism behind it we refer to \cite[\S8 and \S10]{groth:thy-of-der}. In many cases, establishing a rule to manipulate Kan extensions amounts to showing that a certain square is homotopy exact (see \cite{grothendieck:stacks,grothendieck:derivators} and \cite{cisinski:loc-min,cisinski:presheaves,maltsiniotis:grothendieck} for many more advanced examples).
\item All axioms of a derivator ask for certain \emph{properties} of the underlying prederivators, which is the only actual \emph{structure}. In contrast to this, in other approaches such as triangulated categories some \emph{non-canonical structure} is put on an underlying category. 
\end{enumerate}
\end{rmk}

\begin{egs}\label{egs:der}
Let us take up again the above examples of prederivators.
\begin{enumerate}
\item Let $\cC$ be an ordinary category. The category $\cC$ is complete and cocomplete if and only if $y_\cC\colon B\mapsto\cC^B$ is a derivator, the \textbf{derivator represented by $\cC$}. In this case the abstract (co)limits are the usual ones from ordinary category theory, and the more general adjoints $u_!,u_\ast$ are the classical Kan extensions. For an introduction to these Kan extensions we refer to \cite[\S6]{groth:thy-of-der} and for a detailed proof that $y_\cC$ indeed is a derivator to \cite[\S9.2]{groth:thy-of-der}.
\item For every Grothendieck abelian category $\cA$, $\D_\cA\colon B\mapsto D(\cA^B)$ is a derivator, the \textbf{derivator of (chain complexes in) $\cA$}. In this case abstract (co)limits are derived (co)limits, and similarly for Kan extensions. The fact that $\D_\cA$ is a derivator follows from the next example.

As special cases, there are hence the derivator $\D_k$ of a field $k$ and the derivator $\D_R$ of a ring $R$ (see \autoref{thm:grothendieck}). In contrast to the derived category $D(k)$ of a field, the derivator $\D_k$ of a field already is quite interesting as it encodes, for instance, derived categories of path algebras, incidence algebras, and group algebras. As an additional class of examples, associated to every scheme $X$ there is the derivator $\D_X$ of (chain complexes of) quasi-coherent $\mathcal{O}_X$-modules. In fact, by \cite[\S3]{enochs-estrada:relative} the category of quasi-coherent $\mathcal{O}_X$-modules is Grothendieck abelian for arbitrary schemes $X$. 
\item For every Quillen model category $\cM$, the prederivator $\ho_\cM\colon B\mapsto \Ho(\cM^B)$ is a derivator, the \textbf{homotopy derivator of $\cM$}. In full generality this was established by Cisinski as \cite[Thm.~6.11]{cisinski:direct}, and related to this see~\cite{chacholski-scherer:diagram}. For an even more general version we also refer to \cite[Thm.~2.21, Cor.~2.24, and Cor.~2.28]{cisinski:derivable}. For combinatorial Quillen model categories an alternative proof can be found in \cite[\S1.3]{groth:ptstab}. For homotopy derivators abstract (co)limits specialize to homotopy (co)limits and similarly for Kan extensions. Let us specialize this important class to some examples of central interest. 
\begin{enumerate}
\item For every Grothendieck abelian category $\cA$, the category $\nCh(\cA)$ admits a combinatorial model structure with quasi-isomorphisms as weak equivalences (see \cite[Prop.~3.13]{beke:sheafifiable}, \cite[\S2]{hovey:sheaves}, or \cite[\S\S2-3]{cisinski-deglise:local-stable}). Hence, the derivator $\D_\cA$ of $\cA$ can be described as homotopy derivator $\ho_{\nCh(\cA)}$.

\item As an universal example there is the \textbf{derivator of spaces}
\[
\cS=\ho_{\mathrm{Top}},
\]
which arises homotopy derivator of the classical Serre model structure \cite{quillen:ha}. For alternative approaches to $\cS$ we also refer to \cite{thomason:model-Cat,grothendieck:stacks,maltsiniotis:grothendieck,cisinski:presheaves}. In the case of $\cS$ abstract limits are the classical homotopy limits of topological spaces. Systematic classical accounts are in \cite{bousfield-kan:htpy-limits,boardman-vogt} and special instances can, for example, be found in \cite{mather:pullbacks}. Similarly, the \textbf{derivator of pointed spaces} is $\cS_\ast=\ho_{\mathrm{Top}_\ast}$.
\item An important role in this project is played by the derivator of spectra. This derivator is defined as 
\[
\cSp=\ho_{\mathrm{Sp}},
\]
the homotopy derivator with respect to the classical model structure on sequential spectra \cite{bousfield-friedlander}. The underlying category of $\cSp$ is \emph{the} stable homotopy category $\mathcal{SHC}$ \cite{boardman:thesis,puppe:stabil,vogt:boardman}. Alternative (and monoidal) approaches to this derivator are based on \cite{hss:symmetric,ekmm:rings,mmss:diagram}.
\end{enumerate}
\end{enumerate}
\end{egs}

There is also a variant of homotopy derivators of complete and cocomplete $\infty$-categories \cite{riehl-verity:Kan-ext,lenz:derivators}. Morally, the same should be true for all other approaches to a theory of complete and cocomplete $(\infty,1)$-categories, This moral is, for example, based on the model-independent approach to $(\infty,1)$-category theory of Riehl--Verity \cite{riehl-verity:2-cat-qcats,riehl-verity:fibration,riehl-verity:Kan-ext}.

\begin{rmk}\label{rmk:derivator-2}
A few more comments related to \autoref{defn:derivator} are in order.
\begin{enumerate}
\item There is some flexibility with respect to the choice of allowable shapes in the definition of a derivator. In this paper we always allow for arbitrary small shapes, but in some situations more restrictive assumptions are useful. For instance, if one uses bounded derived categories of diagram categories of suitably finite shapes only, then by a result of Keller \cite{keller:exact} for every exact category in the sense of Quillen \cite[\S2]{quillen:k-theory} there is a corresponding derivator.
\item Following the convention of Anderson \cite{anderson:axiomatic}, Heller \cite{heller:htpythies}, and Franke \cite{franke:adams}, the axioms in \autoref{defn:derivator} encode aspects of the calculus of \emph{diagrams} (covariant functors) in a fixed abstract homotopy theory. Consequently, the domain of definition is $\cCat\op$. Grothendieck \cite{grothendieck:derivators}, Cisinski \cite{cisinski:direct,cisinski:derived-kan}, Maltsiniotis \cite{maltsiniotis:intro,maltsiniotis:k-theory} consider \emph{presheaves} (contravariant functors) instead, and consequently their domain of definition is $\cCat\coop$, which is obtained from $\cCat$ by changing the orientation of functors and natural transformations. The resulting theories are equivalent.
\item Up to the fact that we give preference to diagrams and not presheaves, the precise form of \autoref{defn:derivator} is due to Grothendieck \cite[pp.~43-46]{grothendieck:derivators}. Similar axioms were also proposed by Anderson \cite[\S2]{anderson:axiomatic}, Heller \cite{heller:htpythies}, and Franke \cite{franke:adams}. Moreover, related notions were consider by Keller \cite{keller:epivalent} and Garkusha \cite{garkusha:I,garkusha:II}. For an additional introductory account we also refer to work of Maltsiniotis \cite{maltsiniotis:intro,maltsiniotis:seminar}.
\end{enumerate}
\end{rmk}

We conclude this section by one trivial example.

\begin{eg}
Let \D be a derivator. The \textbf{opposite derivator} $\D\op$ of \D is given by
\[
\D\op\colon\cCat\op\to\cCAT\colon B\mapsto \D(B\op)\op.
\]
\end{eg}

As in ordinary category theory, this example is important because of the resulting \emph{duality principle}. In many later statements we allow ourselves to focus on results on colimits and left Kan extensions. An application to opposite derivators yields the corresponding dual statement. The formation of opposites of derivators is compatible with the formation of opposites of complete and cocomplete categories, abelian categories and model categories.

\subsection{Stable derivators}
\label{subsec:stable}

Having introduced derivators as a compromise model for abstract homotopy theories in \S\ref{subsec:der}, in this subsection we briefly discuss stable derivators. The main goal is to develop some intuition for the notion and to discuss the relation to triangulated categories. It will turn out that many of the arguments in \S\ref{sec:rep-thy-tria} apply verbatim in this more general context. 

To build towards stable derivators, we begin by the following definition.

\begin{defn}\label{defn:pointed}
A derivator \D is \textbf{pointed} if the underlying category $\D(\ast)$ has a zero object.
\end{defn}

\begin{rmk}\label{rmk:pointed}
Let \D be a derivator.
\begin{enumerate}
\item As a consequence of (Der1), the underlying category $\D(\ast)$ has an initial object and a final object, and these objects are hence asked to be isomorphic. As usual, any such object is denoted by $0\in\D(\ast)$.
\item It follows from (Der3) that for every $B\in\cCat$ also $\D(B)$ has a zero object, and that all restriction and Kan extension functors preserve these zero objects.
\end{enumerate}
\end{rmk}

\begin{egs}\label{egs:pt-der}
Let us take up again \autoref{egs:der}.
\begin{enumerate}
\item Let $\cC$ be a complete and cocomplete category. The underlying category of the represented derivator $y_\cC$ is $\cC$, and $y_\cC$ is hence pointed if and only if $\cC$ is pointed (has a zero object).
\item Let $\cA$ be a Grothendieck abelian category. The underlying category of the derivator $\D_\cA$ is the derived category $D(\cA)$. As an additive category, $D(\cA)$ clearly has a zero object, and $\D_\cA$ is hence pointed.
\item Similarly, for a pointed Quillen model category $\cM$, the homotopy derivator $\ho_\cM$ is pointed. Among the specific examples $\cS$, $\cS_\ast$, and $\cSp$ only the last two are pointed.
\end{enumerate}
\end{egs}

\begin{defn}\label{defn:cocart}
Let  $\D$ be a derivator and let $X\in\D(\square)$.
\begin{enumerate}
\item The square $X$ is \textbf{cocartesian} if it lies in the essential image of the left Kan extension functor
\[
(i_\ulcorner)_!\colon\D(\ulcorner)\to\D(\square).
\] 
\item The square $X$ is \textbf{cartesian} if it lies in the essential image of the right Kan extension functor
\[
(i_\lrcorner)_\ast\colon\D(\lrcorner)\to\D(\square).
\] 
\end{enumerate}
\end{defn}

\begin{rmk}\label{rmk:cocart}
There are the following remarks related to \autoref{defn:cocart}.
\begin{enumerate}
\item Let \D be a derivator and let $X\in\D(\square)$ be a square in $\D$. It turns out that $X$ is cocartesian if and only if a certain canonical map
\[
\colim_\ulcorner (i_\ulcorner)^\ast X\to X_{1,1}
\]
is an isomorphism. The formalism behind these canonical maps (based on the calculus of canonical mates) will not be made explicit in this paper (but see \cite{kelly-street:review} or \cite{groth:ptstab}). In this paper, we will occasionally refer to certain maps in derivators as ``canonical'', and in such situations they always arise by means of this calculus. 
\item Let us take up again our standard examples (\autoref{egs:der}). A square in a represented derivator is cocartesian if and only if it is a pushout square. For every Grothendieck abelian category $\cA$, a square in $\D_\cA$ is cocartesian if and only if is a derived pushout square. In particular, in the derivator $\D_R$ of a ring we recover \autoref{defn:cocart-D_R}, which played a key role in \S\ref{sec:rep-thy-tria}. Finally, a square in the homotopy derivator of a Quillen model category is cocartesian if and only if it is a homotopy pushout square.
\end{enumerate}
\end{rmk}

\begin{prop}\label{prop:Kan-ff}
Let \D be a derivator and $u\colon A\to B$ a fully faithful functor. The Kan extension functors $u_!,u_\ast\colon\D(A)\to\D(B)$ are fully faithful.
\end{prop}
\begin{proof}
For a proof based on the calculus of canonical mates and the corresponding formalism of homotopy exact squares we refer to \cite[Prop.~1.20]{groth:ptstab}.
\end{proof}

This property applied to $\D_R$ (\autoref{prop:Kan-ff-D_R}) was central to many constructions in \S\ref{sec:rep-thy-tria}. In order to develop some intuition for derivators we generalize some of these constructions to pointed derivators. 

\begin{con}\label{con:basic-ptd}
Let \D be a pointed derivator.
\begin{enumerate}
\item We begin with a derivator version of \autoref{prop:cof-sq-D_R}. Let $f\in\D([1])$ be a morphism looking like $(f\colon x\to y)$, which is to say that $\ndia_{[1]}(f)\colon[1]\to\D(\ast)$ takes the form $(f\colon x\to y)$ (see \autoref{warn:incoh}). In order to extend $f$ to a cofiber square, we again consider the fully faithful inclusion functors
\[
i_\ulcorner\circ i\colon[1]\to\ulcorner\to\square
\]
and their corresponding Kan extension functors
\begin{equation}\label{eq:cof-sq-fun-der}
(i_\ulcorner)_!\circ i_\ast\colon\D([1])\to\D(\ulcorner)\to\D(\square).
\end{equation}
The image $(i_\ulcorner)_!\circ i_\ast(f)\in\D(\square)$ can be restricted along the inclusion $k\colon[1]\to\square$ pointing at the vertical morphism on the right. We define the \textbf{cone functor}
\[
\cof=k^\ast\circ (i_\ulcorner)_!\circ i_\ast\colon\D([1])\to\D([1]).
\]
The object obtained from $\cof(f)$ by an evaluation at $1\in[1]$ is also referred to as the cone of $f$, but the corresponding construction is distinguished notationally by
\[
C=1^\ast\circ\cof\colon\D([1])\to\D(\ast).
\] 
With this notation, the \textbf{cofiber square} $(i_\ulcorner)_!\circ i_\ast(f)$ associated to $f$ looks like
\begin{equation}\label{eq:cof-sq-der}
\vcenter{
\xymatrix{
x\ar[r]^-f\ar[d]&y\ar[d]^-{\cof(f)}\\
0\ar[r]&C(f).\pushoutcorner
}
}
\end{equation}

We invite the reader to dualize these construction in order to define \textbf{fiber functors}
\[
\fib\colon\D([1])\to\D([1])\qquad\text{and}\qquad F\colon\D([1])\to\D(\ast).
\]
\item As a minor variant of the previous case, we now construct suspension and loop functors in pointed derivators, thereby generalizing \autoref{rmk:susp-sq-D_R}. There are at least two ways of defining the suspension. First, for $x\in\D(\ast)$ we would like to make the definition 
\[
\Sigma x=C(x\to 0).
\]
To explain the right hand side in terms of Kan extensions, it suffices to consider the functor $0\colon\ast\to[1]$ pointing at zero and the corresponding right Kan extension functor $0_\ast\colon\D(\ast)\to\D([1])$. In fact, $0_\ast$ extends $x$ to a morphism with $0$ as target as follows from (Der4). Correspondingly, we define the \textbf{suspension functor} as
\[
\Sigma=C\circ 0_\ast\colon\D(\ast)\to\D(\ast).
\]
Second, a canonically isomorphic way of defining the suspension is by means of
\[
\Sigma=(1,1)^\ast\circ(i_\ulcorner)_!\circ (0,0)_\ast\colon\D(\ast)\to\D(\ulcorner)\to\D(\square)\to\D(\ast).
\]
In both cases the suspension $\Sigma x$ sits as lower right corner in the corresponding \textbf{suspension square}
\begin{equation}\label{eq:susp-sq-der}
\vcenter{
\xymatrix{
x\ar[r]\ar[d]&0\ar[d]\\
0\ar[r]&\Sigma x.\pushoutcorner
}
}
\end{equation}
It is straightforward to dualize these constructions in order to obtain a \textbf{loop functor}
\[
\Omega\colon\D(\ast)\to\D(\ast).
\]
\end{enumerate}
\end{con}

\begin{egs}\label{egs:basic-ptd}
Let us specialize these constructions in our examples of pointed derivators (\autoref{egs:pt-der}).
\begin{enumerate}
\item Let $\cC$ be a pointed, complete and cocomplete category and $y_\cC$ its represented derivator. The cofiber squares \eqref{eq:cof-sq-der} and the suspension squares \eqref{eq:susp-sq-der} in $y_\cC$ are ordinary pushout squares. Hence, \eqref{eq:cof-sq-der} implies that $\cof$ and $C$ in $y_\cC$ are the usual cokernel functors (once considered with the structure map and once without). As a special case of this, the suspension squares \eqref{eq:susp-sq-der} imply that $\Sigma x$ is isomorphic to a zero object, and $\Sigma$ is hence naturally isomorphic to the constant functor on the zero object,
\begin{equation}\label{eq:sigma-rep}
\Sigma\cong 0\colon \cC\to\cC.
\end{equation}
\item Let $\cA$ be a Grothendieck abelian category and $\D_\cA$ its derivator. In this case the cone functor $C\colon D(\cA^{[1]})\to D(\cA)$ reproduces the functorial cone construction from \autoref{prop:C-cok} (see also the proof of \autoref{prop:der-push}). As noted in \S\ref{sec:rep-thy-tria}, the functor
\[
\cof\colon D(\cA^{[1]})\to D(\cA^{[1]})
\]
is in the background of the rotation of distinguished triangles (as a shadow of the symmetries of Barratt--Puppe sequences in \autoref{fig:BP}). The suspension functor in $\D_\cA$ specializes to the usual shift functor $\Sigma\colon D(\cA)\to D(\cA)$ (\autoref{rmk:susp-sq-D_R}).
\item As a final example, we consider the homotopy derivator $\cS_\ast$ of pointed spaces which also provides the motivation for some of the terminology employed here. Given a morphism $(f\colon X\to Y)\in\Ho(\mathrm{Top}_\ast^{[1]})$ of pointed topological spaces, the cone $C(f)$ is the usual mapping cone construction. Moreover, the abstract suspension specializes to the reduced suspension functor
\[
\Sigma\colon\Ho(\mathrm{Top}_\ast)\to\Ho(\mathrm{Top}_\ast),
\]
and
\[
\Omega\colon\Ho(\mathrm{Top}_\ast)\to\Ho(\mathrm{Top}_\ast)
\]
is the usual loop space functor (see also \autoref{rmk:loop-sign}).
\end{enumerate}
\end{egs}

\begin{prop}
Let \D be a pointed derivator.
\begin{enumerate}
\item There is an adjunction $(\cof,\fib)\colon\D([1])\rightleftarrows\D([1]).$
\item There is an adjunction $(\Sigma,\Omega)\colon\D(\ast)\rightleftarrows\D(\ast).$
\end{enumerate}
\end{prop}
\begin{proof}
For a proof we refer to \cite[\S3]{groth:ptstab}.
\end{proof}

In triangulated categories we are used to be able to rotate distinguished triangles without a loss of information and similarly that the suspension functor is an equivalence. With the few exceptions of exotic triangulated categories \cite{mss:no-models}, for most triangulated categories arising in nature these features are consequences of properties of \emph{stable} homotopy theories in the background.

\begin{defn}\label{defn:stable}
A pointed derivator \D is \textbf{stable} if every square in $\D$ is cocartesian if and only if it is cartesian, and these squares are then referred to as \textbf{bicartesian}.
\end{defn}

In order to find examples of stable derivators, it is useful to have simpler characterizations of stability.

\begin{thm}\label{thm:stable}
The following are equivalent for a pointed derivator \D.
\begin{enumerate}
\item The derivator $\D$ is stable.
\item The adjunction $(\cof,\fib)\colon\D([1])\rightleftarrows\D([1]$ is an equivalence.
\item The adjunction $(\Sigma,\Omega)\colon\D(\ast)\rightleftarrows\D(\ast)$ is an equivalence.
\end{enumerate}
\end{thm}
\begin{proof}
If $\D$ is stable, then cofiber squares \eqref{eq:cof-sq-der} and fibre squares are essentially the same, and similarly for suspension squares \eqref{eq:susp-sq-der} and loop squares. From this it is fairly straightforward to deduce that (i) implies (ii) and (iii). The converse implications are more involved and we refer the reader to \cite[Thm.~7.1]{gps:mayer}.
\end{proof}

\begin{egs}\label{egs:stable}
With this preparation we can revisit our examples of pointed derivators (\autoref{egs:pt-der}).
\begin{enumerate}
\item Let $\cC$ be a pointed, complete and cocomplete category and $y_\cC$ its represented derivator. In this case the abstract suspension functor is naturally isomorphic to the constant functor on the zero object \eqref{eq:sigma-rep}. Hence, by \autoref{thm:stable} we conclude that $y_\cC$ is stable if and only if $\Sigma\cong 0$ is an equivalence of categories if and only if $\cC$ is equivalent to the terminal category $\ast$.
\item Let $\cA$ be a Grothendieck abelian category and $\D_\cA$ its derivator. In $\D_\cA$ the abstract suspension functor agrees with the shift functor $\Sigma\colon D(\cA)\to D(\cA)$, which clearly is an equivalence of categories. Hence, derivators of the form $\D_\cA$ are stable.
\item A stable model category is a pointed model category such that the abstract suspension functor is an equivalence in its homotopy derivator (\cite[Def.~7.1.1]{hovey:model}). By \autoref{thm:stable} we deduce that homotopy derivators of stable model categories are stable (as it also follows from \cite[Rmk.~7.1.12]{hovey:model}). As a special case, the derivator $\cSp$ of spectra is stable, and there are many additional interesting examples (see \autoref{egs:stable-2} or \cite{schwede-shipley:morita}). 
\end{enumerate}
\end{egs}

Additional classes of examples arise as homotopy derivators of exact categories \cite{gillespie:exact,stovicek:exact-model}, stable cofibration categories \cite{schwede:p-order,lenz:derivators}, or stable $\infty$-categories \cite{HTT,HA,lenz:derivators}.

\begin{rmk}\label{rmk:stable-invisible}
We want to stress the observation that a represented derivator $y_\cC$ is stable if and only if $\cC\simeq\ast$. This essentially says that stability is \emph{invisible to ordinary category theory}. In order to capture the phenomenon of stability as it arises in interesting examples in homological algebra, stable homotopy theory or homotopical algebra we have to use more refined techniques. Among these are triangulated categories, stable derivators, stable model categories, and stable $\infty$-categories (see again \autoref{discl:higher-cats}).
\end{rmk}

We conclude this subsection by sketching the relation to triangulated categories. Many arguments are completely parallel to the discussion in \S\ref{sec:rep-thy-tria}, and here we hence focus on the remaining arguments. We begin by sketching that stability implies additivity. In order to motivate the general approach, we revisit the classical situation in topology.

\begin{rmk}\label{rmk:loop-sign}
Let $(X,x_0)$ be a pointed topological space. The \emph{loop space} $\Omega X$ of $X$ at $x_0$ is the space of paths $[0,1]\to X$ which send both boundary points $0,1$ to $x_0$. The loop space is the homotopy pullback (in the usual Serre model structure) of the cospan on the left in 
\[
\xymatrix{
&\ast\ar[d]^-{x_0}&&\Omega X\ar[r]\ar[d]\pullbackcorner&PX\ar[d]^-{\mathrm{ev}_1}\\
\ast\ar[r]_-{x_0}&X,&& \ast\ar[r]_-{x_0}&X.
}
\]
In fact, if one wants to calculate this homotopy pullback, then it suffices to replace one of the maps $x_0\colon\ast\to X$ by a weakly homotopy equivalent Serre fibration and then to calculate the categorical pullback. The standard example of such an approximation is the path space $PX$ of paths $[0,1]\to X$ starting at $x_0$ endowed with the evaluation map $\mathrm{ev_1}\colon PX\to X$ at the end point $1$. This space is weakly contractible since all such paths are homotopic to the constant path at $x_0$, and the pullback of the cospan on the right gives the above description of the loop space. The concatenation of loops can be used to show that the loop space $\Omega X$ is a group object in the homotopy category $\Ho(\mathrm{Top}_\ast)$.

In order to motivate the general approach in pointed derivators, we explain a purely categorical way of modelling the inversion of loops in topology. And for this purpose it is convenient to use a different model for the above homotopy pullback, which is obtained by replacing \emph{both} maps by the above Serre fibration and then calculating the pullback
\begin{equation}\label{eq:classical-inverse-loops}
\vcenter{
\xymatrix{
\Omega X\pullbackcorner\ar[r]^-{p_2}\ar[d]_-{p_1}&PX\ar[d]^-{\mathrm{ev_1}}\\
PX\ar[r]_-{\mathrm{ev_1}}&X.
}
}
\end{equation}
Up to an obvious homeomorphism, this model of the loop space has as points paths $[-1,1]\to X$ which send $-1,1$ to $x_0$, and in this model the inversion of loops $\iota\colon\Omega X\to\Omega X$ is given by a reparametrization via the reflection at $0\in[-1,1]$. The point is that this reparametrization simply amounts to interchanging the two copies of $PX$ in the pullback square \eqref{eq:classical-inverse-loops}. More formally, the outer commutative square in the diagram
\[
\xymatrix{
\Omega X\ar[rrd]^-{p_1}\ar[ddr]_-{p_2}\ar@{-->}[dr]&&\\
&\Omega X\ar[r]_-{p_2}\ar[d]^-{p_1}\pullbackcorner&PX\ar[d]^-{\mathrm{ev_1}}\\
&PX\ar[r]_-{\mathrm{ev_1}}&X
}
\]
induces by the universal property of the pullback square a canonical dashed morphism making both triangles commute. And this morphism of course is the above reparametrization map $\iota\colon\Omega X\to\Omega X$.
\end{rmk}

This can be abstracted to pointed derivators as follows.

\begin{rmk}\label{rmk:add}
Let \D be a pointed derivator. For every $x\in\D(\ast)$, the loop object $\Omega x$ can be naturally turned into a group object. The vague idea is to model the concatenation of loops from topology purely categorically. The details are a bit more involved (see \cite[\S4.1]{groth:ptstab}) and they rely on the Segalian approach to $A_\infty$-multiplications \cite{segal:categories}. (In fact, using fibrational techniques the arguments in \cite[\S4.1]{groth:ptstab} can be adapted to construct coherently associative multiplication maps.) One can make precise that the inversion map $(-)^{-1}\colon\Omega x\to\Omega x$ comes from the swap symmetry of loop squares (the duals of \eqref{eq:susp-sq-der}) which interchanges $(1,0)$ and $(0,1)$. A variant of the \emph{Eckmann--Hilton argument}~\cite[Thm.~4.17]{eckmann-hilton:group-like-I} implies that two-fold loop objects $\Omega^2x$ are abelian group objects.
\end{rmk}

\begin{cor}\label{cor:stab-add}
The underlying category of a stable derivator is additive.
\end{cor}
\begin{proof}
In a stable derivator $\D$, the suspension $\Sigma\colon\D(\ast)\to\D(\ast)$ is an equivalence (\autoref{thm:stable}), and for every object $x\in\D(\ast)$ there is a natural isomorphism $x\cong\Omega^2\Sigma^2x$. By \autoref{rmk:add} this implies that every object in $\D(\ast)$ is an abelian group object, showing that $\D(\ast)$ is additive.
\end{proof}

With this preparation the construction of triangulations on stable derivators (\autoref{thm:stable-tria}) is fairly parallel to \S\ref{sec:rep-thy-tria}. As a variant of the underlying diagram functor (\autoref{con:dia}), for every prederivator $\D$ there are partial underlying diagram functors
\[
\ndia_{A,B}\colon \D(A\times B)\to\D(A)^B,\qquad A,B\in\cCat,
\]
making a diagram incoherent in the $B$-direction only. More formally, for a coherent diagram $X\in\D(A\times B)$ and $b\in B$ we set
\[
\ndia_{A,B}(X)_b=(\id_A\times b)^\ast(X)\in\D(A),
\]
the restriction along $\id_A\times b\colon A\cong A\times \ast\to A\times B$.

\begin{defn}\label{defn:strong}
A derivator $\D$ is \textbf{strong} if the partial underlying diagram functor
\[
\ndia_{A,F}\colon\D(A\times F)\to\D(A)^F
\]
is essentially surjective and full for all free categories $F\in\cCat$ and all $A\in\cCat$.
\end{defn}

\begin{con}\label{con:can-tria}
Let \D be a stable derivator. By \autoref{con:basic-ptd} the underlying category $\D(\ast)$ can be endowed with a suspension functor
\begin{equation}\label{eq:can-tria-1}
\Sigma\colon\D(\ast)\to\D(\ast)
\end{equation}
which is an equivalence by \autoref{thm:stable}. For every morphism $(f\colon x\to y)\in\D([1])$, completely parallel to the sketch proof of \autoref{thm:verdier} (see the discussion around \eqref{eq:stand-tria}) one constructs the corresponding standard triangle of $f$. A triangle
\begin{equation}\label{eq:can-tria-2}
x\stackrel{f}{\to} y\stackrel{g}{\to} z\stackrel{h}{\to}\Sigma x
\end{equation}
in $\D(\ast)$ is distinguished if it is isomorphic to a standard triangle.
\end{con}

The following theorem has some history, and we refer to work of Franke \cite{franke:adams}, Maltsiniotis \cite{maltsiniotis:seminar}, and the author \cite{groth:ptstab}.

\begin{thm}\label{thm:stable-tria}
For every strong and stable derivator $\D,$ the underlying category $\D(\ast)$ admits a triangulation given by \eqref{eq:can-tria-1} and \eqref{eq:can-tria-2}. 
\end{thm}
\begin{proof}
With this preparation, the proof is essentially the same as the proof of \autoref{thm:verdier}. We expand a bit on the rotation axiom in order to explain where the sign comes from,  and refer the reader to \emph{loc.~cit.} for the remaining aspects. For every morphism $(f\colon x\to y)\in\D([1])$ there is an associated Barratt--Puppe sequence (\autoref{fig:BP}). Let us redraw the relevant part of it and add decorations to the various zero objects in order to distinguish them notationally. So let $B$ be the following poset and let $Y\in\D(B)$ look like
\begin{equation}\label{eq:rot-ax}
\vcenter{
\xymatrix{
x\ar[r]^-f\ar[d]\ar@{}[rd]|{\square}&y\ar[r]\ar[d]_-g\ar@{}[rd]|{\square}&0_2\ar[d]\\
0_1\ar[r]&z\ar[r]^-h\ar[d]\ar@{}[rd]|{\square}&x'\ar[d]^-{f'}\\
&0_3\ar[r]&y'.
}
}
\end{equation}
In order to identify the underlying morphism of $f'$ as $\Sigma f$, we make the identifications
\[
\varphi_1\colon \Sigma x\toiso x'\qquad\text{and}\qquad\psi_1\colon \Sigma y\toiso y',
\]
and these lead to the commutative square
\begin{equation}\label{eq:f'-1}
\vcenter{
\xymatrix{
\Sigma x\ar[r]^-{\varphi_1}_-\sim\ar[d]_-{\Sigma f}&x'\ar[d]^-{f'}\\
\Sigma y\ar[r]_-{\psi_1}^-\sim&y'.
}
}
\end{equation}
Strictly speaking these identification $\varphi_1$ and $\psi_1$ respectively are the canonical maps associated to the suspension squares
\[
\xymatrix{
x\ar[r]\ar[d]\ar@{}[rd]|{\square}&0_2\ar[d]&&y\ar[r]\ar[d]\ar@{}[rd]|{\square}&0_2\ar[d]\\
0_1\ar[r]&x',&&0_3\ar[r]&y'.
}
\]
Similarly, in order to identify the final object in the distinguished triangles of $f$ and $g$ as suitable suspensions, we restrict \eqref{eq:rot-ax} in order to obtain the cofiber sequences
\[
\xymatrix{
x\ar[r]^-f\ar[d]\ar@{}[rd]|{\square}&y\ar[r]\ar[d]_-g\ar@{}[rd]|{\square}&0_2\ar[d]&&
y\ar[r]^-g\ar[d]\ar@{}[rd]|{\square}&z\ar[r]\ar[d]_-h\ar@{}[rd]|{\square}&0_3\ar[d]\\
0_1\ar[r]&z\ar[r]_-h&x'&&
0_2\ar[r]&x'\ar[r]_-{f'}&y',
}
\]
paste the respective bicartesian squares in order to obtain the suspension squares
\[
\xymatrix{
x\ar[r]\ar[d]\ar@{}[rd]|{\square}&0_2\ar[d]&&y\ar[r]\ar[d]\ar@{}[rd]|{\square}&0_3\ar[d]\\
0_1\ar[r]&x',&&0_2\ar[r]&y',
}
\]
and deduce from these the resulting identifications
\[
\varphi_2\colon \Sigma x\toiso x'\qquad\text{and}\qquad\psi_2\colon \Sigma y\toiso y'.
\]
Now, the suspension squares giving rise to the identifications $\varphi_1$ and $\varphi_2$ are literally the same and we deduce $\varphi_1=\varphi_2$. In contrast to this, the suspension squares inducing $\psi_1$ and $\psi_2$ differ by a restriction along the symmetry exchanging $(1,0)$ and $(0,1)$, and by \autoref{rmk:add} it follows that $\psi_1$ agrees with the \emph{negative} of $\psi_2$. Thus, in order to write the third morphism in the distinguished triangle associated to $g$ as a morphism $\Sigma x\to \Sigma y$, invoking \eqref{eq:f'-1} we are led to
\begin{equation}\label{eq:f'-2}
\vcenter{
\xymatrix{
\Sigma x\ar[r]^-{\varphi_1}_-\sim\ar[d]_-{-\Sigma f}&x'\ar[d]^-{f'}\\
\Sigma y\ar[r]_-{\psi_2}^-\sim&y'.
}
}
\end{equation}
Consequently, the distinguished triangle associated to $g$ takes the form
\[
y\stackrel{g}{\to} z\stackrel{h}{\to} \Sigma x\stackrel{-\Sigma f}{\to}\Sigma y,
\]
as asked for by the rotation axiom of triangulated categories. 
\end{proof}

\subsection{Exponentials of derivators}
\label{subsec:exponentials}

In this short subsection we discuss the construction of derivators of representations or exponentials in more categorical terminology. This is a derivatorish version of diagram categories and these exponentials govern the calculus of parametrized limits and parametrized Kan extensions. The formation of represented derivators, of derivators of abelian categories and of homotopy derivators are compatible with the passage to exponentials.

Let \D be a derivator and let $A\in\cCat$. Taking the philosophy of derivators serious, we should not be happy with a \emph{category} $\D(A)$ of coherent diagrams of shape $A$ in $\D$, but we should rather ask for a \emph{derivator} $\D^A$ of such diagrams. The heuristics are as follows: diagrams of shape $B$ in $\D^A$ are diagrams of shape $A$ in diagrams of shape $B$ in $\D$ which by the exponential law should be the same as diagrams of shape $A\times B$ in $\D$. This suggests the following construction.

\begin{con}\label{con:shift}
Let \D be a prederivator and let $A\in\cCat$. We denote by $\D^A$ the $2$-functor
\[
\D^A\colon\cCat\op\to\cCAT\colon B\mapsto \D(A\times B).
\]
This is the prederivator of diagrams of shape $A$ in $\D$. The underlying category is $\D^A(\ast)=\D(A\times\ast)\cong\D(A)$.
\end{con}

The following example illustrates the relation to representation theory. In that example we already invoke the notion of a morphism of (pre-)derivators which we introduce formally in \S\ref{subsec:mor}.

\begin{eg}\label{eg:shift-field}
Let $k$ be a field, let $Q$ be a quiver with finitely many vertices only, and let $kQ$ be the $k$-linear path algebra. Identifying the quiver $Q$ with the free category $Q$ generated by it, we want to describe the prederivator
\[
\D_k^Q\colon\cCat\op\to\cCAT\colon B\mapsto \D_k(Q\times B)
\]
differently. For this purpose, given an abelian category $\cA$, let us write $\Ho(\nCh(\cA))$ for the localization of $\nCh(\cA)$ at the class of quasi-isomorphisms, so that we have $D(\cA)=\Ho(\nCh(\cA))$. With this notation, by means of the exponential law and the equivalence $\Mod{k}^Q\simeq\Mod{kQ}$ there is the following chain of equivalences of categories
\begin{align*}
\D_k^Q(B)&=\D_k(Q\times B)\\
&=\Ho(\nCh(k)^{Q\times B})\\
&\cong\Ho\big((\nCh(k)^Q)^B\big)\\
&\simeq\Ho\big(\nCh(kQ)^B\big)\\
&=\D_{kQ}(B).
\end{align*}
This holds for every $B\in\cCat$ in a compatible way, and we hence conclude that there is an equivalence of derivators
\[
\D_k^Q\simeq\D_{kQ}.
\]
Thus the formation of shifted derivators models the passage to the path algebra, and there are of course variants for incidence algebras, group algebras, and more general category algebras. 

\end{eg}

In the above example, the prederivator $\D_k^Q$ of $Q$-shaped diagrams in $\D_k$ is again a derivator. Let us recall from classical category theory that functor categories $\cC^A$ inherit ``exactness properties'' from $\cC$ by means of pointwise constructions. The following is a derivatorish version of this.

\begin{thm}\label{thm:shift}
Let $\D$ be a derivator and let $A\in\cCat.$ The prederivator $\D^A$ is a derivator, the \textbf{derivator of $A$-shaped diagrams in $\D$}. Moreover, if $\D$ is pointed, stable, or strong, then $\D^A$ is pointed, stable, or strong, respectively.
\end{thm}
\begin{proof}
Most of the axioms of a derivator are straightforward. In order to show that $\D^A$ also satisfies (Der4), some basics on the calculus of homotopy exact squares are needed \cite[Thm.~1.25]{groth:ptstab}. The fact that $\D^A$ is strong if this is the case for $\D$ is immediate from \autoref{defn:strong}. Moreover, by \autoref{rmk:pointed} the property of being pointed is inherited as well, and for the remaining case of stable derivators we refer to \cite[Prop.~4.3]{groth:ptstab}.
\end{proof}

We also refer to these derivators $\D^A$ as \textbf{exponentials} or \textbf{shifted derivators}.

\begin{rmk}
Let us expand a bit on the calculus of Kan extensions in shifted derivators. Let \D be a derivator, let $A\in\cCat$, and let $u\colon B\to B'$ be a functor. The proof of \autoref{thm:shift} shows that left Kan extension along $u$ in $\D^A$ is given by left Kan extension along $\id_A\times u$ in \D,
\[
u_!^{\D^A}=(\id_A\times u)_!^{\D}\colon\D(A\times B)\to\D(A\times B').
\]
Moreover, Kan extensions and restrictions in unrelated variables commute (\cite[Prop.~2.5]{groth:ptstab}). This specializes to the fact that for every $a\in A$ the following square
\[
\xymatrix{
\D(A\times B)\ar[r]^-{(\id\times u)_!}\ar[d]_-{(a\times\id)^\ast}\ar@{}[rd]|{\cong}&\D(A\times B')\ar[d]^-{(a\times\id)^\ast}\\
\D(B)\ar[r]_-{u_!}&\D(B')
}
\]
commutes up to a canonical isomorphism. In particular, for every $X\in\D(A\times B)$ there is a canonical isomorphism
\[
\colim^\D_B (a^\ast X)\toiso a^\ast\colim^{\D^A}_B X
\]
in $\D(\ast)$. In this precise sense (co)limits and Kan extensions in $\D^A$ treat the $A$-variable as parameters.
\end{rmk}

\begin{rmk}
This closure under the formation of exponentials is one of the technical advantages of derivators over triangulated categories. Recall that, in general, for a triangulated category $\cT$ and a small category $A$ there is no natural triangulation on the functor category $\cT^A$ (which, for instance, is compatible with all evaluations). And even if there would be such a triangulation, in most examples, the category $\cT^A$ is not the one we would like to study (for instance, for a field $k$ and a quiver $Q$, we are interested in $D(kQ)$ which is different from $D(k)^Q$, the category of $\lZ$-graded representations of $Q$).

Similarly to derivators, $\infty$-categories also enjoy this closure property \cite[Prop.~1.2.7.3]{HTT}, while for arbitrary Quillen model category the situation is somewhat more subtle (related to this see \cite{bousfield-kan:htpy-limits,heller:htpythies,jardine:presheaves} for early special cases, and \cite[Thm.~11.6.1]{hirschhorn:model} or \cite[\S A2.8]{HTT} for published general results).
\end{rmk}

\begin{egs}
Let $B$ be a small category.
\begin{enumerate}
\item For every complete and cocomplete category $\cC$ the derivators $y_\cC^B$ and $y_{\cC^B}$ are equivalent. In these derivators the abstract calculus of limits and Kan extensions agrees with the classical parametrized versions of limits and Kan extensions (\cite[\S V.3]{maclane}).
\item For every Grothendieck abelian category $\cA$ there is an equivalence of derivators
\[
\D_\cA^B\simeq\D_{\cA^B}.
\]
\item For every model category $\cM$ there is an equivalene of derivators
\[
\ho_{\cM^B}\simeq\ho_\cM^B.
\]
\end{enumerate}
\end{egs}

There is a similar variant for $\infty$-categories.

\begin{rmk}
Let $\D$ be a derivator and $A\in\cCat$.
\begin{enumerate}
\item The underlying category of $\D^A$ is $\D(A)$, which is to say that every stage of a derivator is the underlying category of a derivator. In particular, this often allows us to focus on the underlying category when we want to establish certain properties for arbitrary stages. For instance, the class of strong and stable derivators is closed under exponentials (\autoref{thm:shift}), and we conclude by \autoref{thm:stable-tria} that all stages of such derivators admit canonical triangulations. These triangulations are compatible with each other as it follows from the discussion of morphisms in \S\ref{subsec:mor}.
\item The calculus of limits and Kan extensions in $\D^A$ is the calculus of \emph{parametrized} limits and Kan extensions in $\D$. Implicitly, this parametrized calculus already appeared in proof of the rotation axiom. In fact, the vertical coherent morphism $f'$ in \eqref{eq:rot-ax} really is the value of the parametrized suspension functor $\Sigma\colon\D([1])\to\D([1])$ at $f\in\D([1])$. In fact, by \cite[Lem.~5.13]{gps:mayer} there is a natural isomorphism
\begin{equation}\label{eq:cof-cube}
\Sigma\cong\cof^3\colon\D([1])\to\D([1]).
\end{equation} 
\end{enumerate}
\end{rmk}

\subsection{Morphisms of derivators}
\label{subsec:mor}

In this subsection we briefly discuss morphisms of derivators, the interaction of morphisms with limits and the relation to exact functors of triangulated categories. We also introduce natural transformations between morphisms of derivators and the related notions of adjunctions and equivalences of derivators. To begin with, there is the following definition.

\begin{defn}
A \textbf{morphism} of derivators is a pseudo-natural transformation.
\end{defn}

Let us unravel this definition. For derivators \D and \E, a morphism $F\colon\D\to\E$ consists of 
\begin{enumerate}
\item functors $F_A\colon\D(A)\to\E(A)$ for $A\in\cCat$ and
\item natural isomorphisms $\gamma_u\colon u^\ast F_B\toiso F_A u^\ast$ for $u\colon A\to B,$
\[
\xymatrix{
\D(A)\ar[r]^-{F_A}&\E(A)\dltwocell\omit{\cong}\\
\D(B)\ar[u]^-{u^\ast}\ar[r]_-{F_B}&\E(B),\ar[u]_-{u^\ast}
}
\] 
\end{enumerate}
satisfiying certain coherence axioms. For instance, for every pair of composable functors $u\colon A\to B$ and $v\colon B\to C$ the natural isomorphism $\gamma_{vu}$ agrees with the following pasting of $\gamma_u$ and $\gamma_v$,
\[
\xymatrix{
\D(A)\ar[r]^-{F_A}&\E(A)\dltwocell\omit{\cong}\\
\D(B)\ar[u]^-{u^\ast}\ar[r]_-{F_B}&\E(B),\dltwocell\omit{\cong}\ar[u]_-{u^\ast}\\
\D(C)\ar[u]^-{v^\ast}\ar[r]_-{F_C}&\E(C).\ar[u]_-{v^\ast}
}
\] 
And there is the axiom $\gamma_{\id_A}=\id$ and one additional axiom concerning the interaction of $\gamma_u$ and $\gamma_v$ with natural transformations $\alpha\colon u\to v$. A morphism is \textbf{strict} if all components $\gamma_u$ are identities, which is to say that the functors $F_A$ commute with restrictions on the nose (and not only up specified, coherent isomorphisms).

\begin{egs}\label{egs:mor}
Let us take up again some of \autoref{egs:der}.
\begin{enumerate}
\item Every functor $F\colon\cC\to\cD$ between complete and cocomplete categories induces a strict morphism of represented derivators
\[
F=y_F\colon y_\cC\to y_\cD.
\]
In fact, for every $A\in\cCat$, we define $F_A$ as the postcomposition functor induced by $F$,
\[
F_A\colon\cC^A\to\cD^A\colon X\mapsto F\circ X,
\]
and it is straightforward to verify the axioms of a strict morphism.
\item Every left Quillen functor $F\colon\cM\to\cN$ of model categories induces a total left derived morphism of homotopy derivators
\[
LF\colon\ho_\cM\to\ho_\cN;
\]
see \cite{kahn-maltsiniotis} and \cite[\S3]{cisinski:derivable}. The explicit construction of the components $(LF)_A$ is given by
\[
(LF)_A=L(F_A)\colon\Ho(\cM^A)\to\Ho(\cN^A).
\]
The passage to total left derived functors is pseudo-natural \cite[\S1.4]{hovey:model}, and this yields the desired pseudo-naturality constraints $\gamma_u$. There is, of course, a similar result for right Quillen functors, and these results yield extrinsic constructions of many interesting examples of morphisms of derivators. This includes derived tensor products or derived hom functors, and we refer the reader to the literature for plenty of examples.
\item A relevant intrinsic example is given by restriction morphisms. In fact, for every derivator \D and every $u\colon A\to B$ there is an induced \textbf{restriction morphism}
\[
u^\ast\colon\D^B\to\D^A.
\]
For every $C\in\cCat$ the corresponding component is defined as
\[
u^\ast_C=(u\times\id_C)^\ast\colon\D(B\times C)\to\D(A\times C).
\]
These functors define a strict morphism $u^\ast$ as it follows immediately from $2$-functoriality of $\D$. (There are also Kan extension morphisms as we will see in a bit.)
\end{enumerate}
\end{egs}

\begin{defn}\label{defn:ccts}
A morphism of derivators $F\colon\D\to\E$ preserves colimits of shape $A$ if for every $X\in\D(A)$ the canonical map
\[
\colim_A FX\toiso F\colim_A X
\]
is an isomorphism in $\E(\ast)$.
\end{defn}

These canonical maps are further instances of the canonical mates mentioned in \autoref{rmk:cocart}. In all detail, these are morphisms of the form
\[
\colim_A^\E F_A(X)\toiso F_\ast\colim^\D_A(X),
\]
where $F_\ast\colon\D(\ast)\to\E(\ast)$ is the underlying functor of $F$. In represented derivators this notion reduces to the usual notion of preservation of colimits, while in homotopy derivators of abelian categories or model categories this captures the preservation of derived colimits or homotopy colimits, respectively. Of course, there are refined notions such as the preservation of the $A$-shaped colimit of a fixed diagram only.

\begin{egs}\label{egs:classes-of-mor}
The following classes of morphisms are important.
\begin{enumerate}
\item A morphism of pointed derivators is \textbf{pointed} if and only if it preserves zero object (in the usual sense).
\item A morphism of derivators is \textbf{right exact} if it preserves initial objects and pushouts, i.e., colimits of shape $\ulcorner$. There is the dual notion of a \textbf{left exact morphism} and the combined notion of an \textbf{exact morphism}.
\item Finally, a morphism is \textbf{cocontinuous} if it preserves all colimits and \textbf{continuous} if it preserves all limits.
\end{enumerate}
\end{egs}

\begin{lem}\label{lem:exact}
A morphism of stable derivators is left exact if and only if it is right exact if and only if it is exact.
\end{lem}
\begin{proof}
The key is the following elementary observation: a morphism of derivators preserves colimits of shape $A$ if and only if it is preserves colimiting cocones of shape $A$ \cite[Prop.~3.9]{groth:revisit}. As a special case this implies that a morphism of derivators preserves pushouts if and only if it is preserves cocartesian squares, and the claims follow by definition of stability.
\end{proof}

Exact morphisms are the good notion of morphisms of stable derivators. To provide some first evidence for this claim, we include the following discussion. Recall that in ordinary category theory all colimits can be constructed from coproducts and coequalizers \cite[\S V.2]{maclane}. More relevant to our current discussion is a variant of this for finite colimits. To begin with, let us recall that a category is finite if it has finitely many objects and morphisms only. \emph{Finite colimits} can be constructed from finite coproducts and coequalizers or, alternatively, from initial objects and pushouts. Similarly, a functor preserves finite colimits if and only if it is right exact, i.e., it preserves initial objects and pushouts.

A variant of these results also holds for derivators, but in this case we have to invoke a more restrictive notion of finiteness. Suggestively, this is due to the fact that in derivator land colimits also include derived colimits and homotopy colimits.

\begin{defn}
A category $A$ is \textbf{strictly homotopy finite} if $A$ is finite, if every endomorphism $f\colon x\to x$ in $A$ is equal to $\id_x$, and if isomorphic objects in $A$ are equal ($A$ is skeletal). A category is \textbf{homotopy finite} if it is equivalent to a strictly homotopy finite category.
\end{defn}

The following result is due to Ponto--Shulman. There is also a variant for the construction of colimits, but here we only formulate the result for morphisms.

\begin{thm}\label{thm:rex}
A morphism of derivators is right exact if and only if it preserves homotopy finite colimits.
\end{thm}
\begin{proof}
This is \cite[Thm.~7.1]{ps:linearity}.
\end{proof}

\begin{cor}\label{cor:ex-mor-st}
A morphism of stable derivators is exact if and only if it preserves homotopy finite limits and homotopy finite colimits.
\end{cor}
\begin{proof}
This is immediate from \autoref{thm:rex} and \autoref{lem:exact}.
\end{proof}

\begin{rmk}\label{rmk:ex-vs-Kan}
Generalizing \autoref{defn:ccts}, a morphism of derivators $F\colon\D\to\E$ preserves left Kan extensions along $u\colon A\to B$ in $\cCat$ if for every $X\in\D(A)$ the canonical morphism
\[
u_! F(X)\toiso Fu_!(X)
\]
is an isomorphism in $\E(B)$. Since Kan extensions are pointwise, one checks that a morphism is cocontinuous if and only if it preserves all left Kan extensions \cite[Prop.~2.3]{groth:ptstab}. As a variant, a morphism of stable derivators is exact if and only if it preserves all sufficiently finite left and right Kan extensions (see \cite[Thm.~9.14]{groth:revisit} for more details).
\end{rmk}

In order to offer a second justification that exact morphisms are the good morphisms for stable derivators we include the following discussion. Since our main focus is on stable derivators, in the following construction we immediately specialize to this context.

\begin{con}
Let $G\colon\D\to\E$ be a pointed morphism of stable derivators.
\begin{enumerate}
\item For every $x\in\D(\ast)$ there is the defining suspension square $X\in\D(\square)$ on the left in 
\[
\xymatrix{
x\ar[r]\ar[d]&0\ar[d] &&Gx\ar[r]\ar[d] & 0\ar[d]\\
0\ar[r]&\Sigma x\pushoutcorner, && 0\ar[r]& G\Sigma x.
}
\]
Since $G$ preserves zero objects, the image $GX\in\E(\square)$ looks like the above diagram on the right (here, we invoke that morphisms of derivators commute with evaluation functors as special cases of restriction functors). In general, $GX$ will not be cocartesian. Hence, a comparison of this square against the suspension square of $Gx$ yields a canonical comparison map
\[
\psi\colon \Sigma G x\to G\Sigma x.
\]
This comparison map is an isomorphism if and only if $GX$ is cocartesian, which is certainly the case if $G$ is not only pointed but even exact. In that case, we can consider its inverse
\begin{equation}\label{eq:exact-str}
\varphi=\psi^{-1}\colon G\Sigma x\toiso \Sigma Gx
\end{equation}
which allows us to `pull out the suspension'.
\item As a variant of the previous case, let $(f\colon x\to y)\in\D([1])$ be a morphism and let $X\in\D(\square)$ be the cofiber square associated to it,  
\[
\xymatrix{
x\ar[r]^-f\ar[d]&y\ar[d] &&Gx\ar[r]^-{Gf}\ar[d] & Gy\ar[d]\\
0\ar[r]&C f\pushoutcorner, && 0\ar[r]& GCf.
}
\]
The image square $GX\in\E(\square)$ again vanishes at the lower left corner, but, in general, it fails to be cocartesian. This time this leads to a canonical comparison map
\[
\psi\colon C G f\to G C x,
\]
which is invertible if $G$ is exact.
\end{enumerate}
\end{con}

The details of these constructions are carried out in \cite[\S8]{groth:revisit}, and the connection to exact functors of triangulated categories is provided by the following results. Therein, the underlying categories of strong, stable derivators are of course endowed with the canonical triangulations established in \autoref{thm:stable-tria}.

\begin{prop}\label{prop:exact-exact}
Let $F\colon\D\to\E$ be an exact morphism of strong stable derivators. The natural isomorphism $\varphi\colon F\Sigma \toiso \Sigma F$ defined by \eqref{eq:exact-str} turns the underlying functor $F\colon\D(\ast)\to\E(\ast)$ into an exact functor.
\end{prop}
\begin{proof}
This is \cite[Thm.~10.6]{groth:revisit}.
\end{proof}

\begin{rmk}
The direct proof of \autoref{prop:exact-exact} is a lengthy direct calculation. There is a more systematic approach based on the notion of \emph{exact formulas in stable derivators} \cite[\S 13]{bg:cubical}. An instance of such an exact formula is the canonical isomorphism $\cof^3\cong\Sigma\colon\D([1])\to\D([1])$ from \eqref{eq:cof-cube}, and \autoref{prop:exact-exact} follows from the compatibility of exact morphisms with this formula. However, as detailed in \emph{loc.~cit.} there are many additional such exact formulas (see for instance \cite[Ex.~13.14]{bg:cubical}). 
\end{rmk}

Having discussed the notion of exact morphisms, in order to fulfill our duties in this section it only remains to talk about equivalences of derivators. This notion is, of course, defined internally in the $2$-category of derivators, and correspondingly we begin by defining natural transformations.

\begin{defn}
A natural transformation of morphisms of derivators is a modification.
\end{defn}

Again, for basics on this notion, we refer to \cite[\S2]{groth:ptstab} and \cite[Lem.~3.11]{groth:revisit}, but for convenience we unravel this definition. Let $F,G\colon\D\to\E$ be morphisms of derivators. A natural transformation $\alpha\colon F\to G$ consists of natural transformations
\[
\alpha_A\colon F_A\to G_A\colon \D(A)\to\E(A),\qquad A\in\cCat.
\]
These components $\alpha_A$ are supposed to satisfy the obvious compatibility with the pseudo-naturality constraints of $F$ and $G$.

\begin{rmk}\label{rmk:DerStex}
There is a $2$-category $\cDER$ of derivators, morphisms of derivators, and natural transformations. We denote by $\cDER_{\mathrm{St},\ex}$ the sub-$2$-category given by stable derivators, exact morphisms, and all natural transformations. (The basic calculus of the interaction of morphisms with colimits implies, for instance, that exact morphisms are closed under composition.) This $2$-category plays a key role in abstract representation theory, as we discuss in \S\ref{sec:ART}.
\end{rmk}

\begin{defn}
An \textbf{adjunction} of derivators is an adjunction internally to $\cDER$. An \textbf{equivalence} of derivators is an equivalence internally to $\cDER$.
\end{defn}

Thus, an adjunction consists of morphisms $L\colon\D\to\E$ and $R\colon\E\to\D$ together with natural transformations $\eta\colon\id\to RL$ and $\varepsilon\colon LR\to\id$ which are subject to the triangular identities. We will denote adjunctions by
\[
(L,R)\colon\D\rightleftarrows\E.
\]
There is the following result which often allows us to exhibit a given morphism as part of an adjunction.

\begin{prop}\label{prop:detect-adj}
A morphism $F\colon\D\to\E$ of derivators is a left adjoint if and only if all components $F_A\colon\D(A)\to\E(A), A\in\cCat,$ are left adjoints and the morphism $F$ is cocontinuous.
\end{prop}
\begin{proof}
This is \cite[Prop.~2.9]{groth:ptstab}, but we want to sketch one direction of the proof. Given a morphism $F\colon\D\to\E$ such that all components $F_A$ are left adjoints, we begin by choosing right adjoints $G_A$ and corresponding levelwise adjunctions 
\[
(F_A,G_A)\colon\D(A)\rightleftarrows\E(A)
\]
for all $A\in\cCat$ independently. It turns out that the cocontinuity of $F$ allows us to define pseudo-naturality constraints $u^\ast G_B\toiso G_A u^\ast$, thereby obtaining the intended adjunction $(F,G)\colon\D\rightleftarrows\E$.
\end{proof}

In order to provide some additional feeling for the notion we mention the following compatibility of the various adjunctions $(F_A,G_A)$ constituting an adjunction $(F,G)\colon\D\rightleftarrows\E$ of derivators. For every $u\colon A\to B$ in $\cCat$, $X\in\D(B)$, and $Y\in\E(B)$ the diagram
\[
\xymatrix{
\hom_{\E(B)}(FX,Y)\ar[r]^-\sim\ar[d]_-{u^\ast}&\hom_{\D(B)}(X,GY)\ar[d]^-{u^\ast}\\
\hom_{\E(A)}(u^\ast FX,u^\ast Y)\ar[d]_-\sim&\hom_{\D(A)}(u^\ast X,u^\ast GY)\ar[d]^-\sim\\
\hom_{\E(A)}(Fu^\ast X,u^\ast Y)\ar[r]_-\sim&\hom_{\D(A)}(u^\ast X,G u^\ast Y)
}
\]
commutes. 

\begin{egs}\label{egs:adjunctions}
In order to obtain key examples of adjunctions we revisit the situations considered in \autoref{egs:mor}.
\begin{enumerate}
\item A functor $F\colon\cC\to\cD$ between complete and cocomplete categories is a left adjoint if and only if the morphism $y_F\colon y_\cC\to y_\cD$ is a left adjoint. This follows from the obvious fact that the passage to represented derivators is $2$-functorial.
\item For every Quillen adjunction $(F,G)\colon\cM\rightleftarrows\cN$ there is an induced adjunction
\[
(LF,RG)\colon\ho_\cM\rightleftarrows\ho_\cN.
\]
In particular, if the model categories $\cM$ and $\cN$ are stable, then $LF$ and $RG$ are exact morphisms (by \autoref{prop:detect-adj} and \autoref{cor:ex-mor-st}), and this leads to a rich supply of exact morphisms of stable derivators.
\item Let $\D$ be a derivator, let $u\colon A\to B$ in $\cCat$, and let $u^\ast\colon\D^B\to\D^A$ be the restriction morphism. Clearly, all components $(u\times\id_C)^\ast, C\in\cCat,$ of the restriction morphism have adjoints on both sides, and the morphism is also continuous and cocontinuous. In fact, as a further application of the calculus of canonical mates, Kan extensions and restrictions in unrelated variables commute up to canonical isomorphisms \cite[Prop.~2.5]{groth:ptstab}. As an upshot, by  \autoref{prop:detect-adj} there are adjunctions of derivators
\[
(u_!,u^\ast)\colon\D^A\rightleftarrows\D^B\qquad\text{and}\qquad(u^\ast,u_\ast)\colon\D^B\rightleftarrows\D^A.
\]
The components of these \textbf{Kan extension morphisms} are the Kan extension functors $(u\times\id_C)_!$ and $(u\times\id_C)_\ast$ of \D.
\end{enumerate}
\end{egs}

\begin{cor}
Let \D be a strong, stable derivator and let $u\colon A\to B$ be in $\cCat$. The functors
\[
u^\ast\colon\D(B)\to\D(A),\quad u_!\colon\D(A)\to\D(B),\quad\text{and}\quad u_\ast\colon\D(A)\to\D(B)
\]
are exact functors with respect to the canonical triangulations.
\end{cor}
\begin{proof}
As adjoint morphisms, $u_!$ and $u^\ast$ are cocontinuous, while $u^\ast$ and $u_\ast$ are continuous (\autoref{prop:detect-adj}). Since stability is inherited by the shifted derivators $\D^A$ and $\D^B$ (\autoref{thm:shift}), all three morphisms $u^\ast, u_!,$ and $u_\ast$ are exact (\autoref{lem:exact}), and the claim hence follows from \autoref{prop:exact-exact}.
\end{proof}

There is the following more refined version of this result.

\begin{rmk}\label{rmk:2-Cat-Tria}
Let $\cTriaCAT$ be the $2$-category of triangulated categories, exact functors, and exact natural transformations. For every strong and stable derivator \D, the formation of canonical triangulations and canonical exact structures yields a lift of \D against the forgetful functor $\cTriaCAT\to\cCAT$,
\begin{equation}\label{eq:tria-lift}
\vcenter{
\xymatrix{
&\cTriaCAT\ar[d]\\
\cCat\op\ar@{-->}[ru]^-\D\ar[r]_-\D&\cCAT;
}
}
\end{equation}
see \cite[Thm.~10.14]{groth:revisit}. We want to use this result to stress once more the distinction between properties and structures. As we discussed, the property of being strong and stable implies the existence of canonical triangulations (\autoref{thm:stable-tria}). Similarly, the property of preserving certain basic (co)limits, implies the existence of canonical exact structures (\autoref{prop:exact-exact}). In particular, equivalences of derivators always are exact.

Along these lines, there is the following remark. Let us recall that an exact natural transformation $\alpha\colon F\to G$ of exact functors between triangulated categories is a natural transformation $\alpha$ which commutes with the natural isomorphisms $F\Sigma\toiso\Sigma F$ and $G\Sigma\toiso\Sigma G$. Since in derivator land these isomorphisms arise canonically, there is no counterpart for the notion of an exact natural transformation for derivators. In fact, \emph{every} natural transformation between exact morphisms is compatible with the canonical morphisms \eqref{eq:exact-str} (see \cite[Cor.~10.12]{groth:revisit}). This result is used implicitly in the construction of the lifts in \eqref{eq:tria-lift}.
\end{rmk}

\section{Higher symmetries}
%\section{Abstract representation theory}
\label{sec:ART}

In this section we give an overview over some main results of this project on higher symmetries. In \S\ref{subsec:sse} we make precise the notion of strong stable equivalences as certain uniform versions of derived equivalences for abstract representations. In \S\ref{subsec:Dynkin-A} we briefly illustrate the concept by a discussion of the abstract representation theory of Dynkin quivers of type $A$. In \S\ref{subsec:reflection} we pass to more general abstract reflection functors. In \S\ref{subsec:monoidal} we include a discussion of monoidal derivators, enriched derivators, and the universality of the derivator of spectra. Finally, in \S\ref{subsec:modules} we apply this to the construction of universal tilting modules. These are invertible, spectral bimodules that realize strong stable derivators in arbitrary stable derivators.

\subsection{Strong stable equivalences}
\label{subsec:sse}

Motivated by the compatibility of the formation of derivators of abelian categories and exponentials (\autoref{eg:shift-field}), in this subsection we define strong stable equivalences as a variant of the classical derived equivalences of quivers. 

\begin{con}\label{con:Yoneda}
Let $A$ be a small category and let $\D$ be a stable derivator. We again denote by $\D^A$ the $2$-functor constructed in \autoref{con:shift}, which by \autoref{thm:shift} is a derivator, the derivator of representations of shape $A$ with values in \D. The motivation for this terminology stems from \autoref{eg:shift-field}. It is straightforward to check that the formation $\D\mapsto\D^A$ extends to a $2$-functor
\[
(-)^A\colon\cDER_{\mathrm{St},\mathrm{ex}}\to\cDER\colon \D\mapsto \D^A,
\]
where $\cDER_{\mathrm{St},\mathrm{ex}}\subseteq\cDER$ again denotes the $2$-category of stable derivators, exact morphisms, and arbitrary transformations (\autoref{rmk:DerStex}). 
\end{con}

The slogan is that this $2$-functor encodes the abstract representation theory of the small category $A$. To fill this slogan with more life, we collect the following examples of stable derivators and specialize the shape $A$ to the case of (the path-category of) a quiver $Q$.

\begin{egs}\label{egs:stable-2}
As part of the structure encoded by the $2$-functor $(-)^Q$ of abstract representations of a quiver $Q$ there are the following stable derivators of more specific representations.
\begin{enumerate}
\item For every ordinary, not necessarily commutative ring $R$ there is the stable derivator $\D_R$ of the ring (\autoref{egs:der}). In fact, this derivator arises for instance from the projective model structure \cite[\S2.3]{hovey:model} on the category $\Ch(R)$ of unbounded chain complexes over $R$. From this we obtain by shifting the derivator $\D_R^Q$ of representations of $Q$ in $\D_R$, and there is an equivalence of stable derivators
\[
\D_R^Q\simeq\D_{RQ}.
\]
One way to see this is by observing that $\mathrm{Mod}(RQ)\simeq\mathrm{Mod}(R)^Q$ induces a Quillen equivalence at the level of unbounded chain complexes. 
\item This generalizes immediately to arbitrary Grothendieck abelian categories~$\cA$. In fact, the injective model structure on the category $\Ch(\cA)$ (see e.g.~\cite{hovey:sheaves} or \cite[Chapter~1]{HA}) induces a stable derivator $\D_\cA$. For example, for every quasi-compact, quasi-separated scheme~$X$ there is the stable derivator $\D_X$ of unbounded chain complexes of quasi-coherent $\mathcal{O}_X$-modules \cite{hovey:sheaves}. More generally, associated to every such $\cA$ there is the derivators $\D_\cA^Q$ of representations of $Q$ with values in $\D_\cA$.
\item Still sticking to the framework of classical homological algebra, this can be generalized further to exact categories in the sense of Quillen \cite[\S2]{quillen:k-theory}. At least if we are willing to restrict to suitably finite shapes, for every exact category $\cE$ there is by \cite{keller:exact}, \cite{gillespie:exact} or \cite{stovicek:exact-model} the bounded derivator $\D_\cE$ of $\cE$ enhancing the bounded derived category. Correspondingly, for finite quivers $Q$ there is the derivator $\D_\cE^Q$ of representations.
\item As additional interesting variants, given a differential-graded algebra $A$ over an arbitrary ground ring we can consider \emph{differential-graded representations} of $Q$ over $A$ which is to say functors from $Q$ to dg-modules over $A$. In order to import this to derivators, we recall that the category of dg-modules over $A$ admits suitable stable model structures (see for instance \cite{hinich:homological,schwede-shipley:algebras,fresse:modules}), and consequently we obtain the stable derivator $\D_A$ of dg-modules over $A$. For our quiver $Q$ there is an equivalence $\D_A^Q\simeq\D_{AQ}$, where $AQ$ is a differential-graded version of the usual path-algebra.
\item Another algebraic context giving rise to stable derivators is \emph{stable module theory} and \emph{representation theory of groups}. For every quasi-Frobenius ring or, more generally, Iwanaga--Gorenstein ring \cite[\S9.1]{enochs-jenda:relative} the corresponding category $\mathrm{Mod}(R)$ of modules can be endowed with the Gorenstein projective and the Gorenstein injective model structure~\cite[Theorem~8.6]{hovey:cotorsion}. These Quillen equivalent model structures are stable (see for instance~\cite[Corollary~1.1.16]{becker:models-singularity}), and hence induce up to equivalence the same stable derivator $\D^\mathrm{Gor}_R$. Correspondingly, associated to $Q$ there is the stable derivator $(\D^\mathrm{Gor}_R)^Q$ of representations. A special case occurs when $R$ is the group algebra $kG$ of a finite group $G$ over a field $k$ , and this important special case was for instance studied in~\cite{benson-rickard-carlson:thick-stmod,rickard:idemp,bik:stratification-finite-gp}.
\item Going beyond the algebraic context, we can pass to spectra in the sense of topology (see for example \cite{hss:symmetric,ekmm:rings,mmss:diagram}). For concreteness, let us stick to one of these monoidal models and assume that $E$ is a symmetric ring spectrum. Then the category of $E$-module spectra can be endowed with a stable model structure \cite{hss:symmetric}, and we obtain the associated stable derivator $\D_E$ of $E$-module spectra. Correspondingly, we obtain the derivator $\D_E^Q$ of \emph{spectral representations} of~$Q$ over~$E$. 
\item Finally, as mentioned in \S\ref{subsec:der}, typically derivators arise as shadows of the various approaches to  axiomatic homotopy theory. Hence, an entire zoo of additional examples of stable derivators is induced by the many examples of stable model categories, stable $\infty$-categories \cite{HA}, or stable cofibration categories \cite{schwede:alg-versus-top}. Among others many examples of interest arise in equivariant stable homotopy theory \cite{mandell-may:equivariant,lms:equivariant}, motivic stable homotopy theory \cite{voevodsky:a1,morel-voevodsky:a1,jardine:motivic} or parametrized stable homotopy theory \cite{may-sigurdsson:parametrized,ABGHR:units,ABG:twists}. For many more examples of stable model categories arising in various areas of algebra, geometry, and topology see \cite{schwede-shipley:morita}. 
\end{enumerate}
\end{egs}

Before defining strong stable equivalences we recall the following classical definition. 

\begin{defn}\label{def:der-equiv}
Two quivers $Q$ and $Q'$ are \textbf{derived equivalent} over a field~$k$ if the path-algebras $kQ$ and $kQ'$ are derived equivalent, i.e., if there is an exact equivalence of derived categories
\[
D(kQ)\stackrel{\Delta}{\simeq}D(kQ').
\]
\end{defn}

Such derived equivalences are usually obtained by means of tilting theory (see the handbook~\cite{angeleri-happel-krause:handbook} and the many references therein) and have been studied systematically (also for more general finite dimensional algebras over a field).

\begin{rmk}
In \autoref{def:der-equiv} we were very careful and stressed that the existence of derived equivalences potentially depends on the choice of the field or ring of coefficients. It turns out that certain derived equivalences are ``more combinatorial in nature'' and they even extend to representations with values in arbitrary abelian categories (hence they are \emph{universal derived equivalences} in the sense of Ladkani \cite{ladkani:posets,ladkani:thesis}). 
\end{rmk}

Following the line of though of Ladkani's universal derived equivalences one step further, we are led to the following definition \cite[Def.~5.1]{gst:basic}.

\begin{defn}\label{defn:sse}
Two small categories $A$ and $A'$ are \textbf{strongly stably equivalent}, in notation $A\sse A'$, if there is a pseudo-natural equivalence 
\[
\Phi\colon(-)^A\simeq (-)^{A'}\colon\cDER_{\mathrm{St},\mathrm{ex}}\to\cDER.
\]
We call such a pseudo-natural equivalence a \textbf{strong stable equivalence}.
\end{defn}

In more down-to-earth terms this means the following. Given two small categories $A$ and $A'$, a strong stable equivalence $\Phi\colon A\sse  A'$ consists of  
\begin{enumerate}
\item for \emph{every} stable derivator \D an equivalence of derivators 
\[
\Phi_\D\colon\D^A\simeq\D^{A'}
\]
\item and for every exact morphism of stable derivators $F\colon\D\to\E$ a natural isomorphism 
$\gamma_F\colon F\circ \Phi_\D\to \Phi_E\circ F,$
\[
\xymatrix{
\D^A\ar[r]^-{\Phi_\D}_-\simeq\ar[d]_-F\drtwocell\omit{\cong}&\D^{A'}\ar[d]^-F\\
\E^A\ar[r]^-\simeq_-{\Phi_{\E}}&\E^{A'}.
}
\]
\end{enumerate}
And this datum is supposed to satisfy some obvious coherence axioms. Let us comment a bit on this definition.

\begin{rmk}
The notion of strong stable equivalences is in various respects more restrictive than the notion of a derived equivalence or a universal derived equivalence. 
\begin{enumerate}
\item Strongly stably equivalent small categories have equivalent homotopy theories of representations in Grothendieck abelian categories, of differential-graded representations, of spectral representations and of more general abstract representations (by choosing specific examples in \autoref{egs:stable-2}).
\item The components $\Phi_\D$ of a strong stable equivalence are equivalences of \emph{homotopy theories} and not merely of \emph{homotopy categories} (together with the classical triangulation). This means that also the higher-order homotopy theoretic information is supposed to be preserved.
\item The various equivalences $\Phi_\D$ are suitably compatible with exact morphisms. For instance, Quillen adjunctions between stable model categories induce exact morphisms of homotopy derivators and, similarly, exact functors between stable $\infty$-categories induce exact morphisms of homotopy derivators. For strongly stably equivalent quivers or categories this implies that the equivalences commute with various kinds of restriction of scalar functors, induction and coinduction functors as well as (Bousfield) localizations and colocalizations. 
\end{enumerate}
To put it as a slogan, strongly stably equivalent have the same abstract stable representation theory. 

We also want to point the following. \autoref{defn:sse} is formulated in the language of derivators. However, for every stable equivalence $A\sse B$ we can conclude the following (related to this see \cite{renaudin}, \cite{dugger:combinatorial}, and \cite{HTT}).
\begin{enumerate}
\item For every combinatorial model category $\cM$ the diagram categories $\cM^A$ and $\cM^B$ are Quillen equivalent.
\item For every presentable $\infty$-category $\cC$ the diagram categories $\cC^A$ and $\cC^B$ are equivalent.
\end{enumerate}
\end{rmk}

In \S\S\ref{subsec:Dynkin-A}-\ref{subsec:reflection} we illustrate this notion by some interesting examples of strong stable equivalences. We conclude this section by some obvious closure properties and also obstructions.

\begin{lem}\label{lem:closure}
Let $A,A'$, $B,B'$, and $A_i,A_i',i\in I,$ be small categories.
\begin{enumerate}
\item The relation of `being strongly stably equivalent' $\sse$ defines an equivalence relation. 
\item Equivalent categories are strongly stably equivalent.
\item If $A\sse A'$ and $B\sse B'$, then $A\times B\sse A'\times B'$.
\item If $A_i\sse A_i'$ for $i\in I$, then $\bigsqcup A_i\sse \bigsqcup A_i'$.
\end{enumerate}
\end{lem}
\begin{proof}
The straightforward proof is left as an exercise.
\end{proof}

On the other hand, classical results from representation theory provide us with a non-trivial necessary condition for quivers to be strongly stably equivalent.

\begin{prop} \label{prop:strongly-equiv-necessary}
If two finite quivers without oriented cycles are strongly stably equivalent, then the underlying non-oriented graphs are isomorphic.
\end{prop}
\begin{proof}
This is \cite[Prop.~5.3]{gst:basic}.
\end{proof}

\subsection{Abstract representation theory of $A_n$-quivers}
\label{subsec:Dynkin-A}

In this subsection we illustrate the notion of a strong stable equivalence by a few examples related to Dynkin quivers of type $A$. We construct reflection functors and briefly study the related Coxeter and Serre functors in this case. This subsection is largely based on the paper \cite{gst:Dynkin-A} which is joint with Jan {\v S}{\v t}ov{\'\i}{\v c}ek.  

Let us begin by a few toy examples which make the connection to stability very obvious. The first examples are discussed in quite some detail, but later we allow ourselves to be a bit more concise.

\begin{eg}\label{eg:sse-A3-1}
The source $(\bullet\leftarrow \bullet\rightarrow\bullet)$ and the sink $(\bullet\rightarrow \bullet\leftarrow\bullet)$ of valence two are strongly stably equivalent. In fact, let \D be a stable derivator and let $X$ be an abstract representation of the source of valence two with values in \D as displayed on the left in 
\[
\xymatrix{
x\ar[r]\ar[d]&y&&x\ar[r]\ar[d]&y\ar[d]&& &y\ar[d]\\
z,& && z\ar[r]&w,&& z\ar[r]&w.
}
\]
The idea is that the strong stable equivalence is obtained by first forming the cocartesian square in the middle and then restricting it to the sink of valence two as displayed on the very right, thereby obtaining $\Phi(X)$. To formalize this idea, we recall the following two facts.
\begin{enumerate}
\item Every functor $u\colon A\to B$ between small categories induces by \autoref{egs:adjunctions} Kan extension and restriction morphisms of derivators
\[
u_!\colon\D^A\to\D^B,\quad u^\ast\colon\D^B\to\D^A,\quad\text{and}\quad u_\ast\colon\D^A\to\D^B.
\]
\item Moreover, Kan extensions along fully faithful functors (\autoref{prop:Kan-ff}) are fully faithful and hence induce equivalences on their images.
\end{enumerate}
We are interested in the special case of the fully faithful inclusions $i_\ulcorner\colon\ulcorner\to\square$ and $i_\lrcorner\colon\lrcorner\to\square$. The left Kan extension morphisms $(i_\ulcorner)_!\colon\D^\ulcorner\to\D^\square$ sends an abstract representation of the source to the corresponding cocartesian square. Denoting by $\D^{\square,\mathrm{cocart}}$ the full subderivator of $\D^\square$ consisting of the cocartesian squares, we obtain the equivalence of derivators on the left in
\[
(i_\ulcorner)_!\colon\D^\ulcorner\toiso\D^{\square,\mathrm{cocart}},\qquad (i_\lrcorner)_\ast\colon\D^\lrcorner\toiso\D^{\square,\mathrm{cart}}.
\]
Similarly, the formation of cartesian squares yields the equivalence on the right, and in both cases inverse equivalences are given by the corresponding restriction morphisms. Now, by definition of stability (\autoref{defn:stable}), the derivators $\D^{\square,\mathrm{cocart}}$ and $\D^{\square,\mathrm{cart}}$ agree, and we obtain the desired equivalence $\Phi_\D\colon\D^\ulcorner\toiso\D^\lrcorner$ as a composition of equivalences of derivators
\[
\Phi_\D=(i_\lrcorner)^\ast\circ(i_\ulcorner)_!\colon\D^\ulcorner\toiso\D^{\square,\mathrm{cocart}}=\D^{\square,\mathrm{cart}}\toiso\D^\lrcorner.
\]
Since only restrictions and sufficiently finite Kan extensions are involved in this construction, it is straightforward to verify (invoking \autoref{rmk:ex-vs-Kan}) that these equivalences are pseudo-natural with respect to exact morphisms. Consequently, we obtain the desired strong stable equivalence
\[
\Phi\colon\ulcorner=(\bullet\leftarrow \bullet\rightarrow\bullet)\sse\lrcorner=(\bullet\rightarrow \bullet\leftarrow\bullet).
\]
\end{eg}

The following example is similar, but it involves an additional homotopy finality argument.

\begin{eg}\label{eg:sse-A3-2}
The source of valence two $\ulcorner=(\bullet\leftarrow \bullet\rightarrow\bullet)$ and the linearly oriented quiver $[2]=(\bullet\rightarrow \bullet\rightarrow\bullet)$ are strongly stably equivalent. In fact, let \D be a stable derivator and let $X$ be an abstract representation of the source of valence two with values in \D which is displayed on the left in: 
\[
\xymatrix{
x\ar[r]^-g\ar[d]_-f&y&&Ff\ar[r]^-{\fib(f)}\ar[d]\ar@{}[rd]|{\square}&x\ar_-f[d]\ar[r]^-g&y&& Ff\ar[r]^-{\fib(f)}&x\ar[r]^-g&y\\
z,& && 0\ar[r]&z,&&& 
}
\]
The idea is to simply replace the morphism $f\colon x\to z$ by its fiber $\fib f\colon Ff\to x$, thereby obtaining the above representation $\Phi (X)$ on the right. In more detail, starting from either side, by means of fully faithful Kan extension morphisms we can pass to a representation as displayed in the middle. In fact, starting with our representation $X=(z\ot x\to y)$ we first add a zero object and then the cartesian square as in the diagram
\[
\xymatrix{
x\ar[r]^-g\ar[d]_-f&y&&&x\ar_-f[d]\ar[r]^-g&y&& Ff\ar[r]^-{\fib(f)}\ar[d]\ar@{}[rd]|{\square}&x\ar[d]\ar[r]^-g&y\\
z,& && 0\ar[r]&z,&&&0\ar[r]&z. 
}
\]
To re-express this in terms of Kan extensions, let $C\subseteq[1]\times[2]$ be the full subposet obtained by removing the lower right corner $(1,2)$, and let $B_1\subseteq C$ be the result of also removing the upper left corner $(0,0)$. There are the obvious full inclusions $i_1\colon\ulcorner\to B_1$ and $i_2\colon B_1\to C$ with corresponding Kan extension morphisms
\[
(i_1)_\ast\colon\D^\ulcorner\to\D^{B_1}\qquad\text{and}\qquad (i_2)_!\colon\D^{B_1}\to\D^C.
\]
By \autoref{prop:Kan-ff} both Kan extension morphisms are fully faithful and they induce equivalences onto their respective images. The morphism $(i_1)_!$ precisely amounts to adding a zero object (as a consequence of (Der4)) while $(i_2)_\ast$ adds a cartesian square (and this step invokes a simple homotopy finality argument). As an upshot we obtain an equivalence of derivators
\[
(i_2)_\ast\circ (i_1)_!\colon\D^{\ulcorner}\toiso\D^{C,\mathrm{ex_1}}
\]
where $\D^{C,\mathrm{ex_1}}\subseteq\D^C$ is the full subderivator spanned by all diagrams which vanish on the lower left corner and which make the square cartesian (the square is hence essentially a fiber square).

If we instead begin with a representation of the linearly oriented $A_3$-quiver, we simply add the cofiber square to the first of the two morphisms. Thus, in suggestive notation we carry out the constructions:
\[
\xymatrix{
w\ar[r]^-f& x\ar[r]^-g&y, & w\ar[r]^-f\ar[d]&x\ar[r]^-g&y, & w\ar[r]^-f\ar[d]\ar@{}[rd]|{\square}&x\ar[d]^-{\cof(f)}\ar[r]^-g&y\\
&& & 0&& & 0\ar[r]&Cf 
}
\]
Denoting by $B_2\subseteq C$ the full subposet obtained by removing $(1,1)$, there are full inclusion functors $j_1\colon[2]=(\bullet\rightarrow \bullet\rightarrow\bullet)\to B_2$ and $j_2\colon B_2\to C$. Arguments similar to the previous case imply that we obtain an equivalence of derivators
\[
(j_2)_!\circ (j_1)_\ast\colon\D^{[2]}\toiso\D^{C,\mathrm{ex_2}}
\]
where $\D^{C,\mathrm{ex_2}}\subseteq\D^C$ is the full subderivator spanned by all diagrams which vanish on the lower left corner and which make the square cocartesian (the square is hence a cofiber square). Now, since \D is a stable derivator, the two subderivators $\D^{C,\mathrm{ex_1}}$ and $\D^{C,\mathrm{ex_2}}$ agree, and we obtain the desired equivalence
\[
\Phi_\D=(j_1)^\ast\circ (j_2)^\ast\circ (i_2)_\ast\circ (i_1)_!\colon\D^\ulcorner\toiso\D^{[2]}.
\]
Since only restrictions and sufficiently finite Kan extensions are involved, these equivalences are pseudo-natural with respect to exact morphisms (\autoref{rmk:ex-vs-Kan}), and we obtain  the desired strong stable equivalence
\[
\Phi\colon\ulcorner=(\bullet\leftarrow \bullet\rightarrow\bullet)\sse [2]=(\bullet\rightarrow \bullet\rightarrow\bullet).
\]
\end{eg}

\begin{cor}\label{cor:sse-A3}
All $A_3$-quivers are strongly stably equivalent.
\end{cor}
\begin{proof}
We have to show that the quivers $Q_1=(1\rightarrow 2\rightarrow 3)$, $Q_2=(1\leftarrow 2\rightarrow 3)$, $Q_3=(1\rightarrow 2\leftarrow 3)$, and $Q_4=(1\leftarrow 2\leftarrow 3)$ are strongly stably equivalent. There are strong stable equivalences $Q_1\sse Q_2$ (\autoref{eg:sse-A3-2}) and $Q_2\sse Q_3$ (\autoref{eg:sse-A3-1}), and since $Q_1$ and $Q_4$ are equivalent we also deduce $Q_1\sse Q_4$ (\autoref{lem:closure}). Since~$\sse$ is an equivalence relation (\autoref{lem:closure}), we are done.
\end{proof}

There is a version of \autoref{cor:sse-A3} for longer $A_n$-quivers as well, and a proof of this essentially follows the above pattern. Given two differently oriented $A_n$-quivers $Q_1$ and $Q_2$ and a stable derivator \D, we construct a certain poset $P=P_{Q_1,Q_2}$ together with suitable combinations of fully faithful Kan extensions morphisms
\[
\D^{Q_1}\to\D^P\qquad\text{and}\qquad \D^{Q_2}\to\D^P.
\]
The stability of \D will then imply that in both cases the essential image consist of precisely the same representations of $P$ (which are determined by certain exactness conditions such as the vanishing on certain objects or the fact that certain squares are \emph{bicartesian}). In fact, in \autoref{eg:sse-A3-1} the poset $P=\square$ was enough, while in \autoref{eg:sse-A3-2} we considered the subposet $P=C\subseteq[1]\times[2]$ obtained by removing the final vertex $(1,2)$. For a direct proof of the $A_n$-version of \autoref{cor:sse-A3} we refer the reader to \cite[\S6]{gst:basic}. Here we instead prefer to present the more systematic approach as considered in \cite{gst:Dynkin-A} which also makes more precise the connection to (higher) triangulations (and hence the content of \S\ref{sec:rep-thy-tria} and the construction of canonical triangulations in~\S\ref{sec:crash}). Luckily, it turns out that for every fixed $n$ there is poset which conveniently encodes all $A_n$-quivers with arbitrary orientations, and this poset will be described in detail in the following construction.

\begin{con}
We recall that every quiver $Q$ has associated to it a repetitive quiver $\widehat{Q}$ with the following description: vertices in it are pairs $(k,q)$ with $k\in\lZ$ and $q\in Q$ while associated to every edge $\alpha\colon q_1\to q_2$ in $Q$ there are the two edges $\alpha\colon (k,q_1)\to (k,q_2)$ and $\alpha^\ast\colon (k,q_2)\to (k+1,q_1)$ in $\widehat{Q}$. Here we are only interested in the special case of the linearly oriented $A_n$-quiver $\A{n}=(1<\ldots < n)$. In the special case of $n=3$ the repetitive quiver of $\A{3}$ takes the form:
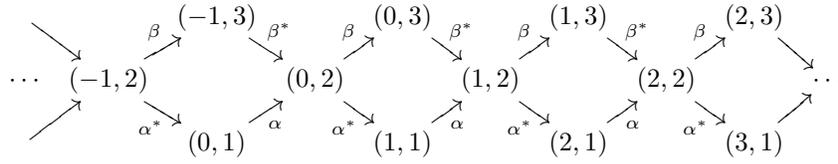
\begin{equation} \label{eq:ar-quiver}
\vcenter{
\xymatrix@R=0.8em@C=0.5em{
\ar[dr] && (-1,3) \ar[dr]^{\beta^*} && (0,3) \ar[dr]^{\beta^*} && (1,3) \ar[dr]^{\beta^*} && (2,3) \ar[dr]\\
\cdots & (-1,2) \ar[ur]^{\beta} \ar[dr]_-{\alpha^\ast} &&
(0,2) \ar[ur]^{\beta}  \ar[dr]_-{\alpha^\ast} && (1,2) \ar[ur]^{\beta}  \ar[dr]_-{\alpha^\ast} &&
(2,2) \ar[ur]^{\beta}  \ar[dr]_-{\alpha^\ast} && \cdots\\
\ar[ur] && (0,1) \ar[ur]_-\alpha && (1,1) \ar[ur]_-\alpha && (2,1) \ar[ur]_-\alpha && (3,1) \ar[ur]
}
}
\end{equation}
Let $M_{\A{n}}$ be the category which is obtained from the repetitive quiver of $\A{n}$ by making all squares commutative. By abuse of terminology, we call this poset $M_{\A{n}}$ the \textbf{mesh category}.
\end{con}

We want to show that, for stable derivators, representations of $\A{n}$ can be equivalently encoded by suitable representations of mesh categories. It is worth to compare the following to the discussion of the rotation axiom in \S\ref{subsec:BP} and of octahedral axiom in \S\ref{subsec:octa}.

\begin{con}
For every $n\geq 0$ we introduce the following short-hand notation for mesh categories: 
\[
M_n=M_{[n+1]}\quad\text{for}\quad [n+1]=(0<\ldots <n+1).
\]
Given a stable derivator \D, we denote by $\D^{M_n,\mathrm{ex}}\subseteq\D^{M_n}$ the full subderivator spanned by all representations which 
\begin{enumerate}
\item vanish on the boundary stripes (i.e., at $(k,0),(k,n+1)$ for all $k\in\lZ$)
\item and which make all squares bicartesian.
\end{enumerate}
The fact that this is a derivator is a consequence of \autoref{thm:AR} since derivators are closed under equivalences of prederivators. In order to relate $\D^{M_n,\mathrm{ex}}$ to $\D^{\A{n}}$, we note that there is the fully faithful functor
\begin{equation}
i\colon \A{n}\to M_n\colon l\mapsto (0,l) 
\label{eq:i}
\end{equation}
\end{con}

The following is a derivatorish version of \autoref{thm:BP} and \autoref{thm:octa}. In fact, the proofs of these two theorems were modeled after the proof of the following result.

\begin{thm}\label{thm:AR}
For every stable derivator \D and $n\geq 0$ restriction along $i$ \eqref{eq:i} induces an equivalence of derivators
\[
i^\ast\colon\D^{M_n,\mathrm{ex}}\toiso\D^{\A{n}}.
\]
This equivalence is pseudo-natural with respect to exact morphisms, and the inclusion $\D^{M_n,\mathrm{ex}}\to\D^{M_n}$ is exact.
\end{thm}
\begin{proof}
We only sketch the proof and refer the reader to \cite[Thm.~4.6]{gst:basic} for details. For every stable derivator \D we want to construct an inverse equivalence of $i^\ast$. Thus, given a representation $X$ of $\A{n}$ with values in $\D$ we want to obtain a coherent diagram of shape $M_n$ satisfying the defining exactness conditions of $\D^{M_n,\mathrm{ex}}$. To this end, similarly to the sketch proofs of \autoref{thm:BP} and \autoref{thm:octa}, we note that the inclusion \eqref{eq:i} factors as a composition of inclusions of full subcategories
\[
i\colon \A{n}\stackrel{i_1}{\to} K_1\stackrel{i_2}{\to} K_2\stackrel{i_3}{\to} K_3\stackrel{i_4}{\to} M_n
\]
where
\begin{enumerate}
\item $K_1$ contains all objects from $\A{n}$ and the objects $(k,n+1)$ for $k\geq 0$ and $(k,0)$ for $k>0$,
\item $K_2$ is obtained from $K_1$ by adding the objects $(k,l), k>0$, and
\item $K_3$ contains all objects from $K_2$ and the objects $(k,n+1)$ for $k<0$ and $(k,0)$ for $k\leq 0$.
\end{enumerate}
The inclusion $i_4$ thus adds the remaining objects in the negative $k$-direction. By \autoref{prop:Kan-ff}, associated to these fully faithful functors there are fully faithful Kan extension functors
\begin{equation}
\vcenter{
\xymatrix{
\D^{\A{n}}\ar[r]^-{(i_1)_\ast}&\D^{K_1}\ar[r]^-{(i_2)_!}&\D^{K_2}\ar[r]^-{(i_3)_!}&
\D^{K_3}\ar[r]^-{(i_4)_\ast}&\D^{M_n}.
}
}
\label{eq:AR}
\end{equation}
These Kan extension morphisms can be analyzed in turn, and the conclusion is that $(i_1)_\ast$ adds zero objects on the boundary stripes in the positive direction, $(i_2)_!$ adds bicartesian squares in the positive direction, $(i_3)_!$ adds zero objects in on the boundary stripes in the negative direction, and finally $(i_4)_\ast$ fills up by bicartesian squares in the negative direction. Thus the essential image agrees with $\D^{M_n,\mathrm{ex}}$, and this concludes the construction of an equivalence $F_{\A{n}}\colon\D^{\A{n}}\toiso\D^{M_n,\mathrm{ex}}$ which is inverse to $i^\ast$. Again, since we only used restrictions and sufficiently finite Kan extension morphisms, both $i^\ast$ and its inverse $F_{\A{n}}$ are pseudo-natural with respect to exact morphisms (\autoref{rmk:ex-vs-Kan}).
\end{proof}

Variants of this theorem for different orientations of $A_n$-quivers lead to the following result.

\begin{thm}\label{thm:sse-An}
Let $n\geq 0$ be fixed. All $A_n$-quivers are strongly stably equivalent.
\end{thm}
\begin{proof}
We sketch very roughly the main idea of the proof, and for this purpose we consider an arbitrarily oriented $A_n$-quiver $Q$. For every such $Q$ there are admissible embeddings $i_Q\colon Q\to M_n$ and the corresponding restriction morphisms $(i_Q)^\ast\colon\D^{M_n}\to\D^Q$ again restrict to equivalences $(i_Q)^\ast\colon\D^{M_n,\mathrm{ex}}\toiso\D^Q$. In this case it is more tricky to write down in closed form the inverse equivalence $F_Q$ (which depends on $Q$ and the choice of the admissible embedding $i_Q$). 

To illustrate this step, we content ourselves by one example, but we invite the reader to come up with additional examples. In the case of $n=3$ and the source of valence two $Q=(1\leftarrow 2\rightarrow 3)$, we can consider the admissible embedding $i_Q\colon Q\to M_3$ given by 
\[
2\mapsto (0,2),\quad 1\mapsto (1,1),\quad\text{and}\quad 2\mapsto (0,3).
\]
The corresponding equivalence $F_Q$ sends a representation $X$ looking like
\[
\xymatrix{
x\ar[r]^-g\ar[d]_-f&y\\
z&
}
\]
to a coherent diagram as in \autoref{fig:octa-recycle}. Therein, the boundary stripes are populated by zero objects and all squares are bicartesian. To put it differently, the diagram $F_Q(X)$ is a refined octahedral diagram of $(Ff\to x\to y)$ (as in \S\ref{subsec:octa}).

\begin{figure}[h]
\centering
\[
\xymatrix{
\ar@{}[dr]|{\ddots}&\ar@{}[dr]|{\ddots}&\ar@{}[dr]|{\ddots}&\ar@{}[dr]|{\ddots}&\ar@{}[dr]|{\ddots}&&&&&\\
&0\ar[r]&\bullet\ar@{}[rd]|{\square}\ar[d]\ar[r]&\bullet\ar@{}[rd]|{\square}\ar[r]\ar[d]&\bullet\ar@{}[rd]|{\square}\ar[d]\ar[r]&0\ar[d]&&&&\\
&&0\ar[r]&\bullet\ar[r]\ar[d]\ar@{}[rd]|{\square}&x\ar[r]^-g\ar[d]_-f\ar@{}[rd]|{\square}&y\ar[d]\ar[r]\ar@{}[rd]|{\square}&0\ar[d]&&&\\
&&&0\ar[r]&z\ar[r]\ar[d]\ar@{}[rd]|{\square}&\bullet\ar[r]\ar[d]\ar@{}[rd]|{\square}&\bullet\ar[d]\ar[r]\ar@{}[rd]|{\square}&0\ar[d]&&\\
&&&&0\ar[r]\ar@{}[dr]|{\ddots}&\bullet\ar[r]\ar@{}[dr]|{\ddots}&\bullet\ar[r]\ar@{}[dr]|{\ddots}&\bullet\ar[r]\ar@{}[dr]|{\ddots}&0\ar@{}[dr]|{\ddots}&\\
&&&&& & & & &
}
\]
\caption{The equivalence $F_Q$ for $Q=(\bullet\ot\bullet\to\bullet)$}
\label{fig:octa-recycle}
\end{figure}
 
Now, given an additional $A_n$-quiver $Q'$ we choose an admissible embedding $i_{Q'}\colon Q'\to M_n$, and obtain the corresponding inverse equivalence $F_{Q'}$ of $(i_{Q'})^\ast$. In order to conclude the proof it suffices to make the definition
\[
\Phi_{Q',Q}=(i_{Q'})^\ast\circ F_Q\colon\D^Q\toiso\D^{M_n,\mathrm{ex}}\toiso\D^{Q'}.
\]
As a composition of pseudo-natural equivalences, this defines a strong stable equivalence $\Phi_{Q',Q}\colon Q\sse Q'$.
\end{proof}

The strong stable equivalences constructed in \autoref{thm:sse-An} arise as finite compositions from certain basic building blocks. These building blocks admit a fairly elementary description as we discuss next, and they turn out to be special cases of the more general reflection morphisms from \S\ref{subsec:reflection}.

\begin{con}\label{con:reflection-An}
Let $n\geq 0$, let \D be a stable derivator, and let $M_n$ be the mesh category. Choosing two $A_n$-quivers $Q,Q'$ and admissible embeddings $i_Q,i_{Q'}$ of the quivers to $M_n$, the resulting strong stable equivalence $\Phi_{Q',Q}\colon\D^Q\toiso\D^{Q'}$ measures the difference between the restriction morphisms $(i_Q)^\ast$ and $(i_{Q'})^\ast$. The above-mentioned basic building blocks arise when the two embeddings ``differ by one square only''. Instead of making this precise by explicit combinatorial formulas, we illustrate this by an example related to \autoref{fig:octa-recycle}.

As in the proof of \autoref{thm:sse-An}, let $Q=(\bullet\ot\bullet\to\bullet)$ and let $i_Q\colon Q\to M_3$ be the embedding indicated by \autoref{fig:octa-recycle}. If $Q'=(\bullet\to\bullet\ot\bullet)$ is the sink of valence two, then there is an obvious embedding of $i_{Q'}\colon Q'\to M_3$ which has two objects in common with $i_Q$ but differs from it by one square. In this case, the resulting strong stable equivalence $\Phi_{Q',Q}$ agrees with the one from \autoref{eg:sse-A3-1}. Alternatively, we can consider the linearly oriented quiver $\A{3}=(\bullet\to\bullet\to\bullet)$ with the standard embedding $i\colon\A{3}\to M_3$ as in \eqref{eq:i}. Also these embeddings differ by one square only and we reproduce the strong stable equivalence from \autoref{eg:sse-A3-2}. Finally, there is the additional vertical embedding $j\colon\A{3}\to M_3$ and in this case the strong stable equivalence $Q\sse\A{3}$ forms the fiber of the morphism labeled by $g$ in \autoref{fig:octa-recycle}. 

All these three cases are special cases of reflection functors. Given an $A_n$-quiver $Q$ and a source $a\in Q$, let $\sigma_aQ$ be the $A_n$-quiver obtained by reorienting all edges adjacent to $a$ such that the vertex $a\in\sigma_aQ$ now is a sink. By \autoref{thm:sse-An} the quivers $Q$ and $Q'=\sigma_aQ$ are strongly stably equivalent, and a preferred such equivalence is constructed as follows. Let $i_Q\colon Q\to M_n$ be an admissible embedding and let $i_{Q'}=i_{\sigma_aQ}\colon Q'=\sigma_aQ\to M_n$ be the induced admissible embedding which differs from $i_Q$ by one square only. The strong stable equivalence
\[
s_a^-=\Phi_{Q',Q}=(i_{Q'})^\ast\circ F_Q\colon\D^Q\toiso\D^{Q'}
\]
and its inverse
\[
s_a^+=\Phi_{Q,Q'}=(i_Q)^\ast \circ F_{Q'}\colon\D^{Q'}\toiso\D^Q
\]
are \textbf{reflection functors}. It is easy to see that these morphisms do not depend on the choice of the embeddings $i_Q$ and $i_{Q'}$ as long as they differ by precisely one square.
\end{con}

The strong stable equivalences from \autoref{thm:sse-An} can be described as compositions of the above reflection morphisms, and by the following remark the order of the reflections does not matter.

\begin{rmk}\label{rmk:sasb-commute}
Let $Q$ be an $A_n$-quiver and let $a_1,a_2\in Q$ be two distinct sinks. In that case the vertex $a_2$ is also a sink in the reflected quiver $\sigma_{a_1}Q$ and we can hence iterate the reflection thereby obtaining $\sigma_{a_2}\sigma_{a_1} Q$. A straightforward argument shows that at the level of the quivers the order of the reflections is irrelevant and definition
\[
\sigma_{\{a_1,a_2\}}Q=\sigma_{a_1}\sigma_{a_2} Q=\sigma_{a_2}\sigma_{a_1} Q
\]
hence makes sense. Playing a bit with the corresponding embeddings, we note that for the corresponding reflection functors there are natural isomorphisms
\[
s_{a_1}^+s_{a_2}^+\cong s_{a_2}^+s_{a_1}^+\colon\D^Q\toiso\D^{\sigma_{\{a_1,a_2\}}Q}.
\]
Of course there are similar results for sinks instead of sources.
\end{rmk}

\begin{defn}
An \textbf{admissible sequence of sinks} in a finite quiver $Q$ is a total ordering $(a_1,\ldots,a_n)$ of all vertices of $Q$ such that $a_1$ is a sink in $Q$ and $a_i$ is a sink in $\sigma_{a_{i-1}}\ldots\sigma_{a_1}Q$ for all $2\leq i\leq n$. \textbf{Admissible sequences of sources} are defined dually.
\end{defn}

It turns out that every finite, acyclic quiver admits an admissible sequence of sources and sinks \cite[Lemma 1.2(1)]{bernstein-gelfand-ponomarev:Coxeter}. Moreover, one can show that the quiver $\sigma_{a_n}\ldots\sigma_{a_1}Q$ agrees with the original quiver $Q$ and similarly in the case of sources. In what follows we only need these results for Dynkin quivers of type $A$, in which case the arguments are more straightforward. While admissible sequences always exist, they are by no means unique as simple examples show. However, at the level of iterated reflection functors there is the following result.

\begin{con}\label{con:indep-admis}
Let $n\geq 0$ and let $Q$ be an $A_n$-quiver. For every stable derivator \D and every admissible sequence of sinks $(a_1,\ldots,a_n)$ in $Q$ there is the iterated reflection functor
\[
\Phi^+_{(a_1,\ldots,a_n)}=s^+_{a_n}\circ\ldots\circ s^+_{a_1}\colon\D^Q\to\D^Q.
\]
A combination of \autoref{rmk:sasb-commute} with some combinatorial arguments imply that up to natural isomorphism this iterated reflection functor is independent of $(a_1,\ldots,a_n)$ (see the proof of~\cite[Lemma 1.2(3)]{bernstein-gelfand-ponomarev:Coxeter}). Of course, the same observation applies to iterated reflections $\Phi^-_{(b_1,\ldots,b_n)}$ at admissible sequences of sources.
\end{con}

By \autoref{con:indep-admis} the following functors are well-defined up to natural isomorphisms. 

\begin{defn} \label{defn:Coxeter}
Let \D be a stable derivator and let $Q$ be an $A_n$-quiver.
\begin{enumerate}
\item The \textbf{Coxeter functor} $\Phi^+=\Phi^+_Q\colon\D^Q\to\D^Q$ is $\Phi^+=\Phi^+_{(a_1,\ldots,a_n)}$ for some admissible sequence of sinks $(a_1,\ldots,a_n)$ in $Q$.
\item The \textbf{Coxeter functor} $\Phi^-=\Phi^-_Q\colon\D^Q\to\D^Q$ is $\Phi^-=\Phi^-_{(b_1,\ldots,b_n)}$ for some admissible sequence of sources $(b_1,\ldots,b_n)$ in $Q$.
\end{enumerate}
\end{defn}

We illustrate these reflections and Coxeter functors by an example.

\begin{eg}\label{eg:Coxeter-A3-linear}
The linearly oriented $A_3$-quiver $\A{3}=(1\to2\to3)$ has a unique admissible sequence of sinks $(3,2,1)$. Consequently, for every stable derivator \D the Coxeter functor is given by
\[
\Phi^+=s^+_1\circ s^+_2\circ s^+_3\colon\D^{\A{3}}\toiso\D^{\A{3}}.
\]
Given a representation $X=(x\stackrel{f}{\to}y\stackrel{g}{\to}z)$ in \D, to calculate $\Phi^+(X)$ we first determine $s^+_3(X)$ by considering the diagram on the left in 
\[
\xymatrix{
&Fg\ar[d]\ar[r]\ar@{}[rd]|{\square}&0\ar[d]&&
F(gf)\ar[r]\ar[d]\ar@{}[rd]|{\square}&Fg\ar[d]\ar[r]\ar@{}[rd]|{\square}&0\ar[d]\\
x\ar[r]_-f&y\ar[r]_-g&z,&&
x\ar[r]_-f&y\ar[r]_-g&z.
}
\]
In order to describe $s^+_2s^+_3(X)$ we next reflect $(x\to y\ot Fg)$ at the vertex decorated by $y$. This is achieved by extending the above diagram by one more bicartesian square. Since the pasting of the squares is again bicartesian, the upper left corner is indeed populated by $F(gf)$. As a final step, it remains to reflect the representation $(x\ot F(gf)\to Fg)$ at the vertex decorated by $x$. For this purpose we extend the diagram by one more fiber square in order to obtain
\[
\xymatrix{
\Omega z\ar[r]\ar[d]\ar[d]\ar@{}[rd]|{\square}&F(gf)\ar[r]\ar[d]\ar@{}[rd]|{\square}&Fg\ar[d]\ar[r]\ar@{}[rd]|{\square}&0\ar[d]\\
0\ar[r]&x\ar[r]_-f&y\ar[r]_-g&z.
}
\]
Again, the outer most square is bicartesian and hence a loop square, and the upper left corner consequently agrees with $\Omega z$. As an upshot we conclude that the Coxeter functor $\Phi^+$ is given by
\begin{equation}\label{eq:Coxeter-linear-A3}
\Phi^+(x\stackrel{f}{\to}y\stackrel{g}{\to}z)=\big(\Omega z\to F(gf)\to Fg\big).
\end{equation}

Similarly, the quiver $\A{2}=(1\to 2)$ has $(2,1)$ as unique admissible sequence of sinks. We leave it to the reader to check that in this simpler case the Coxeter functor $\Phi^+$ is given by
\begin{equation}\label{eq:Coxeter-A2}
\Phi^+\cong\fib^2\colon\D^{\A{2}}\toiso\D^{\A{2}}.
\end{equation}
\end{eg}

\begin{eg}\label{eg:Coxeter-span}
Let \D be a stable derivator and let $Q=(1\ot 2\to 3)$ be the source of valence two. This quiver has $(1,3,2)$ and $(3,1,2)$ as admissible sequences of sinks. In order to calculate the Coxeter functor $\Phi^+\colon\D^Q\toiso\D^Q$ we consider an arbitrary representation $X=(z\ot x\to y)$ and extend it to the diagram
\[
\xymatrix{
w\ar[r]\ar[d]\ar@{}[dr]|{\fbox{2}}&Fg\ar[r]\ar[d]\ar@{}[dr]|{\fbox{3}}&0\ar[d]\\
Ff\ar[r]\ar[d]\ar@{}[dr]|{\fbox{1}}&x\ar[r]^-g\ar[d]_-f\ar@{}[dr]|{\fbox{0}}&y\ar[d]\\
0\ar[r]&z\ar[r]&y\cup_xz
}
\]
consisting of bicartesian cubes. Ignoring the lower right square for the moment, each of the remaining three squares will be used to define a reflection at a sink and the labels match the corresponding vertices. This illustrates again the commutativity of reflections at different sinks, since independently of the order we have
\[
s^+_1s^+_3(X)\cong s^+_3s^+_1(X)\cong(Ff\to x\ot Fg).
\] 
To determine $\Phi^+(X)$ it suffices to understand the upper left corner. For this purpose we also drew the remaining auxiliary cocartesian square $\fbox{0}$, since it allows us to conclude that the total pasting of these squares is a loop square. Consequently, there is a canonical isomorphism $w\cong\Omega(y\cup_xz)$ and we obtain
\begin{equation}\label{eq:Coxeter-span}
\Phi^+(z\stackrel{f}{\ot}x\stackrel{g}{\to}y)\cong(Ff\ot\Omega(y\cup_xz)\to Fg).
\end{equation}
\end{eg}

Motivated by the classical situation (see \cite[\S I.2]{reiten-bergh:serre}) we make the following definition.

\begin{defn}\label{defn:Serre-An}
Let \D be a stable derivator and let $Q$ be an $A_n$-quiver. The \textbf{Serre functor} of $\D^Q$ is
 \[
S=\Sigma\Phi^+\cong\Phi^+\Sigma\colon\D^Q\to\D^Q.
\]
\end{defn}

In order to put this definition into context we recall the notion of a Serre functor. Serre functors have been formalized in various contexts and references include \cite{bondal-kapranov:serre,bondal-orlov:reconstruction} in the framework of $k$-linear categories, \cite{chen:Serre} in the context of $k$-linear triangulated categories, and \cite{lyubashenko-mazyuk:Serre} in the realm of (pretriangulated) $A_\infty$-categories (see also \cite{bondal-kapranov:framed,kontsevich-soibelman:Aoo,BLM:Aoo}). Here we are mainly interested in the case of sufficiently well-behaved triangulated categories (see~\cite[\S I.1]{reiten-bergh:serre} and~\cite{bondal-kapranov:serre}). 

\begin{defn}\label{defn:Serre-triangulated}
Let $k$ be a field and let $\cT$ be a $k$-linear triangulated category with finite-dimensional mapping vector spaces. An exact auto-equivalence $S\colon\cT\to\cT$ is a \textbf{Serre functor} if it is endowed with an isomorphism
\begin{equation} \label{eq:Serre}
\hom_\cT(x,y) \stackrel{\cong}\longrightarrow \big(\hom_\cT(y, S x)\big)^*
\end{equation}
which is natural in $x$ and $y$.
\end{defn}

Here, $(-)^\ast$ of course denotes the vector space duals. As a consequence of the Yoneda lemma we see that Serre functors are essentially unique if they exist, so the existence of a Serre functor is really a property of the triangulated category. Let us recall the following two typical examples of Serre functors.

\begin{egs}\label{rmk:Serre-An}
\begin{enumerate}
\item The example which is of key interest to here is given by bounded derived categories of finitely generated modules over finite dimensional algebras of finite global dimension (\cite[\S3.6]{happel:fd-algebra}, \cite{krause-le:AR}). In that case, Serre functors exist and they are given by \emph{derived Nakayama functors} \cite[\S3.6]{happel:fd-algebra}. We will come back to this explicit description in \S\ref{subsec:modules}.
\item Let $k$ be an algebraically closed field, let $X$ be a smooth, projective variety of dimension $n$ over $k$, and let $D^b(X)$ be the bounded derived category of coherent sheaves on $X$. The derived tensor product with the shifted canonical line bundle 
\[
S_X= \omega_X[n]\otimes -\colon  D^b(X)\toiso D^b(X)
\]
is a Serre equivalence on $D^b(X)$ \cite[Thm.~3.12]{huybrechts:fourier}. Relevant classical references include \cite{bondal-kapranov:serre,bondal-orlov:reconstruction,serre:duality} and also \cite{mukai:duality,huybrechts:fourier}.
\end{enumerate}
\end{egs}

Let us now return to the situation in abstract representation theory, and let us revisit \autoref{eg:Coxeter-A3-linear} and \autoref{eg:Coxeter-span}.

\begin{egs}\label{egs:Serre-An}
There are the following examples of Serre functors.
\begin{enumerate}
\item The Serre functor of $\A{2}$ is $S\cong\cof\colon\D^{\A{2}}\to\D^{\A{2}}$ as it easily follows from \eqref{eq:Coxeter-A2} and \eqref{eq:cof-cube}.
\item In the case of the linearly oriented $A_3$-quiver $\A{3}$ it follows from \eqref{eq:Coxeter-linear-A3} that the Serre functor is given by
\[
S(x\stackrel{f}{\to}y\stackrel{g}{\to}z)=\big(z\to C(gf)\to Cg\big).
\]
\item Similarly, in the case of $Q=(1\ot 2\to 3)$ we can recycle \eqref{eq:Coxeter-span} in order to conclude that the Serre functor has the description
\[
S(z\stackrel{f}{\ot}x\stackrel{g}{\to}y)\cong(Cf\ot y\cup_xz\to Cg).
\]
\end{enumerate}
\end{egs}

\begin{rmk}
The justification for the terminology Serre functors in abstract representation theory is as follows. Specialized to the derivator $\D_k$ of a field, these Serre functors induce on underlying categories the classical Serre functors from representation theory. For an additional example of Serre functors in abstract representation we refer to \cite{bg:cubical}. Jointly with Falk Beckert, we intend to study these abstract Serre functors more systematically elsewhere. 
\end{rmk}

For more details on the abstract representation theory of Dynkin quivers of type $A$ we refer to \cite{gst:Dynkin-A}. In particular, therein is also a discussion of an abstract fractional Calabi--Yau property \cite{keller:orbit,keller:CY-triang,keller-scherotzke:integral,roosmalen:CY} which plays a key role in the context of spectral Picard groups. These Picard groups will be briefly discussed in \autoref{rmk:Picard}. We conclude this subsection by a remark on \emph{higher triangulations}.

\begin{rmk}\label{rmk:higher-triangulations}
Despite the fact that triangulated categories are extremely useful in plenty of situations, from the early days on people were pointing out defects of the theory (see for instance the introduction to the early \cite{heller:shc}). Besides the enhancements in the sense of stable model categories, stable $\infty$-categories or differential-graded categories, there are the attempts to fix some of the deficiencies of triangulated categories by also keeping track of \emph{higher triangles}. The idea of considering higher triangles already appears in \cite[Remark~1.1.14]{beilinson:perverse}, while a more recent definition of strong triangulations is in \cite{maltsiniotis:higher} (see also the closely related notion in \cite{balmer:separability} as well as the thesis of Beckert \cite{beckert:thesis}). 

We sketch the rough idea and consider an additive category $\cA$ together with an auto-equivalence $\Sigma\colon\cA\toiso\cA$. An \emph{$n$-angle} in $(\cA,\Sigma)$ is essentially an ordinary diagram $X\colon M_n\to\cA$ which vanishes on the boundary stripes and such that the triangular fundamental domains match up to suitable powers of the suspension $\Sigma$. Thus, $2$-angles correspond to the usual triangles in the sense of Puppe and Verdier, while $3$-angles are closely related to octahedral diagrams; we refer to \cite{maltsiniotis:higher} for more details. A \emph{strong triangulation} or \emph{\infty-triangulation} on $(\cA,\Sigma)$ is given by classes of distinguished $n$-angles for $n\geq 2$ subject to various axioms. By \cite[Thm.~13.6]{gst:Dynkin-A} the values of a strong, stable derivator \D admit canonical strong triangulations. In fact, analogously to \autoref{rmk:2-Cat-Tria}, there is a lift of \D to the $2$-category of strong triangulated categories (as follows from the details in \cite[\S13]{gst:Dynkin-A} and \cite{groth:revisit}). Moreover, the coherent versions of these distinguished $n$-triangles (thus objects in $\D^{M_n,\mathrm{ex}}$) in combination with the Mayer--Vietoris triangles from \cite{gps:mayer} give rise to many additional triangles in the underlying category $\D(\ast)$.
\end{rmk}

\subsection{Reflection functors}
\label{subsec:reflection}

In this subsection we discuss additional examples of strong stable equivalences given by reflection functors. Happel \cite[\S1.7]{happel:fd-algebra} showed that for every finite, acyclic quiver reflection functors at sources or sinks induce derived equivalences between the corresponding category algebras. In \cite{gst:acyclic} these derived equivalences were shown to be shadows of corresponding strong stable equivalences (even without the finiteness or acyclicity assumption). We illustrate these reflection functors by some important special cases. 

Let us begin this subsection by the following obstruction against the existence of strong stable equivalence.

\begin{prop}
If the finite acyclic quivers $Q$ and $Q'$ are strongly stably equivalent, then the underlying unoriented graphs are isomorphic.\end{prop}
\begin{proof}
This is \cite[Prop.~5.3]{gst:basic} which in turn relies essentially on \cite{happel:fd-algebra}.
\end{proof}

Thus, in the context of finite acyclic quivers all that can happen by strong stable equivalences are reorientations of some of the edges. These can not be carried out arbitrary as the case of oriented cycles shows (see \autoref{cor:sse-cycle}). However, a valid reorientation is the classical reflection of (acyclic) quivers at sources and sinks (see, for instance, \cite[\S VII.5]{assem-simson-skowronski}, or \cite{krause:reflection} for more detail). Let us recall this construction.

\begin{con}\label{con:reflect-quiver}
Let $Q$ be a quiver and let $q\in Q$. The vertex $q$ is a \textbf{source} if there are no edges in $Q$ with $q$ as target, and there is dual notion of a \textbf{sink} of a quiver. Given a source $q$ in a quiver, the \textbf{reflected quiver} $\sigma_qQ$ is obtained from $Q$ by turning the source into a sink. Thus, both quivers have the same vertices and their is a bijection between the edges. The only difference is that the orientations of all edges in $Q$ that start at the source $q\in Q$ have been reversed in $\sigma_qQ$ so that they now end at $q\in\sigma_qQ$.
\end{con}

Special cases of such reflections already occurred in \S\ref{subsec:Dynkin-A}. For instance, in the case of $\A{3}$-quivers, \autoref{eg:sse-A3-1} and \autoref{eg:sse-A3-2} provide examples of strong stable equivalences 
\[
(1\ot 2\to 3)\sse \sigma_2(1\ot 2\to 3)=(1\to 2\ot 3)
\]
and
\[
(1\ot 2\to 3)\sse \sigma_1(1\ot 2\to 3)=(1\to 2\to 3).
\]
These reflections at the level of quivers are accompanied by the following reflection functors.

\begin{con}\label{con:reflection}
Let $Q$ be a finite quiver, let $q_0\in Q$ be a source, and let $k$ be a field. For every representation $M\colon Q\to\mathrm{Mod}(k)$ we obtain an induced representation $s^-_{q_0}M\colon \sigma_{q_0}Q\to\mathrm{Mod}(k)$ of the reflected quiver as follows. For all vertices $q'\neq q_0\in\sigma_{q_0}Q$ we set $(s^-_{q_0}M)_{q'}=M_{q'}$. To describe the remaining component, let $q_1,\ldots, q_n$ denote the finitely many vertices of $Q$ such that there are edges $q_0\to q_i$ for $i=1,\ldots,n$. Then we define
\[
(s^-_{q_0}M)_{q_0}=\cok(M_{q_0}\to M_{q_1}\oplus\ldots\oplus M_{q_n}),
\]
and the canonical homomorphisms yield the required maps $(s^-_{q_0}M)_{q_i}\to (s^-_{q_0}M)_{q_0}$. This construction is clearly functorial in $M$, and we hence obtain the \textbf{reflection functor} $s^-_{q_0}\colon\mathrm{Mod}(kQ)\to\mathrm{Mod}(k\sigma_{q_0}Q)$ associated to $Q$ and the source $q_0$. Dually, if $q_0$ is a sink, then there is a reflection functor $s^+_{q_0}\colon\mathrm{Mod}(kQ)\to\mathrm{Mod}(k\sigma_{q_0}Q)$.
\end{con}

At the level of abelian categories of modules these reflection functors fail to be equivalences (see, for example, \cite[Thm.~VII.5.3]{assem-simson-skowronski}), but there is the following positive derived version of it.

\begin{thm}\label{thm:reflection-Happel}
Let $k$ be a field, let $Q$ be a finite acyclic quiver, let $q_0\in Q$ be a source, and let $Q'=\sigma_{q_0}Q$ be the reflected quiver. The reflection functors induce an exact equivalence of bounded derived categories
\[
(Ls^-_{q_0},Rs^+_{q_0})\colon D^b(kQ)\stackrel{\Delta}{\simeq} D^b(kQ').
\]
\end{thm}
\begin{proof}
This is due to Happel and can be found in \cite[\S1.7]{happel:fd-algebra}.
\end{proof}

It can be shown that these reflection functors extend to abstract stable homotopy theories, and by now there are two approaches to such a generalization. Let us recall from \autoref{con:reflect-quiver} and \autoref{con:reflection}, that these constructions rely on some ``local modifications at the level of shapes and representations''. Correspondingly, the idea of gluing shapes and representations is central to this business. Depending on the formal framework, there are different ways of dealing with it.
\begin{enumerate}
\item Similar to the case of triangulated categories, gluing is not available in the $2$-category $\cDER_{\mathrm{St,ex}}$ of stable derivators. But in that framework (and this project began in this framework of derivators), one can partially sail around this defect by more involved combinatorial arguments. This approach was taken in \cite{gst:basic,gst:tree,gst:acyclic}, where representation-theoretic techniques were developed (one-point extensions and homotopical epimorphisms).
\item In contrast to this, the more flexible language of stable $\infty$-categories allows, among many other things, for gluing. At the price of quoting results from the more involved machinery \cite{HTT,HA}, a very elegant construction of reflection functors in stable $\infty$-categories was worked out by Dyckerhoff, Jasso, and Walde in \cite{dycker-jasso-walde:BGP}.
\end{enumerate}

Because of its importance, we single out the special case of a source and a sink of finite valence. We begin by recalling some notation from \cite{gst:tree} and \cite{bg:cubical}.

\begin{notn}\label{notn:chunks}
For $n\geq 1$ we denote by $\cube{n}$ the $n$-cube which is to say the $n$-fold product
\[
\cube{n}=[1]\times\ldots\times[1].
\]
The $n$-cube $\cube{n}$ can be realized as the power set of $\{1,\ldots,n\}$ with the ordering given by inclusion. Correspondingly, the $n$-cube admits the following filtration by full subcategories. For $0\leq l\leq n$ we denote by $\cube{n}_{0\leq l}$ the full subposet spanned by the subsets of cardinality at most $l$. The subposet $\cube{n}_{k\leq l}$ is defined similarly, and there are fully faithful inclusion functors $\iota_{k,l}\colon\cube{n}_{k\leq l}\to\cube{n}$. As special cases, we obtain the inclusion of the source and the sink of valence $n$,
\[
\iota_{0,1}\colon\cube{n}_{0\leq 1}\to\cube{n}\quad\text{and}\quad\iota_{n-1\leq n}\colon\cube{n}_{n-1\leq n}\to\cube{n}.
\]
Similarly, there are the inclusions of the two possible punctured $n$-cubes
\[
\iota_{0,n-1}\colon\cube{n}_{0\leq n-1}\to\cube{n}\quad\text{and}\quad\iota_{1\leq n}\colon\cube{n}_{1\leq n}\to\cube{n}.
\]
\end{notn}

\begin{defn}\label{defn:strongly-cocart}
Let \D be a derivator, let $n\geq 1$, and let $X\in\D(\cube{n})$. The $n$-cube $X$ is \textbf{cocartesian} if it lies in the essential image of $(\iota_{0,n-1})_!\colon\D(\cube{n}_{0\leq n-1})\to\D(\cube{n})$. The $n$-cube $X$ is \textbf{strongly cocartesian} if it lies in the essential image of the functor $(\iota_{0,1})_!\colon\D(\cube{n}_{0\leq 1})\to\D(\cube{n})$.
\end{defn}

Of course, there are dual notions of \textbf{(strongly) cartesian $n$-cubes}. For squares there is obviously no difference between the two notions in \autoref{defn:strongly-cocart}. In higher dimension every strongly cocartesian $n$-cube is cocartesian since $\iota_{0,1}\colon\cube{n}_{0\leq 1}\to\cube{n}$ factors through $\cube{n}_{0\leq n-1}$. But there is no converse to this and, in fact, there is the following more precise result.

\begin{thm}\label{thm:strongly-cocart}
Let \D be a derivator, let $n\geq 2$, and let $X\in\D(\cube{n})$. The $n$-cube $X$ is strongly cocartesian if and only if every subsquare of it is cocartesian.
\end{thm}
\begin{proof}
This is an entirely combinatorial problem and the details can be found in \cite[Thm.~8.4]{gst:basic}.
\end{proof}

\begin{cor}\label{cor:strongly-bicart}
Let \D be a stable derivator and let $n\geq 2$. An $n$-cube in \D is strongly cocartesian if and only if it is strongly cartesian. 
\end{cor}
\begin{proof}
This follows immediately from \autoref{thm:strongly-cocart} and \autoref{defn:stable}.
\end{proof}

We refer to this common class of $n$-cubes as the class of \textbf{strongly bicartesian} $n$-cubes.

\begin{cor}\label{cor:sse-source-sink}
For every fixed $n\geq 2$ the source and the sink of of valence $n$ are strongly stably equivalent,
\[
\cube{n}_{0\leq 1}\sse\cube{n}_{n-1\leq n}.
\]
\end{cor}
\begin{proof}
It is straightforward to adapt the proof of \autoref{eg:sse-A3-1} to our current situation (invoking \autoref{cor:strongly-bicart} instead of \autoref{defn:stable}).
\end{proof}

\begin{figure}
\begin{displaymath}
\xymatrix{
\emptyset\ar[r]&\cube{n}_{=0}\ar[r]&\cube{n}_{0\leq 1}\ar[r]&\ldots\ar[r]&\cube{n}_{0\leq n-1}\ar[r]&\cube{n}\\
&\emptyset\ar[r]\ar[u]&\cube{n}_{=1}\ar[r]\ar[u]&\ldots\ar[r]\ar[u]&\cube{n}_{1\leq n-1}\ar[r]\ar[u]&\cube{n}_{1\leq n}\ar[u]\\
&&\ldots\ar[u]\ar[r]&\ldots\ar[u]\ar[r]&\ldots\ar[u]\ar[r]&\ldots\ar[u]\\
&&&\ldots\ar[u]\ar[r]&\cube{n}_{=n-1}\ar[r]\ar[u]&\cube{n}_{n-1\leq n}\ar[u]\\
&&&&\emptyset\ar[r]\ar[u]&\cube{n}_{=n}\ar[u]\\
&&&&&\emptyset\ar[u]
}
\end{displaymath}
\caption{The cardinality filtration of the $n$-cube $\cube{n}$}
\label{fig:chunks}
\end{figure}
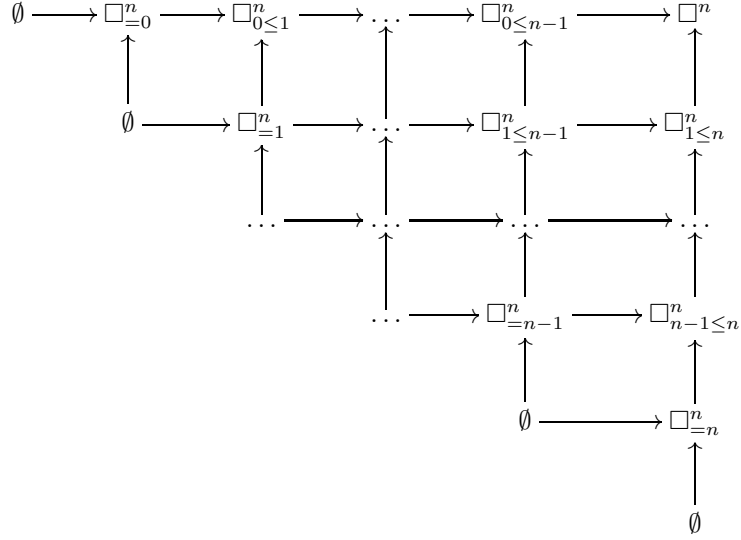

\begin{rmk}
The way we have presented \autoref{cor:sse-source-sink}, it is an immediate consequence of the calculus of $n$-cubes in stable derivators. This cubical calculus was developed more systematically in \cite{bg:cubical}, and here we want to mention some of the key results. Using suggestive notation (compare to \autoref{notn:chunks}), the cardinality filtration of the $n$-cube 
\[
\xymatrix{
\emptyset\ar[r]&\cube{n}_{=0}\ar[r]&\cube{n}_{\leq 1}\ar[r]&\ldots\ar[r]&\cube{n}_{\leq n-1}\ar[r]&\cube{n}
}
\]
leads to an interpolation between cocartesian and strongly cocartesian $n$-cubes. The differences between these individual filtration layers are given by the full subposets $\cube{n}_{k\leq l}\subseteq \cube{n}, 0\leq k\leq l\leq n,$ which we refer to as \textbf{chunks of $n$-cubes}. The chunks and the inclusion functors between them are neatly organized by \autoref{fig:chunks}. In \emph{loc.~cit.} the authors study in detail the calculus of iterated partial cofiber constructions of $n$-cubes and their interaction with functors in \autoref{fig:chunks}. As a generalization of \autoref{cor:sse-source-sink} there is a strong stable equivalence
\[
\cube{n}_{k\leq l}\sse\cube{n}_{n-l\leq n-k};
\]
see \cite[Thm.~10.15]{bg:cubical}. Moreover, for every fixed $n$ these strong stable equivalences for the various $0\leq k\leq l\leq n$ are suitably compatible with each other \cite[Thm.~12.15]{bg:cubical}.

An additional interesting symmetry comes from a further incarnation of Serre equivalences. For every $n\geq 1$ and every $0\leq k\leq l\leq n$ there is a \textbf{Serre equivalence}
\[
S=S_{k,l}\colon\cube{n}_{k\leq l}\sse\cube{n}_{k\leq l}.
\]
Again, also these Serre equivalences turn out to be compatible with the morphisms between the chunks. More specifically, in this situation, for every inclusion functor $i\colon\cube{n}_{k'\leq l'}\to\cube{n}_{k\leq l}$ in \autoref{fig:chunks} there is a canonical isomorphism
\[
S_{k,l}\circ i_!\cong i_\ast\circ S_{k',l'};
\]
see \cite[Thm.~12.6]{bg:cubical}. To put this into plain English, in this case left and right Kan extensions match up to a conjugation by Serre equivalences. As a special case (\cite[Cor.~12.12]{bg:cubical}), for every stable derivator \D there are canonical isomorphisms
\[
\colim_{\cube{n}_{k\leq l}}\cong\mathrm{lim}_{\cube{n}_{k\leq l}}\circ S_{k,l}\colon\D^{\cube{n}_{k\leq l}}\to\D.
\]

These compatibilities of local Serre equivalences with Kan extensions are the starting point of an on-going project on \textbf{global Serre dualities}. Jointly with Falk Beckert we intend to come back to this more systematically elsewhere.
\end{rmk}

Let us now return to the case of acyclic quivers. The approach to reflection functors for acyclic quivers in stable derivators taken in \cite{gst:basic,gst:tree,gst:acyclic} relies largely on the special case in \autoref{cor:sse-source-sink}. More precisely, the classical construction of reflection functors uses finite biproduct in vector spaces (\autoref{con:reflection}). In stable derivators finite biproducts and all the related inclusion and projection maps can be organized in suitable cubical diagrams. For instance, in the case of two objects $x,y\in\D(\ast)$ in the underlying category of a stable derivator \D, there is the associated diagram
 \[
 \xymatrix{
 0\ar[r]\ar[d]\ar@{}[rd]|{\square}&x\ar[r]\ar[d]\ar@{}[rd]|{\square}&0\ar[d]\\
 y\ar[r]\ar[d]\ar@{}[rd]|{\square}&x\oplus y\ar[r]\ar[d]\ar@{}[rd]|{\square}&y\ar[d]\\
 0\ar[r]&x\ar[r]&0,
 }
 \]
 in which all squares are bicartesian. Given a finite acyclic quiver and a source, such cubes can be glued to the quiver in order to then mimic the construcion of reflection functors in \autoref{con:reflection}. One of the main results of \cite{gst:acyclic} is the following.

\begin{thm}\label{thm:sse-reflection}
Let $Q$ be a finite acyclic quiver, let $q_0\in Q$ be a source, and let $Q'=\sigma_{q_0}Q$ be the reflected quiver. The reflection functors define a strong stable equivalence
\[
(s^-_{q_0},s^+_{q_0})\colon Q\sse Q'.
\]
\end{thm}
\begin{proof}
We refer the reader to \cite{gst:acyclic} for a detailed construction of these reflection functors. The theorem appears as \cite[Thm.~10.3]{gst:acyclic} as a specializiation of a more general result (\cite[Thm.~9.11]{gst:acyclic}).
\end{proof}

\begin{cor}
Let $T$ be a finite oriented tree and let $T'$ be an arbitrary reorientation of $T$. There is a strong stable equivalence
\[
T\sse T'.
\]
\end{cor}
\begin{proof}
It is a purely combinatorial argument that a reorientation of a finite tree can be obtained by a finite number of reflections at sources and sinks (\cite[Thm.~1.2(1)]{bernstein-gelfand-ponomarev:Coxeter}). The claim hence follows from \autoref{thm:sse-reflection}.
\end{proof}

In fact, there is also a converse to this result as follows from a specialization to the derivator $\D_k$ of a field $k$. 

\begin{rmk}
\autoref{thm:sse-reflection} is obtained by specialization of the following more general result. Let $A$ be a small category and let $q_1,\ldots, q_n\in A$ be finitely many, not necessarily pairwisely different objects. The two categories which are obtained from $A$ by freely gluing a source or a sink of valence $n$ to the given objects are strongly stable equivalent \cite[Thm.~9.11]{gst:acyclic}. In particular, this shows that a variant of \autoref{thm:sse-reflection} also works for quivers which fail to be finite or acyclic.
\end{rmk}

In contrast to the case of trees, for acyclic quivers the underlying graphs to not determine the strong stable equivalence type. To illustrate this, we revisit the following classical example of orientations $Q$ of an unoriented $n$-cycle
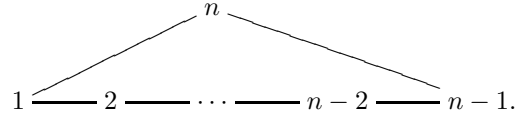
\[
\xymatrix{
&& n \\
1 \ar@{-}[r] \ar@{-}[urr] & 2 \ar@{-}[r] & \cdots \ar@{-}[r] & n-2 \ar@{-}[r] & n-1. \ar@{-}[ull]
}
\]
Such quivers are called \textbf{Euclidean} or \textbf{extended Dynkin} quivers of type $\widetilde{A}_{n-1}$ \cite{ringel:tame-alg,simson-skowronski:2}. For every orientation $Q$ of the $n$-cycle we write $c(Q) = \{p,q\}$, if $p$ arrows in $Q$ are oriented clockwise and $q$ are oriented counterclockwise. 

\begin{cor}\label{cor:sse-cycle}
If $Q$ and $Q'$ are $n$-cycles with $c(Q)=c(Q')$, then there is a strong stable equivalence
\[
Q\sse Q'.
\]
\end{cor}
\begin{proof}
This is \cite[Prop.~10.5]{gst:acyclic}.
\end{proof}

Again, the case of the derivator of a field can be used to show that the converse implication also holds, and $c(Q)$ hence determines the strong stable equivalence type of an $n$-cycle $Q$.

\begin{rmk}
Ladkani studied additional interesting examples of reflection functors for representations of posets with values in arbitrary abelian categories. For these universal derived equivalences we refer to \cite{ladkani:posets,ladkani:posets-tilting,ladkani:posets-cluster-tilting}. Dyckerhoff, Jasso, and Walde \cite{dycker-jasso-walde:BGP} obtain a general construction of reflection functors in stable $\infty$-categories which also generalizes the examples of Ladkani to the homotopical setting.
\end{rmk}

\subsection{Digression: monoidal derivators}
\label{subsec:monoidal}

In this subsection we include a digression on monoidal derivators, the calculus of tensor products of functors, and enriched derivators. We briefly discuss the universality of the derivator of spectra and the resulting canonical enrichment of stable derivators over spectra \cite{heller:htpythies,heller:stable,cisinski:derived-kan,tabuada:universal-invariants,cisinski-tabuada:non-connective,cisinski-tabuada:non-commutative}. These techniques are central to the construction of universal tilting modules in \S\ref{subsec:modules}. 

In this subsection we follow the notional conventions from \cite{gps:additivity} and also refer to that paper for a more detailed discussion of monoidal and stable, monoidal derivators. For basics on monoidal categories see for instance \cite[\S VII]{maclane}.

The starting point of this theory is the following elementary observation.

\begin{rmk}
The $2$-category $\cDER$ of derivators admits ($2$-)products. In more detail, given two derivators $\D_1$ and $\D_2$ the pointwise product
\[
\D_1\times \D_2\colon \cCat\op\to\cCAT\colon A\mapsto \D_1(A)\times\D_2(A)
\]
is again a derivator. The derivator $\D_1\times\D_2$ enjoys the universal property of a product.
\end{rmk}

This simple observation allows us to speak about pseudo-monoid objects and their pseudo-actions in the world of derivators. For instance, a first attempt to define a monoidal derivator would be as a pseudo-monoid object in $\cDER$, i.e., a derivator \V which is endowed with coherently associative and unital multiplications $\otimes\colon\V\times\V\to\V$. Such a pseudo-monoid structure is equivalent to a lift of \V to the $2$-category of monoidal categories, strong monoidal functors and monoidal transformations. However, since we are mostly interested in monoidal structures which are compatible with colimits, we want to build this into the basic notion. Consequently, we begin by making this compatibility precise and, with later applications in mind, we immediately consider arbitrary morphisms of two variables.

\begin{rmk}\label{rmk:2-var-int-ext}
Let $\D_1,\D_2,$ and $\D_3$ be derivators. A \textbf{morphism of two variables} is a morphism of derivators $\otimes\colon\D_1\times\D_2\to\D_3$. By definition of the product $\D_1\times\D_2$ such a morphism has components
\begin{equation}\label{eq:int}
\otimes_A\colon \D_1(A)\times\D_2(A)\to \D_3(A),\qquad A\in\cCat.
\end{equation}
These components are related by pseudo-naturality isomorphisms. In particular, given diagrams $X\in\D_1(A)$ and $Y\in\D_2(A)$, there are canonical isomorphisms
\[
(X\otimes_A Y)_a\cong X_a\otimes Y_a,\qquad a\in A.
\] 
We refer to this description of a morphism of two variables as the \textbf{internal version}. 

Given two categories $A$ and $B$, we can restrict diagrams along the projections $A\times B\to A$ and $A\times B\to B$ and then apply $\otimes_{A\times B}$. This way we obtain functors
\begin{equation}\label{eq:ext}
\otimes\colon\D_1(A)\times\D_2(B)\to\D_3(A\times B),\qquad A,B\in\cCat,
\end{equation}
which again are endowed with suitable coherence isomorphisms. For instance, for diagrams $X\in\D_1(A)$ and $Y\in\D_2(B)$, these specialize to canonical isomorphisms
\[
(X\otimes Y)_{(a,b)}\cong X_a\otimes Y_b,\qquad a\in A, b\in B.
\]
 
One checks that $\otimes\colon\D_1\times\D_2\to\D_3$ can be recovered from this datum by also invoking restrictions along diagonal functors $\Delta\colon A\to A\times A$. In fact, these two descriptions are equivalent \cite[Theorem~3.11]{gps:additivity}, and we refer to the second one as the \textbf{external version} of the morphism of two variables. Note that we distinguish these two versions notationally. 
\end{rmk}

One point of this external reformulation is that it allows for the following simple definition of cocontinuity. 

\begin{defn}\label{defn:2-var-cocont}
Let $\D_1,\D_2$ and $\D_3$ be derivators and let $\otimes\colon\D_1\times\D_2\to\D_3$ be a morphism of two variables. The morphism $\otimes$ \textbf{preserves left Kan extensions along $u\colon A_1\to A_2$ in the first variable} if the canonical morphisms
\begin{equation}\label{eq:first}
(u\times 1)_! (X \otimes Y) \toiso u_! (X) \otimes Y,\qquad X\in\D_1(A_1), Y\in\D_2(B),
\end{equation}
are isomorphisms. 
\end{defn}

This notion has various relevant variants. In particular, we say that $\otimes$ is \textbf{cocontinuous in the first variable} if \eqref{eq:first} are invertible for all $u\colon A_1\to A_2$. Similarly to \cite[Prop.~2.3]{groth:ptstab}, one only has to verify that colimits are preserved. With this preparation we now make the following definition.

\begin{defn}
A \textbf{monoidal derivator} is a pseudo-monoid object $(\V,\otimes,\lS)$ in \cDER such that the monoidal structure $\otimes\colon\V\times\V\to\V$ preserves left Kan extensions in both variables separately.
\end{defn}

Just to stress this, our convention is that a monoidal derivator is obtained by categorifying the notion of a monoid \emph{and} imposing a cocontinuity condition. There are obvious variants of \textbf{braided} or \textbf{symmetric monoidal derivators} (see also \cite[\S XI]{maclane} for the classical case).

\begin{egs}\label{egs:monoidal}
There are the following expected examples of monoidal derivators.
\begin{enumerate}
\item Let $\cV$ be a complete and cocomplete category together with a monoidal structure $\otimes\colon\cV\times\cV\to\cV$ which preserves colimits in both variables separately (this is for example the case when $\cV$ is closed monoidal). Then the represented derivator $y_\cV$ inherits a monoidal structure. In the external version this monoidal structure sends diagrams $X\colon A\to\cV$ and $Y\colon B\to\cV$ to 
\[
(X\otimes Y)_{(a,b)}= X_a\otimes Y_b,\qquad a\in A, b\in B.
\]
\item The homotopy derivator of a monoidal model category is a monoidal derivator. For a more general statements related to Quillen adjunctions of two variables see \cite[Example~3.23]{gps:additivity}. 
\item Let us again consider a commutative ring $R$. The category $\Ch(R)$ of unbounded chain complexes over~$R$ endowed with the projective model structure \cite[\S2.3]{hovey:model} yields the derivator $\D_R$ of the ring. With respect to the usual tensor product $\otimes_R$ of chain complexes, $\Ch(R)$ is a stable, symmetric monoidal model category. We conclude that the derivator $\D_R$ together with the derived tensor product is stable and symmetric monoidal. This applies, in particular, to the derivator $\D_k$ of a field~$k$.
\item There are various Quillen equivalent stable, symmetric monoidal closed model categories of spectra such as the ones in \cite{hss:symmetric,ekmm:rings,mmss:diagram}. Taking any of these as a model for the derivator of spectra \cSp, we conclude that \cSp endowed with the derived smash product is a stable, symmetric monoidal derivator. Similarly, also the derivators \cS and $\cS_\ast$ are symmetric monoidal. (One can show that there is an intrinsic approach to these monoidal structures based on universal properties of the derivators under consideration.)
\end{enumerate}
\end{egs}

Besides the internal and the external versions of morphisms of two variables, there also is the \emph{canceling version} which we discuss next. These specialize to a categorification of the usual tensor products of bimodules over rings and play a key role in \S\ref{subsec:modules}. The classical tensor product in algebra is obtained by coequalizing left and right module structures. The categorification of this construction is given by \emph{coends},  so we begin by extending coends to derivators. Recall that in classical category theory there are various equivalent ways of defining coends \cite[\S IX]{maclane}. In homotopical situations one has to be a bit careful which one to choose, since not all of them lead to the good notion (see \cite[Rmk.~13.21]{bg:cubical}). The following approach works perfectly well (see \cite[\S5]{gps:additivity} and \cite[Appendix~A]{gps:additivity} for more details). 

\begin{con}\label{con:twisted-arrow}
Let $A$ be a small category. The \textbf{twisted arrow category} $\tw(A)$ of $A$ has as objects all morphisms $f\colon a\to b$ in $A$. A morphism $f_1\to f_2$ in $\tw(A)$ is a commutative diagram
\[
\xymatrix{
a_1\ar[r]^-{f_1}&b_1\ar[d]\\
a_2\ar[r]_-{f_2}\ar[u]&b_2,
}
\]
and the ``twist'' is the reorientation of the first coordinate. To put this in plain english, a morphism $f_1\to f_2$ is a $2$-sided factorization of $f_2$ through $f_1$. We note that $\tw(A)$ is simply the category of elements of $\hom_A\colon A\op\times A\to\mathrm{Set}$.

The category $\tw(A)$ comes with a functor 
\begin{equation}\label{eq:(s,t)}
(s,t)\colon\tw(A)\to A\op\times A
\end{equation}
given by the source and target functors. As in the case of ordinary category theory, in order to define the coend construction we need essentially the opposite of this functor, namely the composition
\begin{equation}\label{eq:tw-op}
(t\op,s\op)\colon\tw(A)\op\stackrel{(s,t)\op}{\to}(A\op\times A)\op\cong A\op\times A.
\end{equation}
\end{con}

\begin{defn}\label{defn:coend}
Let $\D$ be a derivator, let $A\in\cCat$, and let $X\in\D(A\op\times A)$. The \textbf{coend} $\int^A X\in\D(\ast)$ of $X$ is given by
\[
\int^AX=\colim_{\tw(A)\op}(t\op,s\op)^\ast X.
\]
\end{defn}

As a composition of functors, the coend functor $\int^A\colon\D(A\op\times A)\to\D(\ast)$ is
\begin{equation}\label{eq:coend}
\int^A\colon \D(A\op\times A)\stackrel{(t\op,s\op)^\ast}{\to}\D(\tw(A)\op)\stackrel{\colim}{\to}\D(\ast).
\end{equation}

\begin{egs}
In the case of a represented derivator, this definition of coends reduces to a formula in \cite{maclane}. Hence, the notion reproduces the classical one in a (complete and) cocomplete category. There is also a different classical description of coends as certain coequalizers. This, however, does not extend that directly to homotopical frameworks. In fact, in that case coequalizers have to be replaced by geometric realizations of simplicial bar constructions and for an extension of this reformulation to derivators we refer to \cite[Appendix~A]{gps:additivity}.
\end{egs}

The motivation for us to discuss coends here is the following construction.

\begin{con}\label{con:tensor-prod-fun}
Let $\D_1,\D_2,$ and $\D_3$ be derivators and let $\otimes\colon\D_1\times\D_2\to\D_3$ be a morphism of two variables. Based on the coend, there is the following third version of such a morphism. Given a small category $A$ and diagrams $X\in\D_1(A\op)$ and $Y\in\D_2(A)$, the external product $X\otimes Y$ lives in $\D_3(A\op\times A)$, and qualifies as an input for the coend. The \textbf{canceling version} of $\otimes$ is
\[
X\otimes_{[A]} Y=\int^AX\otimes Y,\qquad X\in\D_1(A\op), Y\in\D_2(A).
\]
Thus, as a functor $\otimes_{[A]}$ is the composition
\begin{equation}\label{eq:cancel}
\otimes_{[A]}\colon\D_1(A\op)\times \D_2(A)\stackrel{\otimes}{\to}\D_3(A\op\times A)\stackrel{\int^A}{\to}\D_3(\ast),
\end{equation}
and we refer to it as the \textbf{(canceling) tensor product of functors}. Note that we distinguish this canceling version notationally from the internal version \eqref{eq:int} and the external one \eqref{eq:ext}. For the canceling one the subscript $[A]$ is added in order to indicate which ``category is canceled''. 
\end{con}

In order to develop some first intuition for this calculus, we unravel the notion of a twisted arrow category for posets.

\begin{lem}\label{lem:tw-poset}
Let $P$ be a small poset. The functor $(s,t)\colon\tw(P)\to P\op\times P$ induces an isomorphism between $\tw(P)$ and the up-set generated by the diagonal $\Delta_P\subseteq P\op\times P$. 
\end{lem}
\begin{proof}
The functor $(s,t)$ is injective on objects and fully faithful, so it induces an isomorphism onto its image. By reflexivity the diagonal $\Delta_P=\{(p,p)\mid p\in P\}$ lies in the image. Let $f\colon p\to q$ be an object in $\tw(P)$, which is to say that $p,q\in P$ and $p\leq q$. Then in $P\op\times P$ we have $(p,p),(q,q)\leq (p,q)$, and we deduce that the image of $(s,t)$ lies in the up-set generated by the diagonal. Conversely, let $p,q_1,q_2\in P$ be such that $(p,p)\leq (q_1,q_2)$ in $P\op\times P$. Then $q_1\leq p$ and $p\leq q_2$ in $P$, and hence $q_1\to q_2$ defines an object in $\tw(P)$ which under $(s,t)$ is mapped to $(q_1,q_2)$.
\end{proof}

In the calculation of coends we essentially use the opposite functor of $(s,t)$. Thus in order to calculate tensor products of functors over posets we first form the external product, then restrict the diagram to the \emph{down-set} generated by the diagonal and finally calculate the colimit. We illustrate this by the following prominent example of \autoref{con:tensor-prod-fun}. Therein and in what follows, our drawing convention for diagrams of two variable is that the first coordinate is drawn horizontally and the second one vertically. 

\begin{eg}\label{eg:pushout-product}
Let $\D_1,\D_2,$ and $\D_3$ be derivators and let $\otimes\colon\D_1\times\D_2\to\D_3$ be a morphism of two variables. Given coherent morphisms $X=(f\colon x\to y)\in\D_1([1]\op)$ and $X'=(f'\colon x'\to y')\in\D_2([1])$, the external product $X\otimes X'\in\D_3([1]\op\times[1])$ takes by pseudo-naturality the form
\[
\xymatrix{
y\otimes x'\ar[d]_-{\id\otimes f'}&x\otimes x'\ar[l]_-{f\otimes\id}\ar[d]^-{\id\otimes f'}\\
y\otimes y'&x \otimes y'.\ar[l]^-{f\otimes\id}
}
\]
The restriction of $X\otimes X'$ to the down-set generated by the diagonal is the span $(y\otimes x'\ot x\otimes x'\to x\otimes y')$, and the tensor product of functors $X\otimes_{[[1]]}X'\in\D_3(\ast)$ hence sits in a defining cocartesian square
\[
\xymatrix{
y\otimes x'\ar[d]&x\otimes x'\ar[l]_-{f\otimes\id}\ar[d]^-{\id\otimes f'}\\
X\otimes_{[[1]]}X'&x\otimes y'.\ar[l]
}
\]
In this special case the universal object
\[
X\otimes_{[[1]]}X'=y\otimes x'\cup_{x\otimes x'}x\otimes y'
\]
is often referred to as the \textbf{pushout product} of the two morphisms.
\end{eg}

In \S\ref{subsec:modules} we illustrate the calculus of pushout products by specializing to more specific morphisms. Here, we content ourselves by pointing out that this calculus also plays a key role in refined axioms for tensor triangulated categories or monoidal triangulated categories (see \autoref{rmk:monoidal-triang} and the references therein).

\begin{rmk}\label{rmk:variants}
Let $\D_1,\D_2,$ and $\D_3$ be derivators and let $\otimes\colon\D_1\times\D_2\to\D_3$ be a morphism of two variables.
The internal \eqref{eq:int}, external \eqref{eq:ext}, and canceling versions \eqref{eq:cancel} of $\otimes$ can be suitably combined. For instance there is the functor
\[
\D_1(A\times B\op)\times\D_2(A\times B\times C)\to\D_3(A\times C)
\]
which treats the $A$-variable internally and the $C$-variable externally, while the $B$-variable is canceled by means of a coend. Taking the philosophy of derivators serious, we should not be happy with such functors but instead ask for corresponding morphisms of derivators. In the first two cases this leads to parametrized internal and parametrized external tensor products \cite[\S3]{gps:additivity},
\[
\otimes_A\colon\D_1^A\times\D_2^A\to\D_3^A\qquad\text{and}\qquad \otimes\colon\D_1^A\times\D_2^B\to\D_3^{A\times B}.
\]
As a preparation for the remaining case, one notes that for every derivator \D there are parametrized coends
\[
\int^A\colon\D^{A\op}\times\D^A\to\D.
\]
These are derivatorish versions of the usual coends with parameters \cite[\S IX.7]{maclane}. With this in place, the parametrized canceling tensor products
\[
\otimes_{[A]}\colon\D_1^{A\op}\times\D_2^A\to\D_3
\]
can now be defined by the same formula as in \autoref{con:tensor-prod-fun}.
\end{rmk}

For every monoidal derivator \V, the monoidal structure $\otimes\colon\V\times\V\to\V$ is associative and unital up to coherence isomorphisms. This is the case by definition for the internal version and it is straightforward to also check this for the external version. We now turn towards a similar result for the canceling tensor product.

\begin{con}\label{con:bicategory}
Let \V be a monoidal derivator and let $A,B\in\cCat$. We also call objects in $\V(A\times B\op)$ \textbf{bimodules} in $\V$ or \textbf{$(A,B)$-bimodules} in case we want to be more specific. Given a third small category $C$, we can define a \textbf{composition functor} (over $B$) as the parametrized canceling tensor product
\begin{equation}\label{eq:bicat-comp}
\otimes_{[B]}\colon\V(A\times B\op)\times\V(B\times C\op)\to\V(A\times C\op)
\end{equation}
which treats $A$ and $C\op$ externally. For every small category $B$ there is also a preferred $(B,B)$-bimodule $\lI_B$ (which somehow corresponds to the regular bimodule over a ring in algebra). To build towards these bimodules, let us recall that, as part of the monoidal structure, \V is endowed with a monoidal unit \lS. This is a pseudo-functor $\lS\colon y_\ast=\ast\to\V$ defined on the terminal derivator. More concretely, $\lS$ amounts to a pseudo-functorial choice of objects $\lS_A\in\V(A), A\in\cCat$, all of which are monoidal units with respect to the internal tensor products $\otimes_A$. The pseudo-functoriality constraint for the projection $\pi_A\colon A\to\ast$ gives a preferred isomorphism $\lS_A\cong\pi_A^\ast(\lS_\ast)$, which is to say that $\lS_A$ is constant on the monoidal unit $\lS_\ast$ of the underlying category $\V(\ast)$. With this preparation we can define the \textbf{Yoneda bimodule} or \textbf{identity profunctor} of $B$ as the bimodule
\begin{equation}\label{eq:bicat-unit}
\lI_B=(t,s)_!\lS_{\tw(B)}\in\V(B\times B\op).
\end{equation}
Here $(t,s)\colon\tw(B)\to B\times B\op$ sends $(f\colon a\to b)$ to $(b,a)$ (see \autoref{con:twisted-arrow}). By an application of the pointwise formulas for left Kan extensions one checks that for $a,b\in B$ there is a canonical isomorphism
\[
(\lI_B)_{(b,a)}\cong\coprod_{\hom_B(a,b)}\lS_\ast.
\]
In plain English, the value at $(b,a)$ is the coproduct of copies of the unit $\lS_\ast$ parametrized by the morphism set $\hom_B(a,b)$, and this justifies the terminology \emph{Yoneda bimodule}.
\end{con}

The composition of bimodules in monoidal derivators is associative and unital in the following sense. For basic terminology on bicategories we refer the reader to \cite{benabou:intro}, \cite[\S7.7]{borceux:1}, or \cite[\S\S XII.6-7]{maclane}.

\begin{thm}\label{thm:bicategory}
For every monoidal derivator \V there is a bicategory $\cProf(\V)$ with the following description. Its objects are small categories, the category of homomorphisms from $A$ to $B$ is the category of bimodules $\V(A\times B\op)$, the composition functors are given by \eqref{eq:bicat-comp} and the identity 1-cells are the Yoneda bimodules \eqref{eq:bicat-unit}. We refer to $\cProf(\V)$ as the \textbf{bicategory of bimodules} in \V.
\end{thm}
\begin{proof}
The proof of the associativity of composition functors is a fairly direct consequence of a derivatorish version of the Fubini lemma on iterated coends \cite[Lem.~5.3]{gps:additivity}. It turns out that the proof of the unitality is more involved \cite[Appendix~B]{gps:additivity}. For details we refer the reader to \cite[Theorem~5.9]{gps:additivity}.
\end{proof}

In \cite{gps:additivity} we refer to $\cProf(\V)$ as the \emph{bicategory of profunctors} (and this explains the notation), while here we prefer the terminology of bimodules. For later reference we include the following example of Yoneda bimodules.

\begin{eg}\label{eg:Yoneda-poset}
Let \V be a pointed, monoidal derivator and let $P$ be a poset. The Yoneda bimodule $\lI_P\in\cV(P\times P\op)$ restricted to the up-set generated by the diagonal $\Delta_P$ is constant with value the monoidal unit $\lS_\ast\in\V(\ast)$ and it vanishes on the complement \cite[Lem.~7.4]{gst:Dynkin-A}.
\end{eg}

By now we have a reasonably solid understanding of the basic formalism of tensor products of functors in monoidal derivators (more examples will be discussed in \S\ref{subsec:modules}). As a final preparation for the study of tensor-hom-adjunctions we extend the notion of adjunctions of two variables to derivators. There are various ways to make this precise \cite[\S8]{gps:additivity}, and here we choose the following definition. 

\begin{defn}\label{defn:2-var-adj}
Let $\D_1,\D_2,$ and $\D_3$ be derivators. A morphism of two variables $\otimes\colon\D_1\times\D_2\to\D_3$ is a \textbf{two-variable left adjoint} if 
\begin{enumerate}
\item the external components $\otimes\colon\D_1(A)\times\D_2(B)\to\D_3(A\times B)$ are two-variable left adjoints for all $A,B\in\cCat$ and
\item the morphism $\otimes\colon\D_1\times\D_2\to\D_3$ is cocontinuous in both variables separately.
\end{enumerate}
\end{defn}

Thus, for every $A,B\in\cCat$ we ask for the existence of functors
\[
\rhd_{[B]}\colon\D_2(B)\op\times\D_3(A\times B)\to\D_1(A)\quad\text{and}\quad
\lhd_{[A]}\colon\D_3(A\times B)\times\D_1(A)\op\to\D_2(B)
 \]
and for natural isomorphisms
\[
\hom_{\D_3(A\times B)}(X\otimes Y,Z)\cong \hom_{\D_1(A)}(X,Y\rhd_{[B]}Z)\cong\hom_{\D_2(B)}(Y,Z\lhd_{[A]}X).
\]
The notational convention follows \cite{gps:additivity} and it has the feature that it preserves the cyclic structure of the arguments $X,Y,$ and $Z$. The notation of the functors $\rhd_{[B]}$ and $\lhd_{[A]}$ again indicated which ``category is canceled''. One can show that there are morphisms of derivators
\[
\lhd\colon\D_2\op\times\D_3\to\D_1\qquad\text{and}\qquad\rhd\colon\D_3\times\D_1\op\to\D_2
\]
and that the functors $\lhd_{[B]}$ and $\rhd_{[A]}$ arise from these by certain \emph{ends} (see \cite[\S\S8-9]{gps:additivity}).

\begin{defn}
A \textbf{closed monoidal derivator} is a monoidal derivator such that the monoidal structure is a two-variable left adjoint.
\end{defn}

There is the obvious variant of symmetric closed monoidal derivators.

\begin{egs}\label{egs:2-var-adj}
\begin{enumerate}
\item Derivators represented by closed monoidal categories are also closed monoidal. More generally, two-variable adjunctions between complete and cocomplete categories induce two-variable adjunctions of represented derivators \cite[Ex.~8.9]{gps:additivity}.
\item Homotopy derivators of cofibrantly-generated monoidal model categories are closed monoidal \cite[Thm.~9.11]{gps:additivity}. Two-variable Quillen left adjoint functors between combinatorial model categories induces two-variable adjunctions between homotopy derivators \cite[Ex.~8.11]{gps:additivity}.
\item The derivator $\D_R$ of a commutative ring $R$, the derivator $\cS$ of spaces, the derivator $\cS_\ast$ of pointed spaces, and the derivator $\cSp$ of spectra are symmetric closed monoidal.
\item Let $\otimes\colon\D_1\times\D_2\to\D_3$ be a two-variable left adjoint morphism of derivators. Then also the internal, external, and canceling versions of $\otimes$ from \autoref{rmk:variants} are two-variable left adjoints \cite[\S8]{gps:additivity}. In particular, for every closed monoidal derivator \V the bicategory $\cProf(\V)$ of bimodules is closed. 
\end{enumerate}
\end{egs}

\begin{rmk}\label{rmk:monoidal-triang}
Many interesting examples of triangulated categories arising in nature come with an additional monoidal structure (see for instance \cite[\S1]{balmer:TTG}). A natural question then is what kind of axioms should be imposed on a category which is simultaneously endowed with a triangulation and a monoidal structure. The search for good compatibility axioms to be imposed on such categories has already some history and references include \cite{margolis:spectra,HPS:axiomatic,may:additivity,keller-neeman:D4}. Of course, the answer to this question depends on ones purposes (for instance for the discussion of spectra of tensor triangulated categories relatively simple axioms are good enough; see \cite{balmer:TTG} and the many references therein).

A different purpose one might have is the study of \emph{duality phenomena} (see \cite{dold-puppe:duality} as well as \cite{becker-gottlieb:duality} and the references there). Let \cT be monoidal, triangulated category and let $x\in\cT$ be a dualizable object. Every endomorphism $\varphi\colon x\to x$ has a trace $\mathrm{tr}(\varphi)\colon\lS\to\lS$ which is an endomorphism of the monoidal unit. Morally, one would expect that traces of dualizable objects are additive with respect to distinguished triangles. More precisely, given an endomorphism of a distinguished triangle
\[
\xymatrix{
x\ar[r]\ar[d]_-{\varphi_x}&y\ar[r]\ar[d]_-{\varphi_y}&z\ar[r]\ar[d]^-{\varphi_z}&\Sigma x\ar[d]^-{\Sigma\varphi_x}\\
x\ar[r]&y\ar[r]&z\ar[r]&\Sigma x
}
\]
such that $x,y,$ and $z$ are dualizable, the formula
\[
\mathrm{tr}(\varphi_x) + \mathrm{tr}(\varphi_z) =\mathrm{tr}(\varphi_y)
\]
should hold. However, as a further reminiscence of the non-functoriality of cones, such a formula does not hold at the level of triangulated categories \cite{ferrand:nonadd}. And, in fact, this failure was part of the original motivation of Grothendieck to come up with the notion of a derivator in the first place. 

May \cite{may:additivity} proposed very carefully chosen compatibility axioms for monoidal, triangulated categories which can be used to establish additivity of Euler characteristics (traces of identities), and a representation theoretic perspective on these appears in work of Keller--Neeman \cite{keller-neeman:D4}. In \cite{gps:additivity} the basic formalism of monoidal, stable derivators is developed, and canonical monoidal triangulations in the sense of May are constructed for such derivators. As an additional application of these techniques, it is shown that traces of coherent morphisms are additive with respect to cofiber sequences \cite{gps:additivity}.

The credo suggested by this is that all the compatibility is already encoded in a stable, monoidal derivator (or stable, monoidal $\infty$-category). Hence, instead of imposing more and more complicated \emph{axioms} on monoidal, triangulated categories, one should try to prove \emph{lemmas} about, say, stable, monoidal derivators and the corresponding calculus of bimodules. This was illustrated successfully in the papers \cite{ps:linearity,ps:linearity-fp,gallauer:traces} where the above additivity formula was extended to homotopy finite colimits of dualizable objects (see also the closely related \cite{jin-yang}).
\end{rmk}

\begin{defn}\label{defn:enriched}
Let \V be a monoidal derivator. A \textbf{\V-module} is a cocontinuous pseudo-module over \V. A \textbf{\V-enriched derivator} is a \V-module such that the action is a two-variable left adjoint.
\end{defn}

Thus, a \V-module is a derivator \D together with a morphism of two variables
\begin{equation}\label{eq:action}
\otimes\colon\V\times\D\to\D,
\end{equation}
the \emph{action} or \emph{module structure}, such that
\begin{enumerate}
\item the action is associative and unital up to coherence isomorphisms and
\item the action is cocontinuous in both variables separately.
\end{enumerate} 
If \V and \D are stable derivators, then these are derivatorish versions of actions of triangulated categories (as in \cite{stevenson:support-action}). For a \V-enriched derivator the action \eqref{eq:action} is supposed to be a two-variable left adjoint (\autoref{defn:2-var-adj}). Thus, for a \V-enriched derivator \D besides the action morphisms we also have \textbf{internal homs} or \textbf{internal mapping objects}
\[
\rhd\colon\D\op\times\D\to\V
\]
and \textbf{cotensors}
\[
\lhd\colon\D\times\V\op\to\D.
\]

\begin{egs}
\begin{enumerate}
\item If $\cV$ is complete and cocomplete closed symmetric monoidal category and $\cC$ is a complete and cocomplete $\cV$-enriched category in the classical sense \cite{kelly:enriched}, then $y_\cC$ is a $y_\cV$-enriched derivator (\autoref{egs:2-var-adj}).
\item If $\cV$ is symmetric monoidal model category and $\cM$ is a $\cV$-model category in the sense of \cite{hovey:model}, then $\ho_\cM$ is a $\ho_\cV$-enriched derivator.
\item Any closed monoidal derivator is enriched over itself.
\end{enumerate}
\end{egs}

This gives rise to many interesting examples of enriched derivators. We invite the reader to come up with additional closure properties of enriched derivator such as the passage to shifted derivators or the behavior under monoidal adjunctions. We conclude this subsection by a discussion of the universality of the classical homotopy theories of spaces, pointed spaces, and spectra, and the resulting canonical enrichments of stable derivators over spectra. These are deep theorems and they are crucial to our applications in \S\ref{subsec:modules}. Related references include \cite{heller:htpythies,heller:stable,franke:adams,cisinski:derived-kan,tabuada:universal-invariants,cisinski-tabuada:non-connective,cisinski-tabuada:non-commutative}.

\begin{rmk}\label{rmk:universal-htpy-thy}
The homotopy derivator of spaces~\cS is freely generated under colimits by the singleton~$\Delta^0$ \cite[Theorem~3.24]{cisinski:derived-kan}. In more detail, given derivators $\D,\E$ we denote by $\Hom_!(\D,E)$ the category of cocontinuous morphisms $\D\to\E$ and all natural transformations. In particular, given a cocontinuous morphism $F\colon\cS\to\D$, we can evaluate the underlying functor $\cS(\ast)\to\D(\ast)$ at $\Delta^0$. The universality of spaces is made precise by the statement that the evaluation induces an equivalence of categories
\[
\Hom_!(\cS,\D)\stackrel{\sim}{\to}\D(\ast)\colon F\mapsto F(\Delta^0).
\]
This universal property explains the ubiquity of spaces in abstract homotopy theory. A formally correct proof of this result is highly non-trivial, and we refer the reader to \cite{heller:htpythies} and \cite{cisinski:derived-kan}. But to provide some evidence for this result let us recall the following.
\begin{enumerate}
\item Every topological space has a CW-approximation.
\item CW-complexes are constructed from $n$-cells under coproducts, pushouts and countable colimits.
\end{enumerate}
Since the $n$-cells are all weakly contractible, this means that every space can be built from $\Delta^0$ using homotopy colimits only. These heuristics make the universality at least quite plausible.

There are variants of this result for pointed spaces and spectra. The derivator of pointed spaces $\cS_\ast$ is the free pointed derivator generated by the $0$-sphere $\lS^0$. And the derivator $\cSp$ is the free stable derivator generated by the sphere spectrum~$\lS$ \cite[Theorem~A.11]{cisinski-tabuada:non-connective}. Thus, for every stable derivator \D the evaluation
\[
\Hom_!(\cSp,\D)\stackrel{\sim}{\to}\D(\ast)\colon F\mapsto F(\lS).
\]
is an equivalence of categories.
\end{rmk}

\begin{thm}\label{thm:stable-enriched}
Every stable derivator is enriched over the derivator $\cSp$ of spectra.
\end{thm}
\begin{proof}
The proof is essentially a consequence of the above-mentioned universal property of the derivator of spectra, and we refer the reader to \cite[Appendix~A.3]{cisinski-tabuada:non-connective} for details. It turns out that the action
\[
\otimes\colon\cSp\times\D\to\D
\]
which belongs to this enrichment is characterized by two properties. First, the action preserves colimits in both variables separately, and, second, the sphere spectrum acts trivially (which is to say $\lS\otimes-\cong\id_\D$ as it is the case for all actions).
\end{proof}

In \S\ref{subsec:modules} we use this canonical enrichment to construct universal tilting modules. Let us close this subsection by the following obvious remark just to underline the usefulness of mapping spectra.

\begin{rmk}
Let \D be a stable derivator, let $A\in\cCat$, and let $X,Y\in\D(A)$. We refer to $X\rhd_{[A]} Y\in\cSp(\ast)$ as the \textbf{mapping spectrum} of $X,Y$. This is a higher-structured version of the categorical morphisms $\hom_{\D(A)}(X,Y)$. In fact, these can be recovered as
\[
\pi_0(X\rhd_{[A]}Y)\cong\hom_{\cSp(\ast)}(\lS,X\rhd_{[A]}Y)\cong\hom_{\D(A)}(\lS\otimes X,Y)\cong\hom_{\D(A)}(X,Y).
\]
Similarly, the extension groups can be described as homotopy groups of the mapping spectrum. The formalism of two-variable adjunctions allows us to calculate the mapping spectrum $X\rhd_{[A]} Y$ as an end which is useful in many situations. 
\end{rmk}

\subsection{Universal tilting modules}
\label{subsec:modules}

In this section we discuss universal tilting modules: certain explicitly constructed spectral bimodules realizing our strong stable equivalences. These universal tilting modules are spectral refinements of the classical tilting complexes and they are non-trivial invertible elements in the bicategory of spectral bimodules. Considered this way, abstract representation contributes to the calculation of spectral Picard groupoids.

Thus, analogously to the situation in tilting theory, here the focus shifts from the equivalences themselves to the representing bimodules. 

We begin this section by defining the class of morphisms which are associated to spectral bimodules.

\begin{defn}
Let $A,B\in\cCat$ and $M\in\cSp(B\times A\op)$. For every stable derivator \D, the \textbf{weighted colimit} with weight $M$ is the morphism
\[
M\otimes_{[A]}-\colon\D^A\to\D^B.
\]
\end{defn}

There is the dual notion of a \textbf{weighted limit} (using cotensors and ends instead of tensors end coends). In this subsection we will always focus on the case of weighted colimits.

\begin{rmk}\label{rmk:weight-unique}
Let \D be a stable derivator. We say that a morphism $F\colon\D^A\to\D^B$ is a weighted colimit if there is a spectral bimodule $M\in\cSp(B\times A\op)$ and a natural isomorphism
\[
F\cong M\otimes_{[A]}-\colon\D^A\to\D^B.
\]
If we do not merely have a morphism $F\colon\D^A\to\D^B$ for a fixed stable derivator, but a pseudo-natural transformation
\[
F\colon(-)^A\to(-)^B\colon\cDER_{\mathrm{St,ex}}\to\cDER
\]
between the corresponding abstract representation theories, then the weight is uniquely determined by $F$. In fact, it suffices to consider the universal stable derivator $\D=\cSp$ of spectra. The natural isomorphism $F_{\cSp}\cong M\otimes_{[A]}-$ specializes to a natural isomorphism of functors
\[
\cSp^A(A\op)=\cSp(A\times A\op)\to\cSp^B(A\op)=\cSp(B\times A\op).
\]
Plugging in the Yoneda bimodule $\lI_A$, we can invoke the unitality constraint from \autoref{thm:bicategory} in order to obtain
\[
F(\lI_A)\cong M\otimes_{[A]}\lI_A\cong M.
\]
\end{rmk}

This remark shows that in favorable cases the weight is obtained if one evaluates the morphism on a suitable Yoneda bimodule. This idea is also central to the following two results. The first of these two results also justifies the terminology of weighted \emph{colimits}.

\begin{thm}\label{thm:weighted-colim}
Let \D be a stable derivator and let $u\colon A\to B$ be a functor.
\begin{enumerate}
\item The restriction morphism $u^\ast\colon\D^B\to\D^A$ is a weighted colimit.
\item The left Kan extension morphism $u_!\colon\D^A\to\D^B$ is a weighted colimit. In particular, the colimit morphism $\colim_A\colon\D^A\to\D$ is a weighted colimit.
\end{enumerate}
\end{thm}
\begin{proof}
The stable derivator \D is by \autoref{thm:stable-enriched} a $\cSp$-enriched derivator, and let us denote the corresponding adjunction of two variables by 
\[
\otimes\colon\cSp\times\D\to\D.
\]
Moreover, for arbitrary categories $C,D,E$ it follows from \autoref{rmk:variants} that also the external-canceling version of this action
\[
\otimes_{[D]}\colon\cSp^{C\times D\op}\times\D^{D\times E\op}\to\D^{C\times E\op}
\]
is a left adjoint of two variables. In particular, the morphism $\otimes_{[D]}$ is compatible with restrictions and left Kan extensions in both variables separately. To conclude the proof it suffices to specialize this to the two situations under consideration. 

In the first case, let $X\in\D(B)$ and let $\lI_B\in\cSp(B\times B\op)$ be the Yoneda bimodule. By left unitality and pseudo-naturality we obtain
\[
u^\ast(X)\cong u^\ast (\lI_B\otimes_{[B]}\otimes X)\cong \big((u\times\id_{B\op})^\ast\lI_B\big)\otimes_{[B]}X, 
\]
and this defines the intended natural isomorphism
\begin{equation}\label{eq:res-weight}
u^\ast\cong \big((u\times\id_{B\op})^\ast\lI_B\big)\otimes_{[B]}-\colon\D^B\to\D^A,
\end{equation} 
exhibiting the restriction morphism as a weighted colimit. 

Similarly, in the second case, let $X\in\D(A)$ and let $\lI_A\in\cSp(A\times A\op)$ be the Yoneda bimodule. Using the fact that the morphism $-\otimes_{[A]}X$ preserves colimits and hence left Kan extensions, we obtain the isomorphisms
\[
u_!(X)\cong u_!(\lI_A\otimes_{[A]}X)\cong \big((u\times\id_{A\op})_!\lI_A\big)\otimes_{[A]}X.
\]
This yields the intended natural isomorphism
\begin{equation}\label{eq:lkan-weight}
u_!\cong\big((u\times\id_{A\op})_!\lI_A\big)\otimes_{[A]}-\colon\D^A\to\D^B
\end{equation}
exhibiting left Kan extensions as weighted colimits. Of course, the case of colimits is obtained by specializing to $u=\pi_A$.
\end{proof}
 
This first result generalizes to more general enriched derivators in not necessarily stable situations. In contrast to this, \autoref{thm:weighted-colim-II} relies crucially on stability.

\begin{rmk}
Let us recall that a left exact morphism of derivators is defined as a morphism which preserves terminal objects and cartesian squares (\autoref{egs:classes-of-mor}). By \autoref{thm:rex} such morphisms preserve homotopy finite limits. Since right Kan extensions in derivators can be calculated pointwise in terms of limits (by (Der4)), it follows that right exact morphisms also preserve \emph{right homotopy finite} right Kan extension (see \cite[Thm.~9.14]{groth:revisit}). This includes those right Kan extensions for which the corresponding slice categories are homotopy finite, and the following result hence covers a large class of examples.
\end{rmk}
 
\begin{thm}\label{thm:weighted-colim-II}
Let \D be a stable derivator and let $u\colon A\to B$ be a functor.
\begin{enumerate}
\item If $u$ is a sieve, then the right Kan extension morphism $u_\ast\colon\D^A\to\D^B$ is a weighted colimit.
\item If $u\colon A\to B$ is right homotopy finite, then the right Kan extension morphism $u_\ast\colon\D^A\to\D^B$ is a weighted colimit. In particular, for every homotopy finite category $A$ the limit morphism $\mathrm{lim}_A\colon\D^A\to\D$ is a weighted colimit.
\end{enumerate}
\end{thm}
\begin{proof}
As in the proof of \autoref{thm:weighted-colim} for small categories $C,D,E$ the action of spectra yields a left adjoint of two variables
\[
\otimes_{[D]}\colon\cSp^{C\times D\op}\times\D^{D\times E\op}\to\D^{C\times E\op}
\]
between \emph{stable} derivators. In particular, for a fixed object $X\in\D(A)$ the morphism
\[
-\otimes_{[A]}X\colon\cSp^{A\op}\to\D
\]
is exact and hence preserves right Kan extensions along sieves and, more generally, right homotopy finite right Kan extensions (\cite[Thm.~9.17]{groth:revisit}). Thus, in both situations we obtain isomorphisms
\[
u_\ast(X)\cong u_\ast(\lI_A\otimes_{[A])}X)\cong \big((u\times\id_{A\op})_\ast\lI_A\big)\otimes_{[A]}X.
\]
Letting the diagram $X$ vary we obtain the natural isomorphism
\begin{equation}\label{eq:rkan-weight}
u_\ast\cong \big((u\times\id_{A\op})_\ast\lI_A\big)\otimes_{[A]}-\colon\D^A\to\D^B
\end{equation}
showing that these particular \emph{right} Kan extensions are weighted colimits.
\end{proof}

\begin{rmk}\label{rmk:thms-on-weights}
We want to include a short discussion of these two theorems.
\begin{enumerate}
\item The proofs of \autoref{thm:weighted-colim} and \autoref{thm:weighted-colim-II} are constructive in that they offer formulas which allow us to calculate the weights for restrictions and suitable Kan extensions. In fact, as detailed by the formulas \eqref{eq:res-weight}, \eqref{eq:lkan-weight}, and \eqref{eq:rkan-weight}, in all cases we start with a suitable Yoneda bimodule. One of the variables is bound by the canceling tensor product (or the corresponding coend), and we simply apply the corresponding operation to the remaining free variable. This allows us to explicitly calculate the representing weights as we illustrate a bit further below.
\item Restrictions and all sufficiently finite left and right Kan between stable derivators are weighted colimits, and there are dual statements for \emph{weighted limits}. The corresponding representing spectral bimodules admit a fairly rich calculus, and we refer the reader to \cite{gps:additivity,gst:Dynkin-A,gs:generalized,shulman:framed,ps:linearity} for more details. 
\item The motivation to break up the above examples of weighted colimits into two classes is twofold. First, \autoref{thm:weighted-colim} extends to more general \V-modules and \V-enriched derivator, while only in \autoref{thm:weighted-colim-II} we rely on additional exactness properties. In fact, the result about right Kan extensions along sieves is also valid in modules over pointed monoidal derivators. Second, the results in \autoref{thm:weighted-colim-II} indicate that in stable land the distinction between constructions on the left and constructions on the right is blurred to some extent (sufficiently finite limits are weighted colimits). It turns out that this is a defining feature of stability and this perspective offers an interesting invitation to a formal study of abstract stability (see \cite{gs:generalized} for first steps along these lines).
\end{enumerate}
\end{rmk}

\begin{eg}\label{eg:weight-susp}
For every diagram of spectra $M\in\cSp(A)$ and every stable derivator the action
\[
M\otimes-\colon\D\to\D^A
\]
is a weighted colimit.  As a special case, for the suspension of the sphere spectrum $\Sigma\lS\in\cSp(\ast)$ and $X\in\D(B)$ we have
\[
(\Sigma\lS)\otimes X\cong \Sigma(\lS\otimes X)\cong \Sigma X.
\]
Consequently, the suspension morphism $\Sigma\colon\D\to\D$ is a weighted colimit with weight $\Sigma\lS$,
\[
\Sigma\lS\otimes-\cong\Sigma\colon\D\to\D.
\]
There is a similar description of $\Sigma^n, n\in\lZ$ as weighted colimit with weight $\Sigma^n\lS$.
\end{eg}

In the following examples we study special cases of weighted colimits of morphisms. In those cases we start from the Yoneda bimodule $\lI_{[1]}$ and follow the above recipe to construct weights (see \autoref{rmk:thms-on-weights}). To also get more used to the calculus of coends from \S\ref{subsec:monoidal}, we double-check that the weights really represent the constructions we started with. Let us recall that our drawing convention for bimodules is that the first coordinate is drawn horizontally and the second one vertically. As a special case of \autoref{eg:Yoneda-poset}, there is the following example.

\begin{eg}\label{eg:Yoneda-A2}
The Yoneda module $\lI_{[1]}\in\cSp([1]\times[1]\op)$ looks like
\[
\xymatrix{
\lS\ar[r]^-\id &\lS\\
\lS\ar[r]\ar[u] &\lS.\ar[u]_-\id
}
\]
\end{eg}

\begin{eg}\label{eg:weight-0-1}
For every stable derivator the evaluation morphism $0^\ast\colon\D^{[1]}\to\D$ is a weighted colimit (\autoref{thm:weighted-colim}). The weight $P_0$ is obtained from $\lI_{[1]}$ by evaluation at $0$ in the covariant variable and is hence given by
\[
P_0=(0\times\id_{[1]\op})^\ast\lI_{[1]}=(\lS\ot 0)\in\cSp([1]\op).
\]
To double-check this result, for $X=(f\colon x\to y)\in\D([1])$ we calculate the canceling tensor product $P_0\otimes_{[[1]]}X=(\lS\ot 0)\otimes_{[[1]]}(x\to y)$. By \autoref{eg:pushout-product} this is simply the pushout-product of these two morphisms whose calculation is displayed in the following diagram
\[
\xymatrix{
\lS\otimes x\ar[d]&0\otimes x\ar[l]\ar[d]^-\cong&& x\ar@{-->}[d]_-\cong\ar@{}[rd]|{\square}&0\ar[l]\ar[d]^-\cong\\
\lS\otimes y&0\ar[l]\otimes y,&& x&0.\ar@{-->}[l]
}
\]
Since isomorphisms are stable under cobase change \cite[Prop.~3.12]{groth:ptstab}, the desired pushout is simply $x$. This calculation hence confirms that there is a natural isomorphism
\[
0^\ast\cong P_0\otimes_{[[1]]}-=(\lS\ot 0)\otimes_{[[1]]}-\colon\D^{[1]}\to\D.
\]

Similarly, the weight $P_1$ for the evaluation morphisms $1^\ast\colon\D^{[1]}\to\D$ is given by
\[
P_1=(1\times\id_{[1]\op})^\ast\lI_{[1]}=(\lS\ot \lS)\in\cSp([1]\op).
\]
Analogously to the previous case, in order to calculate $P_1\otimes_{[[1]]}X$ we contemplate the diagrams
\[
\xymatrix{
\lS\otimes x\ar[d]&\lS\otimes x\ar[l]_-\cong\ar[d]&& x\ar@{-->}[d]\ar@{}[rd]|{\square}&x\ar[l]_-\cong\ar[d]^-f\\
\lS\otimes y&\lS\ar[l]^-\cong\otimes y,&& y&y.\ar@{-->}[l]^-\cong
}
\]
Again, this calculation confirms the above claim and we obtain a natural isomorphism
\[
1^\ast\cong P_1\otimes_{[[1]]}-=(\lS\ot \lS)\otimes_{[[1]]}-\colon\D^{[1]}\to\D.
\]
\end{eg}

This example generalizes as follows.

\begin{rmk}
Let \D be a stable derivator, $A\in\cCat$, and $a\in A$. The evaluation morphism $a^\ast\colon\D^A\to\D$ is a weighted colimit with weight $P_a\in\cSp(A\op),$ the free diagram generated at $a$ by $\lS\in\cSp(\ast).$ More formally, let $a\colon\ast\to A\op$ be the functor classifying the object $a$, then $P_a$ is given by
\[
P_a\cong a_!(\lS)\in\cSp(A\op).
\]
This conclusion also holds in arbitrary derivators, and for this it suffices to consider the corresponding weight in spaces instead of in spectra.
\end{rmk}

\begin{eg}\label{eg:weight-C}
Let \D be a stable derivator and let $C\colon\D^{[1]}\to\D$ be the cone morphism. We recall from \autoref{con:basic-ptd} that $C$ is a finite composition of suitably finite Kan extensions and evaluation morphisms. By \autoref{thm:weighted-colim} and \autoref{thm:weighted-colim-II} the cone morphism is a weighted colimit. Moreover, the weight is obtained from the Yoneda bimodule $\lI_{[1]}\in\cSp([1]\times[1]\op)$ by an application of the cone with parameters in $[1]\op$. The isomorphisms
\[
C(0\to\lS)\cong\lS\qquad\text{and}\qquad C(\lS\toiso\lS)\cong 0
\] 
imply that the weight $M$ of $C$ is isomorphic to $(0\ot\lS)\in\cSp([1]\op)$. We again verify this description. Given a morphism $X=(f\colon x\to y)\in\D([1])$ the following calculation
\[
\xymatrix{
0\otimes x\ar[d]&\lS\otimes x\ar[l]\ar[d]&& 0\ar@{-->}[d]\ar@{}[rd]|{\square}&x\ar[l]\ar[d]^-f\\
0\otimes y&\lS\otimes y,\ar[l]&& Cf&y,\ar@{-->}[l]
}
\]
whose details are left to the reader double-checks the claim (compare again to \autoref{con:basic-ptd}). Thus, we obtain the intended natural isomorphism
\[
C\cong (0\ot \lS)\otimes_{[[1]]}-\colon\D^{[1]}\to\D.
\]
\end{eg}

\begin{rmk}
This description of cones as weighted colimits generalizes to pointed derivators. In fact, in that case the derivator $\cS_\ast$ of pointed spaces is universal \autoref{rmk:universal-htpy-thy}, and every pointed derivator is a $\cS_\ast$-module. Of course, in contrast to this the cone is \emph{not an ordinary colimit} but merely a weighted colimit. This is easy to see already for represented derivators because in that case the cone agrees with the usual cokernel (\autoref{egs:basic-ptd}).
\end{rmk}

\begin{eg}\label{eg:weight-cof-seq}
For every stable derivator \D there is the morphism 
\[
G\colon\D^{[1]}\to\D^{[1]\times[2]}
\]
which sends a morphism to a (coherent) cofiber sequence. Let us recall from the proof of \autoref{thm:stable-tria} that $G$ is a finite composition of sufficiently finite homotopy Kan extensions (see \autoref{prop:cof-seq} for the classical case which serves as a blueprint for arbitrary stable derivators). Consequently, $G$ is a weighted colimit whose weight $M$ lives in $\cSp(([1]\times[2])\times[1]\op)$. As one easily verifies by calculations similar to the previous cases, the representing weight $M$ looks like \autoref{fig:weight-cof-seq}. In that figure the paper plane corresponds to the coordinates in $[1]\times[2]$ while the diagonal direction corresponds to the copy of $[1]\op$. Evaluating at $(0,0), (0,1),$ and $(1,1)\in[1]\times[2]$, we obtain the coherent morphisms
\[
M_{0,0}=(\lS\ot 0),\quad M_{0,1}=(\lS\ot\lS), \quad\text{and}\quad M_{1,1}=(0\ot\lS)\in\cSp([1]\op)
\]
which already occured as weights for $0^\ast, 1^\ast,$ and $C$, respectively (see \autoref{eg:weight-0-1} and \autoref{eg:weight-C}). Let us only double-check the remaining non-trivial evaluation $M_{1,2}\in\cSp([1]\op)$ which is given by
\[
(\Sigma\lS\ot 0)\cong (\Sigma\lS)\otimes(\lS\ot 0).
\]

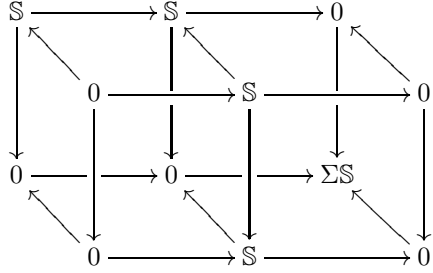
\begin{figure}[h]
\centering
\[
\xymatrix@-0.5pc{
\lS\ar[rr]\ar[dd] && \lS \ar'[d][dd] \ar[rr]&& 0 \ar'[d][dd]\\
& 0 \ar[rr] \ar[dd] \ar[ul]&& \lS \ar[dd] \ar[ul] \ar[rr]&& 0 \ar[dd] \ar[ul]\\
0 \ar'[r][rr] && 0  \ar'[r][rr]&& \Sigma\lS\\
& 0 \ar[ul] \ar[rr] && \lS \ar[ul] \ar[rr] && 0 \ar[ul] 
}
\]
\caption{The universal constructor for cofiber sequences.}
\label{fig:weight-cof-seq}
\end{figure}

Invoking \autoref{eg:weight-susp} and \autoref{eg:weight-0-1}, the corresponding weighted colimit sends $X=(f\colon x\to y)\in\D([1])$ to
\begin{align*}
(\Sigma\lS\ot 0)\otimes_{[[1]]}X&\cong \big((\Sigma\lS)\otimes(\lS\ot 0)\big)\otimes_{[[1]]}X\\
&\cong (\Sigma\lS)\otimes\big((\lS\ot 0)\otimes_{[[1]]}X\big)\\
&\cong (\Sigma\lS)\otimes x\\
&\cong \Sigma x,
\end{align*}
as it is supposed to be the case. Since $M$ is a representing weight for $G$, we also refer to $M$ as the \textbf{universal constructor for cofiber sequences}.

If we restrict $M$ along the functor $i\colon [1]\to[1]\times[2]$ which classifies the morphism $(0,1)\to(1,1)$, then we obtain the bimodule $N=i^\ast M\in\cSp([1]\times[1]\op)$ looking like
\begin{equation}\label{eq:weight-cof}
\vcenter{
\xymatrix{
\lS\ar[r] &0\\
\lS\ar[r]\ar[u] &\lS.\ar[u]
}
}
\end{equation}
The corresponding weighted colimit is the morphism $\cof\colon\D^{[1]}\to\D^{[1]},$
\[
N\otimes_{[[1]]}-\cong\cof\colon\D^{[1]}\to\D^{[1]},
\]
which at the level of canonical triangulations corresponds to the rotation of triangles. 

From these considerations we can also calculate the weight which describes the fiber morphism $\fib\colon\D^{[1]}\to\D^{[1]}$ as a weighted \emph{colimit}. In fact, by \eqref{eq:cof-cube} there is a natural isomorphism $\cof^3\cong\Sigma\colon\D^{[1]}\to\D^{[1]}$, and in combination with \autoref{thm:stable} this implies
\[
\fib\cong \Omega\circ\cof^2.
\]
The weight for $\cof^2$ can be calculated as $N\otimes_{[[1]]} N$. More efficiently, it can simply be read off from $M$ (see \autoref{fig:weight-cof-seq}) by restriction along the functor $j\colon[1]\to[1]\times[2]$ classifying the morphism $(1,1)\to(1,2)$. By \autoref{eg:weight-susp} and \autoref{thm:stable} we conclude that the weight for $\fib$ is given by $\Omega (j^\ast M)$ and looks like
\[
\xymatrix{
0\ar[r] \ar@{}[rd]|{\square}&\lS\\
\Omega\lS\ar[r]\ar[u] &0.\ar[u]
}
\]
\end{eg}

This example concludes the little detour which was included in order to develop some first feeling for the calculus of weighted colimits and canceling tensor products (for additional examples we refer to \cite{gps:additivity} or \cite{gst:Dynkin-A}). We now turn to the case which is of particular interest to abstract representation theory namely the case of weights such that the corresponding weighted colimits are equivalences. After all our goal is to describe strong stable equivalences by means of spectral bimodules.

\begin{prop}\label{prop:invertible}
Let $A,B\in\cCat$ and let $M\in\cSp(B\times A\op), N\in\cSp(A\op\times B)$ be spectral bimodules. The following are equivalent.
\begin{enumerate}
\item There are isomorphisms $M\otimes_{[A]}N\cong\lI_B$ and $N\otimes_{[B]} M\cong\lI_A$.
\item The weighted colimits $M\otimes_{[A]}-\colon\D^A\to\D^B$ and $N\otimes_{[B]}-\colon\D^B\to\D^A$ are inverse equivalences for all stable derivators \D.
\item The weighted colimits $M\otimes_{[A]}-\colon\cSp^A\to\cSp^B$ and $N\otimes_{[B]}-\colon\cSp^B\to\cSp^A$ are inverse equivalences.
\end{enumerate}
In this situation we say that $M$ is an \textbf{invertible spectral bimodule} and that $N$ is an \textbf{inverse bimodule} of $M$.
\end{prop}
\begin{proof}
Statement (ii) follows easily from (i) since weighted colimits associated to Yoneda bimodules are naturally isomorphic to the corresponding identity morphisms. Clearly, (iii) is a consequence of (ii), and it remains to show that (iii) implies (i). This implication is immediate from the uniqueness of representing weights for weighted colimits (see \autoref{rmk:weight-unique}).
\end{proof} 

\begin{rmk}
This proposition of course only makes explicit the notion of equivalences internal to the bicategory $\cProf(\cSp)$ of spectral bimodules. The point is that the strong stable equivalences discussed in this paper give rise to such representing invertible spectral bimodules. In that context we also refer to these bimodules as \textbf{universal tilting modules} for the following two reasons.
\begin{enumerate}
\item The word ``universal'' alludes to the fact that invertible bimodules realize simultaneous symmetries in all stable derivators. 
\item These spectral bimodules are spectral refinements of the more classical tilting complexes. In more detail, for every commutative ring $k$ extensions of scalars along $\lZ\to k$ defines a cocontinuous monoidal morphism $k\otimes-\colon\D_\lZ\to\D_k$. Moreover, let $H\lZ$ be the integral Eilenberg--MacLane spectrum (realized as a symmetric ring spectrum) and let $\D_{H\lZ}$ be the derivator of $H\lZ$-module spectra. There is a zigzag of weakly monoidal Quillen equivalences relating $\Ch(\lZ)$ and $\mathrm{Mod}(H\lZ)$ \cite{shipley:spectra-dga}, and this induces a monoidal equivalence of derivators $\D_{H\lZ}\simeq\D_{\lZ}$. As a final ingredient, inducing up along $\lS\to H\lZ$ yields a cocontinuous monoidal morphism $H\lZ\wedge-\colon\cSp\to\D_{H\lZ}$, and we end up with a cocontinuous monoidal morphism of derivators
\begin{equation}\label{eq:smash-with-field}
k\otimes-\colon\cSp\stackrel{H\lZ\wedge-}{\to}\D_{H\lZ}\simeq\D_\lZ\stackrel{k\otimes-}{\to}\D_k.
\end{equation}
In the case of a field $k$, an application of this morphism to invertible spectral bimodules yields tilting complexes as in tilting theory (see for example \cite{apr:tilting,brenner-butler:tilting,happel-ringel:tilted-algebras,rickard:derived-fun,keller:deriving-dg} or the survey articles in \cite{angeleri-happel-krause:handbook}).
\end{enumerate}
\end{rmk}

We illustrate this notion of invertible spectral bimodules by some examples.

\begin{eg}\label{eq:susp-tilting}
For every stable derivator the morphisms $\Sigma^n\colon\D\to\D$ for $n\in\lZ$ are invertible (\autoref{thm:stable}) and weighted colimits (\autoref{eg:weight-susp}). By \autoref{prop:invertible} the corresponding weights $\Sigma^n\lS\in\cSp(\ast)$ are invertible. 
\end{eg}

\begin{eg}\label{eg:cof-tilting}
For every stable derivator \D the iterated cofiber morphisms
\[
\cof^n\colon\D^{[1]}\to\D^{[1]}
\]
are invertible for $n\in\lZ$ (\autoref{thm:stable}) and these are by \autoref{eg:weight-cof-seq} weighted colimits. In more detail, the weights for $\cof$ and $\fib$ are respectively given by
\begin{equation}\label{eq:weights-cof-fib}
\vcenter{
\xymatrix{
\lS\ar[r] &0 && 0\ar[r] \ar@{}[rd]|{\square}&\lS\\
\lS\ar[r]\ar[u] &\lS,\ar[u] && \Omega\lS\ar[r]\ar[u] &0,\ar[u]
}
}
\end{equation}
and these bimodules are inverse to each other (\autoref{prop:invertible}). 
\end{eg}

\begin{eg}
By \autoref{cor:sse-A3} all $A_3$-quivers are strongly stably equivalent. The proof of this statement shows that these strong stable equivalences are weighted colimits which by \autoref{prop:invertible} give rise to invertible spectral bimodules. To consider a specific example, let $Q=\ulcorner=(\bullet\ot\bullet\to\bullet)$ be the source of valence two and let $\A{3}=(\bullet\to\bullet\to\bullet)$ be the linearly oriented $A_3$-quiver. We explicitly describe the universal tilting modules associated to the strong stable equivalence $\Phi\colon\A{3}\sse Q$ and its inverse as constructed in \autoref{eg:sse-A3-2}.

The universal tilting module $T_{Q,\A{3}}\in\cSp(Q\times\A{3}\op)$ is obtained from the Yoneda bimodule $\lI_{\A{3}}$ by an application of the strong stable equivalence $\Phi\colon\cSp^{\A{3}}\to\cSp^Q$ (\autoref{rmk:thms-on-weights}). The bimodule $\lI_{\A{3}}$ takes by \autoref{eg:Yoneda-poset} the form as shown on the left in:
\[
\xymatrix{
&\lS\ar[r]&\lS\ar[r]&\lS && &0&\lS\ar[r]\ar[l]&\lS\\
\lI_{\A{3}}\colon&0\ar[r]\ar[u]&\lS\ar[r]\ar[u]&\lS\ar[u] && T_{Q,\A{3}}\colon&\lS\ar[u]&\lS\ar[r]\ar[l]\ar[u]&\lS\ar[u]\\
&0\ar[r]\ar[u]&0\ar[r]\ar[u]&\lS\ar[u] && &0\ar[u]&0\ar[r]\ar[l]\ar[u]&\lS\ar[u]
}
\]
\noindent
The strong stable equivalence $\Phi$ simply forms the cofiber of the first morphism, and this amounts in the three cases to
\[
\xymatrix{
\lS\ar[r]\ar[d]\ar@{}[rd]|{\square}&\lS\ar[d]\ar[r]&\lS&
0\ar[r]\ar[d]\ar@{}[rd]|{\square}&\lS\ar[d]\ar[r]&\lS& 
0\ar[r]\ar[d]\ar@{}[rd]|{\square}&0\ar[d]\ar[r]&\lS\\
0\ar[r]&0,& & 
0\ar[r]&\lS,& & 
0\ar[r]&0.& 
}
\]
Consequently, for every stable derivator \D there is a natural isomorphism 
\[
\Phi_\D\cong T_{Q,\A{3}}\otimes_{[\A{3}]}-\colon\D^{\A{3}}\toiso\D^Q,
\]
where $T_{Q,\A{3}}$ takes the form as displayed in the above diagram on the right.

Similarly, also the inverse strong stable equivalence $\Phi^{-1}\colon Q\sse\A{3}$ is given by a universal tilting module $T_{\A{3},Q}\in\cSp(\A{3}\times Q\op)$. This invertible spectral bimodule is obtained from $\lI_Q$ by an application of $\Phi^{-1}$ which amounts to forming fibers of the morphisms pointing from right to left. We leave it to the reader to double-check that $\lI_Q$ and $T_{\A{3},Q}$ look like:
\[
\xymatrix{
&\lS\ar[d]&0\ar[r]\ar[l]\ar[d]&0\ar[d] && &\Omega\lS\ar[r]\ar[d]&0\ar[d]\ar[r]&0\ar[d]\\
\lI_{Q}\colon&\lS&\lS\ar[r]\ar[l]&\lS && T_{\A{3},Q}\colon&0\ar[r]&\lS\ar[r]&\lS\\
&0\ar[u]&0\ar[u]\ar[l]\ar[r]&\lS\ar[u] && &0\ar[u]\ar[r]&0\ar[r]\ar[u]&\lS\ar[u]
}
\]
These universal tilting modules $T_{Q,\A{3}}$ and $T_{\A{3},Q}$ are inverse to each other.
\end{eg}

Additional interesting examples of universal tilting bimodules arise from the Coxeter and the Serre functors discussed in \S\ref{subsec:Dynkin-A}. 

\begin{eg}\label{eg:Serre-A2}
For every stable derivator \D the Serre functor $S\colon\D^{\A{2}}\toiso\D^{\A{2}}$ is naturally isomorphic to the cofiber morphism $\cof$ (\autoref{egs:Serre-An}). The universal tilting bimodule is by \autoref{eg:cof-tilting} given by 
\[
\xymatrix{
\lS\ar[r] &0\\
\lS\ar[r]\ar[u] &\lS.\ar[u]
}
\]
\end{eg}

We invite the reader to directly calculate the universal tilting bimodules for the Serre functors in the cases $\A{3}$ and $Q=(\bullet\ot\bullet\to\bullet)$ (see \autoref{eg:Coxeter-A3-linear} and \autoref{eg:Coxeter-span}). Here, instead we turn to the following different perspective on these bimodules. Let us recall from our discussion of opposite derivators that the Yoneda bimodule $\lI_A\in\cSp(A\times A\op)$ can also be considered as an object in
\[
\cSp(A\times A\op)\op\cong\cSp\op((A\times A\op)\op)\cong\cSp\op(A\op\times A).
\]
The point of this reformulation is that the sphere spectrum $\lS$ defines a morphism $-\rhd\lS\colon\cSp\op\to\cSp$.

\begin{defn} \label{defn:dualizing}
For every small category $A$ the \textbf{canonical duality bimodule} $D_A\in\cSp(A\times A\op)$ is the image of the Yoneda bimodule~$\lI_A\in\cSp(A\times A\op)$ under
\[
\cSp\op(A\op\times A)\stackrel{-\rhd\lS}{\to}\cSp(A\op\times A)\stackrel{\cong}{\to}\cSp(A\times A\op).
\]
\end{defn}

\begin{eg}\label{eg:dualizing-poset}
For every poset $P$ the canonical duality bimodule $D_P$ admits the following description. Let $P\times P\op$ be endowed with the product order and let $\Delta_P\subseteq P\times P\op$ be the diagonal. The bimodule $D_P\in\cSp(P\times P\op)$ restricted to the down-set generated by $\Delta_P$ is constant with value the sphere spectrum $\lS$ and it vanishes on the complement \cite[Lem.~7.4]{gst:Dynkin-A}.
\end{eg}

We illustrate this class of examples in the following specific cases. As in previous cases, for the bimodules $D_P\in\cSp(P\times P\op)$ the $P$-coordinate is drawn horizontally and the $P\op$-coordinate vertically.

\begin{eg}\label{eg:dualizing-A2}
The canonical duality bimodule $D_{\A{2}}\in\cSp(\A{2}\times\A{2}\op)$ looks like
\[
\xymatrix{
\lS\ar[r] &0\\
\lS\ar[r]\ar[u] &\lS.\ar[u]
}
\]
\end{eg}

\begin{eg}\label{eg:dualizing-A3}
The canonical duality bimodule $D_{\A{3}}\in\cSp(\A{3}\times\A{3}\op)$ takes the form
\[
\xymatrix{
\lS\ar[r]&0\ar[r]&0\\
\lS\ar[u]\ar[r]&\lS\ar[u]\ar[r]&0\ar[u]\\
\lS\ar[u]\ar[r]&\lS\ar[r]\ar[u]&\lS.\ar[u]
}
\]
\end{eg}

\begin{eg}\label{eg:dualizing-span}
Let $Q=(\bullet\ot\bullet\to\bullet)$ be the source of valence two. The canonical duality bimodule $D_Q\in\cSp(Q\times Q\op)$ is given by
\[
\xymatrix{
\lS\ar[d]&\lS\ar[r]\ar[l]\ar[d]&0\ar[d] \\
0&\lS\ar[r]\ar[l]&0\\
0\ar[u]&\lS\ar[u]\ar[l]\ar[r]&\lS\ar[u]
}
\]
\end{eg}

\begin{defn}\label{defn:nakayama}
Let \D be a stable derivator. The \textbf{Nakayama functor} associated to $A\in\cCat$ is
\[
D_A\otimes_{[A]}-\colon\D^A\to\D^A.
\]
\end{defn}

This definition is of course motivated by the corresponding situation in representation theory (see for instance \cite[\S4.6]{happel:triangulated}). Note that in the case of $\A{2}$ the Serre functor and the Nakayama functor are naturally isomorphic (\autoref{eg:Serre-A2} and \autoref{eg:dualizing-A2}). This holds in full generality.

\begin{thm}\label{thm:serre-nakayama}
Let \D be a stable derivator and let $Q$ be a Dynkin quiver of type~$A$. The Serre functor and the Nakayama functor are naturally isomorphic
\[
S\cong D_Q\otimes_{[Q]}-\colon\D^Q\toiso\D^Q.
\]
\end{thm}
\begin{proof}
This theorem is in \cite{gst:Dynkin-A}. The special case of linearly oriented Dynkin quivers is established by direct calculation as \cite[Thm.~7.7]{gst:Dynkin-A}. And the general case is deduced from it as an application of the general calculus of bimodules \cite[Thm.~10.9]{gst:Dynkin-A}.
\end{proof}

Additional universal tilting bimodules arise from more general reflection functors (as in \S\ref{subsec:reflection}). We conclude this paper by the following remark offering a different perspective on the project on abstract representation theory.

\begin{rmk}\label{rmk:Picard}
A different way of thinking of abstract representation theory is that it contributes to the calculation of spectral Picard groups (or groupoids). For every closed monoidal derivator \V there is the bicategory $\cProf(\V)$ of bimodules (\autoref{thm:bicategory}). In particular, for every small category $A$ we obtain the monoidal category
\[
\cProf(\V)(A,A)=\V(A\times A\op)
\]
of $(A,A)$-bimodules with values in \V. The monoidal structure is the canceling tensor product $\otimes_{[A]}$ and the monoidal unit is the Yoneda bimodule $\lI_A$. Correspondingly, we define the \textbf{Picard group $\mathrm{Pic}_\V(A)$ of $A$ relative to \V} as the Picard group of the monoidal category $\cProf(\V)(A,A)$,
\[
\mathrm{Pic}_\V(A)=\mathrm{Pic}\big(\cProf(\V)(A,A)\big).
\]
Elements are simply isomorphism classes of invertible bimodules and the isomorphism class of $\lI_A$ is the neutral element. Particularly interesting cases arise for the stable, closed monoidal derivators $\cSp, \D_\lZ,$ and $\D_k$ for arbitrary commutative rings~$k$. Correspondingly, we obtain the \textbf{spectral Picard group} $\mathrm{Pic}_{\cSp}(A)$, the \textbf{integral Picard group} $\mathrm{Pic}_{\D_\lZ}(A)$, and also the \textbf{$k$-linear Picard group} $\mathrm{Pic}_{\D_k}(A)$. In particular, in the case of a field $k$ for special choices of $A$ the group $\mathrm{Pic}_{\D_k}(A)$ agrees with the \emph{derived Picard group} of Miyachi and Yekutieli~\cite{miyachi-yekutieli}. For every stable, closed monoidal derivator \V the suspension of the monoidal unit $\Sigma\lS$ defines an element in $\mathrm{Pic}_\V(\ast)$ (\autoref{eg:weight-susp}).

To the best of the knowledge of the author, the only shape for which the spectral Picard group is known is $A=\ast$. In that case, the Picard group $\mathrm{Pic}_{\cSp}(\ast)$ is the Picard group of the stable homotopy category which is isomorphic to $\lZ$ with $\Sigma\lS$ as generator (\cite{HMS:Picard} or \cite[Thm.~2.2]{strickland:interpolation}). Every construction of a strong stable equivalence $A\sse A$ and its universal tilting bimodule $T_{A,A}$ yields an element in $\mathrm{Pic}_{\cSp}(A)$, and the hope is that at least for some shapes $A$ this leads to a calculation of the spectral Picard group. 

One nice feature of these Picard groups is their functoriality in $\V$. In fact, every cocontinuous, monoidal morphism $F\colon\V\to\W$ of monoidal derivators induces a morphism of bicategories $F\colon\cProf(\V)\to\cProf(\W)$ (as follows from the fact that cocontinuous morphisms preserve coends). In particular, for every $A\in\cCat$ we obtain a monoidal functor $\cProf(\V)(A,A)\to\cProf(\W)(A,A)$ and hence a group homomorphism $F\colon\mathrm{Pic}_\V(A)\to\mathrm{Pic}_\W(A)$. A particularly interesting case is given by the morphism \eqref{eq:smash-with-field}. Hence, associated to every commutative ring $k$ a group homomorphism
\[
k\otimes-\colon\mathrm{Pic}_{\cSp}(A)\to\mathrm{Pic}_{\D_k}(A).
\]

This functoriality of Picard groups is calculationally useful. For instance,in combination with calculations of derived Picard groups by Miyachi and Yekutieli \cite[Theorem~4.1]{miyachi-yekutieli} and an abstract fractionally Calabi--Yau property of Dynkin quivers of type $A$ \cite[Cor.~5.20]{gst:Dynkin-A}, this was used to show that for every such quiver $Q$ the homomorphism
\[
k\otimes-\colon\mathrm{Pic}_{\cSp}(Q)\to\mathrm{Pic}_{\D_k}(Q)
\]
is a split epimorphism \cite[Thm.~12.6]{gst:Dynkin-A}. Conjecturally, there is no kernel, and, jointly with Jan {\v S}{\v t}ov{\'\i}{\v c}ek, we intend to come back to this and related calculations elsewhere. 
\end{rmk}

%Final list of TODOs for ICRA notes:
%\begin{enumerate}
%\item spell-check
%\item check all refs, citations, and overfull boxes
%\item check citations to own papers by comparing with published versions!
%\end{enumerate}
%
%Additional list of TODOs for habilitation thesis:
%\begin{enumerate}
%\item add subsection in \S\ref{sec:ART} on abstract cubical homotopy theory 
%\item add subsection in \S\ref{sec:ART} on characterizations of stability and generalizations of it (relative version of abstract rep thy?)
%\item add subsection in \S\ref{sec:ART} on calculation of spectral Picard groups (depending on progress)
%\item add subsection in \S\ref{sec:ART} on global Serre dualities (depending on progress)
%\item Incorporate these subsections in main introduction
%\end{enumerate}

\bibliographystyle{alpha}
\bibliography{symm}

\newcommand{\etalchar}[1]{$^{#1}$}
\def\cprime{$'$}
\begin{thebibliography}{EKMM97}

\bibitem[ABG{\etalchar{+}}09]{ABGHR:units}
Matthew Ando, Andrew~J. Blumberg, David Gepner, Michael~J. Hopkins, and Charles
  Rezk.
\newblock Units of ring spectra and {T}hom spectra.
\newblock arXiv:0506589, 2009.

\bibitem[ABG10]{ABG:twists}
Matthew Ando, Andrew~J. Blumberg, and David Gepner.
\newblock Twists of {$K$}-theory and {TMF}.
\newblock In {\em Superstrings, geometry, topology, and {$C^\ast$}-algebras},
  volume~81 of {\em Proc. Sympos. Pure Math.}, pages 27--63. Amer. Math. Soc.,
  Providence, RI, 2010.

\bibitem[AHHK07]{angeleri-happel-krause:handbook}
Lidia Angeleri~H{\"u}gel, Dieter Happel, and Henning Krause, editors.
\newblock {\em Handbook of tilting theory}, volume 332 of {\em London
  Mathematical Society Lecture Note Series}.
\newblock Cambridge University Press, Cambridge, 2007.

\bibitem[And79]{anderson:axiomatic}
D.~W. Anderson.
\newblock Axiomatic homotopy theory.
\newblock In {\em Algebraic topology, {W}aterloo, 1978 ({P}roc. {C}onf.,
  {U}niv. {W}aterloo, {W}aterloo, {O}nt., 1978)}, volume 741 of {\em Lecture
  Notes in Math.}, pages 520--547. Springer, Berlin, 1979.

\bibitem[APR79]{apr:tilting}
Maurice Auslander, Mar{\'{\i}}a~In{\'e}s Platzeck, and Idun Reiten.
\newblock Coxeter functors without diagrams.
\newblock {\em Trans. Amer. Math. Soc.}, 250:1--46, 1979.

\bibitem[ASS06]{assem-simson-skowronski}
Ibrahim Assem, Daniel Simson, and Andrzej Skowro\'{n}ski.
\newblock {\em Elements of the representation theory of associative algebras.
  {V}ol. 1}, volume~65 of {\em London Mathematical Society Student Texts}.
\newblock Cambridge University Press, Cambridge, 2006.
\newblock Techniques of representation theory.

\bibitem[Bal10]{balmer:TTG}
Paul Balmer.
\newblock Tensor triangular geometry.
\newblock In {\em Proceedings of the {I}nternational {C}ongress of
  {M}athematicians. {V}olume {II}}, pages 85--112. Hindustan Book Agency, New
  Delhi, 2010.

\bibitem[Bal11]{balmer:separability}
Paul Balmer.
\newblock Separability and triangulated categories.
\newblock {\em Adv. Math.}, 226(5):4352--4372, 2011.

\bibitem[BB80]{brenner-butler:tilting}
Sheila Brenner and M.~C.~R. Butler.
\newblock Generalizations of the {B}ernstein--{G}el\cprime fand--{P}onomarev
  reflection functors.
\newblock In {\em Representation theory, {II} ({P}roc. {S}econd {I}nternat.
  {C}onf., {C}arleton {U}niv., {O}ttawa, {O}nt., 1979)}, volume 832 of {\em
  Lecture Notes in Math.}, pages 103--169. Springer, Berlin-New York, 1980.

\bibitem[BBD82]{beilinson:perverse}
Alexander Be{\u\i}linson, Joseph Bernstein, and Pierre Deligne.
\newblock Faisceaux pervers.
\newblock In {\em Analysis and topology on singular spaces, {I} ({L}uminy,
  1981)}, volume 100 of {\em Ast\'erisque}, pages 5--171. Soc. Math. France,
  Paris, 1982.

\bibitem[BCR97]{benson-rickard-carlson:thick-stmod}
D.~J. Benson, Jon~F. Carlson, and Jeremy Rickard.
\newblock Thick subcategories of the stable module category.
\newblock {\em Fund. Math.}, 153(1):59--80, 1997.

\bibitem[Bec14]{becker:models-singularity}
Hanno Becker.
\newblock Models for singularity categories.
\newblock {\em Adv. Math.}, 254:187--232, 2014.

\bibitem[Bec18]{beckert:thesis}
Falk Beckert.
\newblock {\em The bivariant parasimplicial $S_\bullet$-construction}.
\newblock PhD thesis, Bergische Universit{\"a}t Wuppertal, 2018.

\bibitem[Bek00]{beke:sheafifiable}
Tibor Beke.
\newblock Sheafifiable homotopy model categories.
\newblock {\em Math. Proc. Cambridge Philos. Soc.}, 129(3):447--475, 2000.

\bibitem[B{\'e}n67]{benabou:intro}
Jean B{\'e}nabou.
\newblock Introduction to bicategories.
\newblock In {\em Reports of the {M}idwest {C}ategory {S}eminar}, pages 1--77.
  Springer, Berlin, 1967.

\bibitem[Ber07]{bergner:scat}
Julia~E. Bergner.
\newblock A model category structure on the category of simplicial categories.
\newblock {\em Trans. Amer. Math. Soc.}, 359(5):2043--2058 (electronic), 2007.

\bibitem[Ber10]{bergner:survey}
Julia~E. Bergner.
\newblock A survey of {$(\infty,1)$}-categories.
\newblock In {\em Towards higher categories}, volume 152 of {\em IMA Vol. Math.
  Appl.}, pages 69--83. Springer, New York, 2010.

\bibitem[BF78]{bousfield-friedlander}
A.~K. Bousfield and E.~M. Friedlander.
\newblock Homotopy theory of {$\Gamma $}-spaces, spectra, and bisimplicial
  sets.
\newblock In {\em Geometric applications of homotopy theory ({P}roc. {C}onf.,
  {E}vanston, {I}ll., 1977), {II}}, volume 658 of {\em Lecture Notes in Math.},
  pages 80--130. Springer, Berlin, 1978.

\bibitem[BG99]{becker-gottlieb:duality}
James~C. Becker and Daniel~Henry Gottlieb.
\newblock A history of duality in algebraic topology.
\newblock In {\em History of topology}, pages 725--745. North-Holland,
  Amsterdam, 1999.

\bibitem[BG18]{bg:cubical}
Falk Beckert and Moritz Groth.
\newblock Abstract cubical homotopy theory.
\newblock \url{https://arxiv.org/abs/1803.06022}, 2018.
\newblock Preprint.

\bibitem[BGP73]{bernstein-gelfand-ponomarev:Coxeter}
I.~N. Bern{\v{s}}te{\u\i}n, I.~M. Gel{\cprime}fand, and V.~A. Ponomarev.
\newblock Coxeter functors, and {G}abriel's theorem.
\newblock {\em Uspehi Mat. Nauk}, 28(2(170)):19--33, 1973.

\bibitem[BIK11]{bik:stratification-finite-gp}
David~J. Benson, Srikanth~B. Iyengar, and Henning Krause.
\newblock Stratifying modular representations of finite groups.
\newblock {\em Ann. of Math. (2)}, 174(3):1643--1684, 2011.

\bibitem[BK72]{bousfield-kan:htpy-limits}
Aldridge~Knight Bousfield and Daniel~Marinus Kan.
\newblock {\em Homotopy limits, completions and localizations}.
\newblock Lecture Notes in Mathematics, Vol. 304. Springer-Verlag, Berlin,
  1972.

\bibitem[BK89]{bondal-kapranov:serre}
A.~I. Bondal and M.~M. Kapranov.
\newblock Representable functors, {S}erre functors, and reconstructions.
\newblock {\em Izv. Akad. Nauk SSSR Ser. Mat.}, 53(6):1183--1205, 1337, 1989.

\bibitem[BK90]{bondal-kapranov:framed}
A.~I. Bondal and M.~M. Kapranov.
\newblock Framed triangulated categories.
\newblock {\em Mat. Sb.}, 181(5):669--683, 1990.

\bibitem[BLM08]{BLM:Aoo}
Yu. Bespalov, V.~Lyubashenko, and O.~Manzyuk.
\newblock {\em Pretriangulated {$A_\infty$}-categories}, volume~76 of {\em
  {Prats\=\i\ \=Institutu Matematiki Nats\=\i onal\cprime no\"\i\ Akadem\=\i
  \"\i\ Nauk Ukra\"\i ni. Matematika ta \"\i \"\i\ Zastosuvannya} [Proceedings
  of Institute of Mathematics of NAS of Ukraine. Mathematics and its
  Applications]}.
\newblock Nats\=\i onal\cprime na Akadem\=\i ya Nauk Ukra\"\i ni \=Institut
  Matematiki, Kiev, 2008.

\bibitem[BO01]{bondal-orlov:reconstruction}
Alexei Bondal and Dmitri Orlov.
\newblock Reconstruction of a variety from the derived category and groups of
  autoequivalences.
\newblock {\em Compositio Math.}, 125(3):327--344, 2001.

\bibitem[Boa64]{boardman:thesis}
Michael Boardman.
\newblock {\em {S}table {H}omotopy {T}heory and {S}ome {A}pplications}.
\newblock PhD thesis, University of Cambridge, 1964.

\bibitem[Bor94a]{borceux:1}
Francis Borceux.
\newblock {\em Handbook of categorical algebra. 1}, volume~50 of {\em
  Encyclopedia of Mathematics and its Applications}.
\newblock Cambridge University Press, Cambridge, 1994.
\newblock Basic category theory.

\bibitem[Bor94b]{borceux:2}
Francis Borceux.
\newblock {\em Handbook of categorical algebra. 2}, volume~51 of {\em
  Encyclopedia of Mathematics and its Applications}.
\newblock Cambridge University Press, Cambridge, 1994.
\newblock Categories and structures.

\bibitem[BV73]{boardman-vogt}
Michael Boardman and Rainer Vogt.
\newblock {\em Homotopy invariant algebraic structures on topological spaces}.
\newblock Lecture Notes in Mathematics, Vol. 347. Springer-Verlag, Berlin,
  1973.

\bibitem[Cam13]{camerona:whirlwind}
Omar~Antolin Camerona.
\newblock {A} {W}hirlwind tour of the {W}orld of $(\infty,1)$-categories.
\newblock \url{http://arxiv.org/abs/1303.4669}, 2013.
\newblock preprint.

\bibitem[CD09]{cisinski-deglise:local-stable}
Denis-Charles Cisinski and Fr{\'e}d{\'e}ric D{\'e}glise.
\newblock Local and stable homological algebra in {G}rothendieck abelian
  categories.
\newblock {\em Homology, Homotopy Appl.}, 11(1):219--260, 2009.

\bibitem[Che11]{chen:Serre}
Xiao-Wu Chen.
\newblock Generalized {S}erre duality.
\newblock {\em J. Algebra}, 328:268--286, 2011.

\bibitem[Cis03]{cisinski:direct}
Denis-Charles Cisinski.
\newblock Images directes cohomologiques dans les cat\'egories de mod\`eles.
\newblock {\em Ann. Math. Blaise Pascal}, 10(2):195--244, 2003.

\bibitem[Cis04]{cisinski:loc-min}
Denis-Charles Cisinski.
\newblock Le localisateur fondamental minimal.
\newblock {\em Cah. Topol. G\'eom. Diff\'er. Cat\'eg.}, 45(2):109--140, 2004.

\bibitem[Cis06]{cisinski:presheaves}
Denis-Charles Cisinski.
\newblock Les pr\'efaisceaux comme mod\`eles des types d'homotopie.
\newblock {\em Ast\'erisque}, (308):xxiv+390, 2006.

\bibitem[Cis08]{cisinski:derived-kan}
Denis-Charles Cisinski.
\newblock Propri\'et\'es universelles et extensions de {K}an d\'eriv\'ees.
\newblock {\em Theory Appl. Categ.}, 20:No. 17, 605--649, 2008.

\bibitem[Cis10]{cisinski:derivable}
Denis-Charles Cisinski.
\newblock Cat\'egories d\'erivables.
\newblock {\em Bull. Soc. Math. France}, 138(3):317--393, 2010.

\bibitem[CS02]{chacholski-scherer:diagram}
Wojciech Chach{\'o}lski and J{\'e}r{\^o}me Scherer.
\newblock Homotopy theory of diagrams.
\newblock {\em Mem. Amer. Math. Soc.}, 155(736):x+90, 2002.

\bibitem[CT11]{cisinski-tabuada:non-connective}
Denis-Charles Cisinski and Gon{\c{c}}alo Tabuada.
\newblock Non-connective {$K$}-theory via universal invariants.
\newblock {\em Compos. Math.}, 147(4):1281--1320, 2011.

\bibitem[CT12]{cisinski-tabuada:non-commutative}
Denis-Charles Cisinski and Gon{\c{c}}alo Tabuada.
\newblock Symmetric monoidal structure on non-commutative motives.
\newblock {\em J. K-Theory}, 9(2):201--268, 2012.

\bibitem[DJW18]{dycker-jasso-walde:higher-AR}
Tobias Dyckerhoff, Gustavo Jasso, and Tashi Walde.
\newblock Simplicial structures in higher {A}uslander--{R}eiten theory.
\newblock \url{http://arxiv.org/pdf/1811.02461}, 2018.
\newblock preprint.

\bibitem[DJW19]{dycker-jasso-walde:BGP}
Tobias Dyckerhoff, Gustavo Jasso, and Tashi Walde.
\newblock Generalised {BGP} reflection functors via the {G}rothendieck
  construction.
\newblock \url{http://arxiv.org/pdf/1901.06993}, 2019.
\newblock preprint.

\bibitem[DP80]{dold-puppe:duality}
Albrecht Dold and Dieter Puppe.
\newblock Duality, trace, and transfer.
\newblock In {\em Proceedings of the International Conference on Geometric
  Topology (Warsaw, 1978)}, pages 81--102, Warsaw, 1980. PWN.

\bibitem[Dug01a]{dugger:combinatorial}
Daniel Dugger.
\newblock Combinatorial model categories have presentations.
\newblock {\em Adv. Math.}, 164(1):177--201, 2001.

\bibitem[Dug01b]{dugger:universal}
Daniel Dugger.
\newblock Universal homotopy theories.
\newblock {\em Adv. Math.}, 164(1):144--176, 2001.

\bibitem[EE05]{enochs-estrada:relative}
Edgar Enochs and Sergio Estrada.
\newblock Relative homological algebra in the category of quasi-coherent
  sheaves.
\newblock {\em Adv. Math.}, 194(2):284--295, 2005.

\bibitem[EH62]{eckmann-hilton:group-like-I}
B.~Eckmann and P.~J. Hilton.
\newblock Group-like structures in general categories. {I}. {M}ultiplications
  and comultiplications.
\newblock {\em Math. Ann.}, 145:227--255, 1961/1962.

\bibitem[EJ00]{enochs-jenda:relative}
Edgar~E. Enochs and Overtoun M.~G. Jenda.
\newblock {\em Relative homological algebra}, volume~30 of {\em de Gruyter
  Expositions in Mathematics}.
\newblock Walter de Gruyter \& Co., Berlin, 2000.

\bibitem[EKMM97]{ekmm:rings}
A.~D. Elmendorf, I.~Kriz, M.~A. Mandell, and J.~P. May.
\newblock {\em Rings, modules, and algebras in stable homotopy theory},
  volume~47 of {\em Mathematical Surveys and Monographs}.
\newblock American Mathematical Society, Providence, RI, 1997.
\newblock With an appendix by M. Cole.

\bibitem[Fer06]{ferrand:nonadd}
Daniel Ferrand.
\newblock On the non-additivity of traces in derived categories.
\newblock arXiv:0506589, 2006.

\bibitem[Fra96]{franke:adams}
Jens Franke.
\newblock Uniqueness theorems for certain triangulated categories with an
  {A}dams spectral sequence, 1996.
\newblock Preprint.

\bibitem[Fre09]{fresse:modules}
Benoit Fresse.
\newblock {\em Modules over operads and functors}, volume 1967 of {\em Lecture
  Notes in Mathematics}.
\newblock Springer-Verlag, Berlin, 2009.

\bibitem[Gab72]{gabriel:unzerlegbare}
Peter Gabriel.
\newblock Unzerlegbare {D}arstellungen. {I}.
\newblock {\em Manuscripta Math.}, 6:71--103; correction, ibid. 6 (1972), 309,
  1972.

\bibitem[GAdS14]{gallauer:traces}
Martin Gallauer Alves~de Souza.
\newblock Traces in monoidal derivators, and homotopy colimits.
\newblock {\em Adv. Math.}, 261:26--84, 2014.

\bibitem[Gar05]{garkusha:II}
Grigory Garkusha.
\newblock Systems of diagram categories and {K}-theory. {II}.
\newblock {\em Math. Z.}, 249(3):641--682, 2005.

\bibitem[Gar06]{garkusha:I}
G.~Garkusha.
\newblock Systems of diagram categories and {$K$}-theory. {I}.
\newblock {\em Algebra i Analiz}, 18(6):131--186, 2006.

\bibitem[GGN15]{ggn:infinite}
David Gepner, Moritz Groth, and Thomas Nikolaus.
\newblock Universality of multiplicative infinite delooping machines.
\newblock {\em Algebr. Geom. Topol.}, 15(6):3107--3153, 2015.

\bibitem[Gil11]{gillespie:exact}
James Gillespie.
\newblock Model structures on exact categories.
\newblock {\em J. Pure Appl. Algebra}, 215(12):2892--2902, 2011.

\bibitem[GM03]{gelfand-manin:homological}
Sergei~I. Gelfand and Yuri~I. Manin.
\newblock {\em Methods of homological algebra}.
\newblock Springer Monographs in Mathematics. Springer-Verlag, Berlin, second
  edition, 2003.

\bibitem[Goo92]{goodwillie:II}
Thomas~G. Goodwillie.
\newblock Calculus. {II}. {A}nalytic functors.
\newblock {\em $K$-Theory}, 5(4):295--332, 1991/92.

\bibitem[GPS14a]{gps:additivity}
Moritz Groth, Kate Ponto, and Michael Shulman.
\newblock The additivity of traces in monoidal derivators.
\newblock {\em {J}ournal of {K}-{T}heory}, 14:422--494, 2014.

\bibitem[GPS14b]{gps:mayer}
Moritz Groth, Kate Ponto, and Michael Shulman.
\newblock {M}ayer--{V}ietoris sequences in stable derivators.
\newblock {\em Homology, Homotopy Appl.}, 16:265--294, 2014.

\bibitem[Gro]{grothendieck:derivators}
Alexander Grothendieck.
\newblock Les d\'erivateurs.
\newblock \\\url{http://www.math.jussieu.fr/~maltsin/groth/Derivateurs.html}.
\newblock Manuscript.

\bibitem[Gro83]{grothendieck:stacks}
A.~Grothendieck.
\newblock Pursuing stacks.
\newblock manuscript, to appear in {D}ocuments {M}ath\'ematics, 1983.

\bibitem[Gro10]{groth:scinfinity}
Moritz Groth.
\newblock A short course on $\infty$-categories.
\newblock \url{http://arxiv.org/abs/1007.2925}, 2010.
\newblock Preprint. To appear in \emph{Handbook of homotopy theory} edited by
  Haynes Miller.

\bibitem[Gro13]{groth:ptstab}
Moritz Groth.
\newblock Derivators, pointed derivators, and stable derivators.
\newblock {\em Algebr. Geom. Topol.}, 13:313--374, 2013.

\bibitem[Gro16]{groth:revisit}
Moritz Groth.
\newblock Revisiting the canonicity of canonical triangulations.
\newblock arXiv:1602.04846, 2016.

\bibitem[Gro18]{groth:thy-of-der}
Moritz Groth.
\newblock The theory of derivators, 2018.
\newblock Book project in preparation. Draft versions of parts of it are
  available under \url{http://www.math.uni-bonn.de/~mgroth/monos.htmpl}.

\bibitem[G{\v S}14]{gst:basic}
Moritz Groth and Jan {\v S}{\v t}ov{\'\i}{\v c}ek.
\newblock Tilting theory via stable homotopy theory.
\newblock To appear in {\em Crelle's Journal}. Available at arXiv:1401.6451,
  2014.

\bibitem[G{\v S}15]{gst:acyclic}
Moritz Groth and Jan {\v S}{\v t}ov{\'\i}{\v c}ek.
\newblock Abstract tilting theory for quivers and related categories.
\newblock arXiv:1512.06267, 2015.

\bibitem[G{\v{S}}16a]{gst:Dynkin-A}
Moritz Groth and Jan {\v{S}}{\v{t}}ov{\'{\i}}{\v{c}}ek.
\newblock Abstract representation theory of {D}ynkin quivers of type {$A$}.
\newblock {\em Adv. Math.}, 293:856--941, 2016.

\bibitem[G{\v S}16b]{gst:tree}
Moritz Groth and Jan {\v S}{\v t}ov{\'\i}{\v c}ek.
\newblock Tilting theory for trees via stable homotopy theory.
\newblock {\em Journal of Pure and Applied Algebra}, 220(6):2324 -- 2363, 2016.

\bibitem[GS17]{gs:generalized}
Moritz Groth and Mike Shulman.
\newblock Characterizations of abstract stable homotopy theories.
\newblock available at arXiv:1704.08084, 2017.

\bibitem[GZ67]{gabriel-zisman:calculus}
Peter Gabriel and Michel Zisman.
\newblock {\em Calculus of fractions and homotopy theory}.
\newblock Ergebnisse der Mathematik und ihrer Grenzgebiete, Band 35.
  Springer-Verlag New York, Inc., New York, 1967.

\bibitem[Hap87]{happel:fd-algebra}
Dieter Happel.
\newblock On the derived category of a finite-dimensional algebra.
\newblock {\em Comment. Math. Helv.}, 62(3):339--389, 1987.

\bibitem[Hap88]{happel:triangulated}
Dieter Happel.
\newblock {\em Triangulated categories in the representation theory of
  finite-dimensional algebras}, volume 119 of {\em London Mathematical Society
  Lecture Note Series}.
\newblock Cambridge University Press, Cambridge, 1988.

\bibitem[Hel68]{heller:shc}
Alex Heller.
\newblock Stable homotopy categories.
\newblock {\em Bull. Amer. Math. Soc.}, 74:28--63, 1968.

\bibitem[Hel88]{heller:htpythies}
Alex Heller.
\newblock Homotopy theories.
\newblock {\em Mem. Amer. Math. Soc.}, 71(383):vi+78, 1988.

\bibitem[Hel97]{heller:stable}
Alex Heller.
\newblock Stable homotopy theories and stabilization.
\newblock {\em J. Pure Appl. Algebra}, 115(2):113--130, 1997.

\bibitem[Hin97]{hinich:homological}
Vladimir Hinich.
\newblock Homological algebra of homotopy algebras.
\newblock {\em Comm. Algebra}, 25(10):3291--3323, 1997.

\bibitem[Hir03]{hirschhorn:model}
Philip~Steven Hirschhorn.
\newblock {\em Model categories and their localizations}, volume~99 of {\em
  Mathematical Surveys and Monographs}.
\newblock American Mathematical Society, Providence, RI, 2003.

\bibitem[HJR10]{hjr:triangulated}
Thorsten Holm, Peter J{\o}rgensen, and Rapha{\"e}l Rouquier, editors.
\newblock {\em Triangulated categories}, volume 375 of {\em London Mathematical
  Society Lecture Note Series}.
\newblock Cambridge University Press, Cambridge, 2010.

\bibitem[HMS94]{HMS:Picard}
Michael~J. Hopkins, Mark Mahowald, and Hal Sadofsky.
\newblock Constructions of elements in {P}icard groups.
\newblock In {\em Topology and representation theory ({E}vanston, {IL}, 1992)},
  volume 158 of {\em Contemp. Math.}, pages 89--126. Amer. Math. Soc.,
  Providence, RI, 1994.

\bibitem[Hov99]{hovey:model}
Mark Hovey.
\newblock {\em Model categories}, volume~63 of {\em Mathematical Surveys and
  Monographs}.
\newblock American Mathematical Society, Providence, RI, 1999.

\bibitem[Hov01]{hovey:sheaves}
Mark Hovey.
\newblock Model category structures on chain complexes of sheaves.
\newblock {\em Trans. Amer. Math. Soc.}, 353(6):2441--2457 (electronic), 2001.

\bibitem[Hov02]{hovey:cotorsion}
Mark Hovey.
\newblock Cotorsion pairs, model category structures, and representation
  theory.
\newblock {\em Math. Z.}, 241(3):553--592, 2002.

\bibitem[HPS97]{HPS:axiomatic}
Mark Hovey, John~H. Palmieri, and Neil~P. Strickland.
\newblock Axiomatic stable homotopy theory.
\newblock {\em Mem. Amer. Math. Soc.}, 128(610):x+114, 1997.

\bibitem[HR82]{happel-ringel:tilted-algebras}
Dieter Happel and Claus~Michael Ringel.
\newblock Tilted algebras.
\newblock {\em Trans. Amer. Math. Soc.}, 274(2):399--443, 1982.

\bibitem[HS01]{hirschowitz-simpson}
A.~Hirschowitz and C.~Simpson.
\newblock Descente pour les {$n$}-champs.
\newblock \href{http://arxiv.org/abs/math/9807049}{arXiv:math/9807049v3
  [math.AG]}, 2001.

\bibitem[HSS00]{hss:symmetric}
Mark Hovey, Brooke Shipley, and Jeff Smith.
\newblock Symmetric spectra.
\newblock {\em J. Amer. Math. Soc.}, 13(1):149--208, 2000.

\bibitem[Huy06]{huybrechts:fourier}
D.~Huybrechts.
\newblock {\em Fourier--{M}ukai transforms in algebraic geometry}.
\newblock Oxford Mathematical Monographs. The Clarendon Press, Oxford
  University Press, Oxford, 2006.

\bibitem[Jar87]{jardine:presheaves}
J.~F. Jardine.
\newblock Simplicial presheaves.
\newblock {\em J. Pure Appl. Algebra}, 47(1):35--87, 1987.

\bibitem[Jar00]{jardine:motivic}
J.~F. Jardine.
\newblock Motivic symmetric spectra.
\newblock {\em Doc. Math.}, 5:445--553 (electronic), 2000.

\bibitem[Joy08]{joyal:barca}
Andr{\'e} Joyal.
\newblock The theory of quasi-categories and its applications.
\newblock Lectures at the CRM (Barcelona). Preprint, 2008.

\bibitem[JY18]{jin-yang}
Fangzhou Jin and Enlin Yang.
\newblock Kunneth formulas for motives and additivity of traces.
\newblock arXiv:1812.06441, 2018.

\bibitem[Kel91]{keller:epivalent}
Bernhard Keller.
\newblock Derived categories and universal problems.
\newblock {\em Comm. Algebra}, 19:699--747, 1991.

\bibitem[Kel94]{keller:deriving-dg}
Bernhard Keller.
\newblock Deriving {DG} categories.
\newblock {\em Ann. Sci. \'Ecole Norm. Sup. (4)}, 27(1):63--102, 1994.

\bibitem[Kel05a]{keller:orbit}
Bernhard Keller.
\newblock On triangulated orbit categories.
\newblock {\em Doc. Math.}, 10:551--581, 2005.

\bibitem[Kel05b]{kelly:enriched}
Gregory~Maxwell Kelly.
\newblock Basic concepts of enriched category theory.
\newblock {\em Repr. Theory Appl. Categ.}, pages vi+137 pp. (electronic), 2005.
\newblock Reprint of the 1982 original [Cambridge Univ. Press, Cambridge;
  MR0651714].

\bibitem[Kel07a]{keller:exact}
Bernhard Keller.
\newblock Appendice: {L}e d\'erivateur triangul\'e associ\'e \`a une
  cat\'egorie exacte.
\newblock In {\em Categories in algebra, geometry and mathematical physics},
  volume 431 of {\em Contemp. Math.}, pages 369--373. Amer. Math. Soc.,
  Providence, RI, 2007.

\bibitem[Kel07b]{keller:der-cat-and-tilt}
Bernhard Keller.
\newblock Derived categories and tilting.
\newblock In {\em Handbook of tilting theory}, volume 332 of {\em London Math.
  Soc. Lecture Note Ser.}, pages 49--104. Cambridge Univ. Press, Cambridge,
  2007.

\bibitem[Kel08]{keller:CY-triang}
Bernhard Keller.
\newblock Calabi--{Y}au triangulated categories.
\newblock In {\em Trends in representation theory of algebras and related
  topics}, EMS Ser. Congr. Rep., pages 467--489. Eur. Math. Soc., Z\"urich,
  2008.

\bibitem[KL06]{krause-le:AR}
Henning Krause and Jue Le.
\newblock The {A}uslander--{R}eiten formula for complexes of modules.
\newblock {\em Adv. Math.}, 207(1):133--148, 2006.

\bibitem[KM08]{kahn-maltsiniotis}
Bruno Kahn and Georges Maltsiniotis.
\newblock Structures de d\'erivabilit\'e.
\newblock {\em Adv. Math.}, 218(4):1286--1318, 2008.

\bibitem[KN02]{keller-neeman:D4}
Bernhard Keller and Amnon Neeman.
\newblock The connection between {M}ay's axioms for a triangulated tensor
  product and {H}appel's description of the derived category of the quiver
  {$D_4$}.
\newblock {\em Doc. Math.}, 7:535--560 (electronic), 2002.

\bibitem[Kra08]{krause:reflection}
Henning Krause.
\newblock Representations of quivers via reflection functors.
\newblock \url{http://arxiv.org/abs/math/0804.1428}, 2008.
\newblock Course notes.

\bibitem[KS74]{kelly-street:review}
G.~M. Kelly and Ross Street.
\newblock {\em Category Seminar: Proceedings Sydney Category Theory Seminar
  1972/1973}, chapter Review of the elements of 2-categories, pages 75--103.
\newblock Springer Berlin Heidelberg, Berlin, Heidelberg, 1974.

\bibitem[KS09]{kontsevich-soibelman:Aoo}
M.~Kontsevich and Y.~Soibelman.
\newblock Notes on {$A_\infty$}-algebras, {$A_\infty$}-categories and
  non-commutative geometry.
\newblock In {\em Homological mirror symmetry}, volume 757 of {\em Lecture
  Notes in Phys.}, pages 153--219. Springer, Berlin, 2009.

\bibitem[KS12]{keller-scherotzke:integral}
Bernhard Keller and Sarah Scherotzke.
\newblock The integral cluster category.
\newblock {\em Int. Math. Res. Not. IMRN}, (12):2867--2887, 2012.

\bibitem[Lac10]{lack:2-cat-companion}
Stephen Lack.
\newblock A 2-categories companion.
\newblock In {\em Towards higher categories}, volume 152 of {\em IMA Vol. Math.
  Appl.}, pages 105--191. Springer, New York, 2010.

\bibitem[Lad07a]{ladkani:posets}
Sefi Ladkani.
\newblock Universal derived equivalences of posets.
\newblock arXiv:0705.0946, 2007.

\bibitem[Lad07b]{ladkani:posets-cluster-tilting}
Sefi Ladkani.
\newblock Universal derived equivalences of posets of cluster tilting modules.
\newblock arXiv:0710.2860, 2007.

\bibitem[Lad07c]{ladkani:posets-tilting}
Sefi Ladkani.
\newblock Universal derived equivalences of posets of tilting modules.
\newblock arXiv:0708.1287, 2007.

\bibitem[Lad08]{ladkani:thesis}
Sefi Ladkani.
\newblock {\em Homological Properties of Finite Partially Ordered Sets}.
\newblock PhD thesis, The Hebrew University of Jerusalem, 2008.

\bibitem[Len10]{lenhardt:frames}
Fabian Lenhardt.
\newblock Stable frames in model categories.
\newblock \url{http://front.math.ucdavis.edu/1002.2837}, 2010.
\newblock Preprint.

\bibitem[Len17]{lenz:derivators}
Tobias Lenz.
\newblock Homotopy ({P}re-){D}erivators of {C}ofibration and
  {Q}uasi-{C}ategories.
\newblock arXiv:1712.07845, 2017.

\bibitem[LM08]{lyubashenko-mazyuk:Serre}
Volodymyr Lyubashenko and Oleksandr Manzyuk.
\newblock {$A_\infty$}-bimodules and {S}erre {$A_\infty$}-functors.
\newblock In {\em Geometry and dynamics of groups and spaces}, volume 265 of
  {\em Progr. Math.}, pages 565--645. Birkh\"auser, Basel, 2008.

\bibitem[LMSM86]{lms:equivariant}
L.~G. Lewis, Jr., J.~P. May, M.~Steinberger, and J.~E. McClure.
\newblock {\em Equivariant stable homotopy theory}, volume 1213 of {\em Lecture
  Notes in Mathematics}.
\newblock Springer-Verlag, Berlin, 1986.
\newblock With contributions by J. E. McClure.

\bibitem[Lur09]{HTT}
Jacob Lurie.
\newblock {\em Higher topos theory}, volume 170 of {\em Annals of Mathematics
  Studies}.
\newblock Princeton University Press, Princeton, NJ, 2009.

\bibitem[Lur14]{HA}
Jacob Lurie.
\newblock Higher algebra.
\newblock \url{http://www.math.harvard.edu/~lurie/}, 2014.
\newblock Preprint.

\bibitem[Mal01a]{maltsiniotis:seminar}
Georges Maltsiniotis.
\newblock Groupe de travail sur les d\'erivateurs.
\newblock \url{http://people.math.jussieu.fr/~maltsin/textes.html}, 2001.
\newblock Seminar in Paris.

\bibitem[Mal01b]{maltsiniotis:intro}
Georges Maltsiniotis.
\newblock Introduction \`a la th\'eorie des d\'erivateurs (d'apr\`es
  {G}rothendieck).
\newblock \url{http://people.math.jussieu.fr/~maltsin/textes.html}, 2001.
\newblock Preprint.

\bibitem[Mal05a]{maltsiniotis:higher}
Georges Maltsiniotis.
\newblock Cat\'egories triangul\'ees sup\'erieures.
\newblock \url{http://people.math.jussieu.fr/~maltsin/textes.html}, 2005.
\newblock Preprint.

\bibitem[Mal05b]{maltsiniotis:grothendieck}
Georges Maltsiniotis.
\newblock La th\'eorie de l'homotopie de {G}rothendieck.
\newblock {\em Ast\'erisque}, (301):vi+140, 2005.

\bibitem[Mal07]{maltsiniotis:k-theory}
Georges Maltsiniotis.
\newblock La {$K$}-th\'eorie d'un d\'erivateur triangul\'e.
\newblock In {\em Categories in algebra, geometry and mathematical physics},
  volume 431 of {\em Contemp. Math.}, pages 341--368. Amer. Math. Soc.,
  Providence, RI, 2007.

\bibitem[Mal12]{maltsiniotis:htpy-exact}
Georges Maltsiniotis.
\newblock Carr\'es exacts homotopiques et d\'erivateurs.
\newblock {\em Cah. Topol. G\'eom. Diff\'er. Cat\'eg.}, 53(1):3--63, 2012.

\bibitem[Mar83]{margolis:spectra}
Harvey~Robert Margolis.
\newblock {\em Spectra and the {S}teenrod algebra}, volume~29 of {\em
  North-Holland Mathematical Library}.
\newblock North-Holland Publishing Co., Amsterdam, 1983.
\newblock Modules over the Steenrod algebra and the stable homotopy category.

\bibitem[Mat76]{mather:pullbacks}
Michael Mather.
\newblock Pullbacks in homotopy theory.
\newblock {\em Canad. J. Math.}, 28(2):225--263, 1976.

\bibitem[May01]{may:additivity}
J.~P. May.
\newblock The additivity of traces in triangulated categories.
\newblock {\em Adv. Math.}, 163(1):34--73, 2001.

\bibitem[ML98]{maclane}
Saunders Mac~Lane.
\newblock {\em Categories for the working mathematician}, volume~5 of {\em
  Graduate Texts in Mathematics}.
\newblock Springer-Verlag, New York, second edition, 1998.

\bibitem[MM02]{mandell-may:equivariant}
M.~A. Mandell and J.~P. May.
\newblock Equivariant orthogonal spectra and {$S$}-modules.
\newblock {\em Mem. Amer. Math. Soc.}, 159(755):x+108, 2002.

\bibitem[MMSS01]{mmss:diagram}
M.~A. Mandell, J.~P. May, S.~Schwede, and B.~Shipley.
\newblock Model categories of diagram spectra.
\newblock {\em Proc. London Math. Soc. (3)}, 82(2):441--512, 2001.

\bibitem[MS06]{may-sigurdsson:parametrized}
J.~P. May and J.~Sigurdsson.
\newblock {\em Parametrized homotopy theory}, volume 132 of {\em Mathematical
  Surveys and Monographs}.
\newblock American Mathematical Society, Providence, RI, 2006.

\bibitem[MSS07]{mss:no-models}
Fernando Muro, Stefan Schwede, and Neil Strickland.
\newblock Triangulated categories without models.
\newblock {\em Invent. Math.}, 170(2):231--241, 2007.

\bibitem[Muk81]{mukai:duality}
Shigeru Mukai.
\newblock Duality between {$D(X)$}\ and {$D(\hat X)$}\ with its application to
  {P}icard sheaves.
\newblock {\em Nagoya Math. J.}, 81:153--175, 1981.

\bibitem[MV99]{morel-voevodsky:a1}
Fabien Morel and Vladimir Voevodsky.
\newblock {${\bf A}^1$}-homotopy theory of schemes.
\newblock {\em Inst. Hautes \'Etudes Sci. Publ. Math.}, (90):45--143 (2001),
  1999.

\bibitem[MV15]{munson-volic}
Brian~A. Munson and Ismar Voli{\'c}.
\newblock {\em Cubical homotopy theory}, volume~25 of {\em New Mathematical
  Monographs}.
\newblock Cambridge University Press, Cambridge, 2015.

\bibitem[MY01]{miyachi-yekutieli}
Jun-ichi Miyachi and Amnon Yekutieli.
\newblock Derived {P}icard groups of finite-dimensional hereditary algebras.
\newblock {\em Compositio Math.}, 129(3):341--368, 2001.

\bibitem[Nee01]{neeman:triangulated}
Amnon Neeman.
\newblock {\em Triangulated categories}, volume 148 of {\em Annals of
  Mathematics Studies}.
\newblock Princeton University Press, Princeton, NJ, 2001.

\bibitem[Pog17]{poguntke:higher-segal}
Thomas Poguntke.
\newblock Higher {S}egal structures in algebraic {$K$}-theory.
\newblock \url{https://arxiv.org/abs/1709.06510}, 2017.
\newblock Preprint.

\bibitem[PS14]{ps:linearity-fp}
Kate Ponto and Mike Shulman.
\newblock The linearity of fixed point invariants.
\newblock arXiv:1406.7861, 2014.

\bibitem[PS16]{ps:linearity}
Kate Ponto and Michael Shulman.
\newblock The linearity of traces in monoidal categories and bicategories.
\newblock {\em Theory Appl. Categ.}, 31:Paper No. 23, 594--689, 2016.

\bibitem[Pup58]{puppe:induzierten-I}
Dieter Puppe.
\newblock Homotopiemengen und ihre induzierten {A}bbildungen. {I}.
\newblock {\em Math. Z.}, 69:299--344, 1958.

\bibitem[Pup67]{puppe:stabil}
D.~Puppe.
\newblock Stabile {H}omotopietheorie. {I}.
\newblock {\em Math. Ann.}, 169:243--274, 1967.

\bibitem[Qui67]{quillen:ha}
Daniel~Gray Quillen.
\newblock {\em Homotopical algebra}.
\newblock Lecture Notes in Mathematics, No. 43. Springer-Verlag, Berlin, 1967.

\bibitem[Qui73]{quillen:k-theory}
Daniel~Gray Quillen.
\newblock Higher algebraic {$K$}-theory. {I}.
\newblock In {\em Algebraic {$K$}-theory, {I}: {H}igher {$K$}-theories ({P}roc.
  {C}onf., {B}attelle {M}emorial {I}nst., {S}eattle, {W}ash., 1972)}, pages
  85--147. Lecture Notes in Math., Vol. 341. Springer, Berlin, 1973.

\bibitem[RB06]{RB:ABC}
Andrei Radulescu-Banu.
\newblock Cofibrations in {H}omotopy {T}heory.
\newblock \url{http://arxiv.org/abs/math/0610009}, 2006.
\newblock Preprint.

\bibitem[Ren09]{renaudin}
Olivier Renaudin.
\newblock Plongement de certaines th\'eories homotopiques de {Q}uillen dans les
  d\'erivateurs.
\newblock {\em J. Pure Appl. Algebra}, 213(10):1916--1935, 2009.

\bibitem[Rez01]{rezk:model}
Charles Rezk.
\newblock A model for the homotopy theory of homotopy theory.
\newblock {\em Trans. Amer. Math. Soc.}, 353(3):973--1007 (electronic), 2001.

\bibitem[Ric91]{rickard:derived-fun}
Jeremy Rickard.
\newblock Derived equivalences as derived functors.
\newblock {\em J. London Math. Soc. (2)}, 43(1):37--48, 1991.

\bibitem[Ric97]{rickard:idemp}
Jeremy Rickard.
\newblock Idempotent modules in the stable category.
\newblock {\em J. London Math. Soc. (2)}, 56(1):149--170, 1997.

\bibitem[Rin84]{ringel:tame-alg}
Claus~Michael Ringel.
\newblock {\em Tame algebras and integral quadratic forms}, volume 1099 of {\em
  Lecture Notes in Mathematics}.
\newblock Springer-Verlag, Berlin, 1984.

\bibitem[RV]{riehl-verity:2-cat-qcats}
Emily Riehl and Dominic Verity.
\newblock The 2-category theory of quasi-categories.
\newblock {\em Adv. Math.}, 280:549--642.

\bibitem[RV17a]{riehl-verity:fibration}
Emily Riehl and Dominic Verity.
\newblock Fibrations and {Y}oneda's lemma in an {$\infty$}-cosmos.
\newblock {\em J. Pure Appl. Algebra}, 221(3):499--564, 2017.

\bibitem[RV17b]{riehl-verity:Kan-ext}
Emily Riehl and Dominic Verity.
\newblock Kan extensions and the calculus of modules for {$\infty$}-categories.
\newblock {\em Algebr. Geom. Topol.}, 17(1):189--271, 2017.

\bibitem[RVdB02]{reiten-bergh:serre}
I.~Reiten and M.~Van~den Bergh.
\newblock Noetherian hereditary abelian categories satisfying {S}erre duality.
\newblock {\em J. Amer. Math. Soc.}, 15(2):295--366, 2002.

\bibitem[Sch10]{schwede:alg-versus-top}
Stefan Schwede.
\newblock Algebraic versus topological triangulated categories.
\newblock In {\em Triangulated categories}, volume 375 of {\em London Math.
  Soc. Lecture Note Ser.}, pages 389--407. Cambridge Univ. Press, Cambridge,
  2010.

\bibitem[Sch13]{schwede:p-order}
Stefan Schwede.
\newblock The {$p$}-order of topological triangulated categories.
\newblock {\em J. Topol.}, 6(4):868--914, 2013.

\bibitem[Seg74]{segal:categories}
Graeme Segal.
\newblock Categories and cohomology theories.
\newblock {\em Topology}, 13:293--312, 1974.

\bibitem[Ser55]{serre:duality}
Jean-Pierre Serre.
\newblock Un th\'eor\`eme de dualit\'e.
\newblock {\em Comment. Math. Helv.}, 29:9--26, 1955.

\bibitem[Shi07]{shipley:spectra-dga}
Brooke Shipley.
\newblock {$H\Bbb Z$}-algebra spectra are differential graded algebras.
\newblock {\em Amer. J. Math.}, 129(2):351--379, 2007.

\bibitem[Shu08]{shulman:framed}
Michael Shulman.
\newblock Framed bicategories and monoidal fibrations.
\newblock {\em Theory Appl. Categ.}, 20(18):650--738 (electronic), 2008.
\newblock arXiv:0706.1286.

\bibitem[Sim12]{simpson:higher}
Carlos Simpson.
\newblock {\em Homotopy theory of higher categories}, volume~19 of {\em New
  Mathematical Monographs}.
\newblock Cambridge University Press, Cambridge, 2012.

\bibitem[SS00]{schwede-shipley:algebras}
Stefan Schwede and Brooke~E. Shipley.
\newblock Algebras and modules in monoidal model categories.
\newblock {\em Proc. London Math. Soc. (3)}, 80(2):491--511, 2000.

\bibitem[SS03]{schwede-shipley:morita}
Stefan Schwede and Brooke Shipley.
\newblock Stable model categories are categories of modules.
\newblock {\em Topology}, 42(1):103--153, 2003.

\bibitem[SS07]{simson-skowronski:2}
Daniel Simson and Andrzej Skowro\'{n}ski.
\newblock {\em Elements of the representation theory of associative algebras.
  {V}ol. 2}, volume~71 of {\em London Mathematical Society Student Texts}.
\newblock Cambridge University Press, Cambridge, 2007.
\newblock Tubes and concealed algebras of Euclidean type.

\bibitem[Ste13]{stevenson:support-action}
Greg Stevenson.
\newblock Support theory via actions of tensor triangulated categories.
\newblock {\em J. Reine Angew. Math.}, 681:219--254, 2013.

\bibitem[{\v{S}}{\v{t}}o14]{stovicek:exact-model}
Jan {\v{S}}{\v{t}}ov{\'{\i}}{\v{c}}ek.
\newblock Exact model categories, approximation theory, and cohomology of
  quasi-coherent sheaves.
\newblock In David~J. Benson, Henning Krause, and Andrzej Skowro{\'n}ski,
  editors, {\em Advances in Representation Theory of Algebras ({C}onf. {ICRA}
  {B}ielefeld, {G}ermany, 8-17 {A}ugust, 2012)}, EMS Series of Congress
  Reports, pages 297--367. EMS Publishing House, Z{\"u}rich, 2014.

\bibitem[Str92]{strickland:interpolation}
N.~P. Strickland.
\newblock On the {$p$}-adic interpolation of stable homotopy groups.
\newblock In {\em Adams {M}emorial {S}ymposium on {A}lgebraic {T}opology, 2
  ({M}anchester, 1990)}, volume 176 of {\em London Math. Soc. Lecture Note
  Ser.}, pages 45--54. Cambridge Univ. Press, Cambridge, 1992.

\bibitem[Tab08]{tabuada:universal-invariants}
Gon{\c{c}}alo Tabuada.
\newblock Higher {$K$}-theory via universal invariants.
\newblock {\em Duke Math. J.}, 145(1):121--206, 2008.

\bibitem[Tho80]{thomason:model-Cat}
R.~W. Thomason.
\newblock Cat as a closed model category.
\newblock {\em Cahiers Topologie G\'eom. Diff\'erentielle}, 21(3):305--324,
  1980.

\bibitem[Ver67]{verdier:thesis}
Jean-Louis Verdier.
\newblock {\em Des cat\'egories d\'eriv\'ees des cat\'egories ab\'eliennes}.
\newblock PhD thesis, Universit\'e de Paris, 1967.

\bibitem[Ver96]{verdier:derived}
Jean-Louis Verdier.
\newblock Des cat\'egories d\'eriv\'ees des cat\'egories ab\'eliennes.
\newblock {\em Ast\'erisque}, (239):xii+253 pp. (1997), 1996.
\newblock With a preface by Luc Illusie, Edited and with a note by Georges
  Maltsiniotis.

\bibitem[Voe98]{voevodsky:a1}
Vladimir Voevodsky.
\newblock {$\mathbf A^1$}-homotopy theory.
\newblock In {\em Proceedings of the {I}nternational {C}ongress of
  {M}athematicians, {V}ol. {I} ({B}erlin, 1998)}, number Extra Vol. I, pages
  579--604 (electronic), 1998.

\bibitem[Vog70]{vogt:boardman}
R.~Vogt.
\newblock {\em Boardman's stable homotopy category}.
\newblock Lecture Notes Series, No. 21. Matematisk Institut, Aarhus
  Universitet, Aarhus, 1970.

\bibitem[vR12]{roosmalen:CY}
Adam-Christiaan van Roosmalen.
\newblock Abelian hereditary fractionally {C}alabi--{Y}au categories.
\newblock {\em Int. Math. Res. Not. IMRN}, (12):2708--2750, 2012.

\bibitem[Wal85]{waldhausen:k-theory}
Friedhelm Waldhausen.
\newblock Algebraic {$K$}-theory of spaces.
\newblock In {\em Algebraic and geometric topology ({N}ew {B}runswick,
  {N}.{J}., 1983)}, volume 1126 of {\em Lecture Notes in Math.}, pages
  318--419. Springer, Berlin, 1985.

\bibitem[Wei94]{weibel:homological}
Charles~A. Weibel.
\newblock {\em An introduction to homological algebra}, volume~38 of {\em
  Cambridge Studies in Advanced Mathematics}.
\newblock Cambridge University Press, Cambridge, 1994.

\end{thebibliography}

\end{document}